\documentclass[11pt,a4paper]{amsart}

\usepackage{lipsum}
\makeatletter
\g@addto@macro{\endabstract}{\@setabstract}
\newcommand{\authorfootnotes}{\renewcommand\thefootnote{\@fnsymbol\c@footnote}}%
\makeatother

\usepackage{mathptmx}
\usepackage{tikz}
\usetikzlibrary{arrows.meta}

\usepackage{mathdots}
\usepackage{amssymb}
\usepackage{mathtools}
\usepackage{mathrsfs}
\usepackage{mathptmx}
\usepackage[utf8]{inputenc}
\usepackage{bm}
\usepackage{bbm}
\usepackage{amsmath}
\usepackage{amsfonts}
\makeatletter
\def\amsbb{\use@mathgroup \M@U \symAMSb}
\makeatother

\usepackage{comment}
\usepackage{graphicx}
\usepackage{grffile}
\usepackage{caption}
\usepackage{subcaption}

\usepackage{bbold}

\usepackage{mathptmx}

\newcommand{\bga}{\begin{aligned}}
	\newcommand{\ena}{\end{aligned}}

\newcommand{\bge}{\begin{enumerate}}
	\newcommand{\ene}{\end{enumerate}}

\newcommand{\hide}[1]{}

\usepackage{dutchcal}

\usepackage{fancyhdr}
\usepackage{subcaption}

\usepackage{newtxtext,newtxmath}
\usepackage{pgfplots}
\pgfplotsset{compat=1.15}

\usepackage{xcolor}

\newcommand{\blue}[1]{{\color{blue} #1}}

\definecolor{webgreen}{rgb}{0,.5,0}
\definecolor{webbrown}{rgb}{.6,0,0}
\definecolor{RoyalBlue}{cmyk}{1, 0.50, 0, 0}
\usepackage[colorlinks=true, breaklinks=true, urlcolor=webbrown, linkcolor=RoyalBlue, citecolor=webgreen,backref=page]{hyperref}

\usepackage{tikz}
\usepackage{pict2e}
\usepackage{todonotes}
\usetikzlibrary{decorations.markings}

\DeclareSymbolFont{bbold}{U}{bbold}{m}{n}
\DeclareSymbolFontAlphabet{\mathbbold}{bbold}

\newcommand{\R}{{\mathbb R}}
\newcommand{\C}{{\mathbb C}}

\newcommand{\Z}{{\mathbb Z}}

\newcommand{\N}{{\mathbb N}}

\newcommand{\T}{{\mathbb T}}

\newcommand{\al}{\alpha}
\newcommand{\be}{\beta}
\newcommand{\f}{\varphi}
\newcommand{\ga}{\gamma}
\newcommand{\Ga}{\Gamma}

\newcommand{\de}{\delta}
\newcommand{\ka}{\kappa}
\newcommand{\De}{\Delta}

\newcommand{\Om}{\Omega}
\newcommand{\ze}{\zeta}

\newcommand{\sg}{\sigma}

\newcommand{\di}{\displaystyle}

\newcommand{\ii}{\textrm{i}}
\newcommand{\dd}{\textrm{d}}

\newcommand{\ds}{\displaystyle}

\newcommand{\qasq}{\quad \text{as} \quad}
\newcommand{\qandq}{\quad \text{and} \quad}

\newtheorem{theorem}{Theorem}

\newtheorem{definition}[theorem]{Definition}

\newtheorem{remark}[theorem]{Remark}
\newtheorem{lemma}[theorem]{Lemma}
\newtheorem{proposition}[theorem]{Proposition}
\newtheorem{corollary}[theorem]{Corollary}
\numberwithin{equation}{section}
\numberwithin{theorem}{section}
\numberwithin{notation}{section}

\usepackage[T1]{fontenc}

\usepackage{mathtools}

\makeatletter
\DeclareRobustCommand\widecheck[1]{{\mathpalette\@widecheck{#1}}}
\def\@widecheck#1#2{%
	\setbox\z@\hbox{\m@th$#1#2$}%
	\setbox\tw@\hbox{\m@th$#1%
		\widehat{%
			\vrule\@width\z@\@height\ht\z@
			\vrule\@height\z@\@width\wd\z@}$}%
	\dp\tw@-\ht\z@
	\@tempdima\ht\z@ \advance\@tempdima2\ht\tw@ \divide\@tempdima\thr@@
	\setbox\tw@\hbox{%
		\raise\@tempdima\hbox{\scalebox{1}[-1]{\lower\@tempdima\box
				\tw@}}}%
	{\ooalign{\box\tw@ \cr \box\z@}}}
\makeatother

\usepackage[mathscr]{euscript}

\newcommand{\ic}{\textrm{i}}
\newcommand{\abs}{\textrm{abs}}
\newcommand{\eq}{\begin{equation}}
\newcommand{\eeq}{\end{equation}}
\def\de{\delta}

\begin{document}

	\tikzset{->-/.style={decoration={
				markings,
				mark=at position #1 with {\arrow{latex}}},postaction={decorate}}}
	
	\tikzset{-<-/.style={decoration={
				markings,
				mark=at position #1 with {\arrowreversed{latex}}},postaction={decorate}}}

\title{Bordered and Framed Toeplitz and Hankel Determinants with Applications in Integrable Probability}

\maketitle

	\begin{center}
		\authorfootnotes
		Roozbeh Gharakhloo\footnote{Mathematics Department, University of California, Santa Cruz, CA, USA. E-mail: roozbeh@ucsc.edu}, 		Karl Liechty\footnote{Department of Mathematical Sciences, DePaul University, Chicago, IL, USA. E-mail: kliechty@depaul.edu}
		\par \bigskip
	\end{center}

\begin{abstract}
    Bordered and framed Toeplitz/Hankel determinants have the same structure as Toeplitz/Hankel determinants except in a small number of matrix rows and/or columns. We review these structured determinants and their connections to orthogonal polynomials, collecting well-known and perhaps less well-known results. We present some applications for these structured determinants to ensembles of non-intersecting paths and the six-vertex model, with an eye towards asymptotic analysis.  We also prove some asymptotic formulas for the probability of non-intersection for an ensemble of continuous time random walks for certain choices of starting and ending points as the number of random walkers tends to infinity.
\end{abstract}

\tableofcontents

\section{Introduction}
\subsection{Definitions and motivation}

	For $\phi \in  L^1(\T)$ denote the $n \times n$ (pure) Toeplitz determinant by  
	\begin{equation}\label{ToeplitzDet}
		D_n[\phi]  = \underset{0 \leq j,k\leq n-1}{\det} \{ \phi_{j-k} \},
	\end{equation}
	 where \begin{equation}
		\phi_j = \int_{\T} \phi(\ze)\ze^{-j} \frac{\dd \ze}{2 \pi \ii \ze}, \qquad j \in \Z,
	\end{equation}
	is the $j$-th Fourier coefficient of $\phi$, and $\T$ denotes the unit circle oriented in the counterclockwise direction. More generally we can consider Toeplitz matrices with respect to a 	measure on $\T$. Indeed for a finite measure $\varphi$ on $\T$, define $\varphi_j$ as
	\begin{equation}\label{eq:FC_measure}
		\varphi_j = \int_{\T} \ze^{-j} \dd\varphi(\ze), \qquad j \in \Z,
	\end{equation}
	and the corresponding $n\times n$ Toeplitz determinant as
		\begin{equation}\label{ToeplitzDet_measure}
		D_n[\varphi]  = \underset{0 \leq j,k\leq n-1}{\det} \{ \varphi_{j-k} \}.
	\end{equation}
For a finite measure $\mu$ on  $\R$ such that all moments exist, i.e.,
\[
\mu_k:=\int_\R x^k \dd\mu(x) <\infty, \quad k=0,1,2,3,\dots,
\]
the Hankel matrix of size $n\times n$ corresponding to the measure $\mu$ is
\[
{\bf H}_n[\mu]=
\begin{pmatrix} \mu_0 & \mu_1 & \cdots & \mu_{n-1} \\
                 \mu_1 & \mu_2 & \cdots & \mu_n \\
                \vdots & \vdots & \ddots & \vdots \\
                \mu_{n-1} & \mu_n & \cdots & \mu_{2n-2}
\end{pmatrix}.
\]
We will denote the determinant of this matrix as
\[
{H}_n[\mu]:= \det {\bf H}_n[\mu].
\]

Toeplitz and Hankel determinants appear in many different contexts in mathematical physics and random matrix theory. Here we list a few such applications. Toeplitz determinants are central objects in characterization of
\begin{itemize}
	\item[-] Diagonal and horizontal two point correlations in the 2D Ising model \cite{McCoy-Wu},
	
	\item[-] Spectral analysis of large Toeplitz matrices \cite{DIK2}, 
	
	\item[-] Gap probabilities in the circular unitary ensemble in random matrix theory \cite{CC, Mehta},
	
	\item[-] The distribution of the length of the longest increasing subsequence of random permutations \cite{BDJ};
\end{itemize}
while Hankel determinants play fundamental roles in characterization of
\begin{itemize}
	\item[-] Gap probabilities in various random matrix ensembles, like the Gaussian unitary ensemble, Laguerre unitary ensemble, and Jacobi unitary ensemble \cite{C, CG, Mehta},
	
	\item[-] Correlations of the characteristic polynomials in the
	Gaussian Unitary Ensemble which is related to the mean values of the zeta-function on the critical line \cite{Krasovsky},  
	
	\item[-] Topological expansions in random matrix models and graph enumeration problems \cite{BD, BGM},
	
	\item[-] The partition function of the six-vertex model \cite{Bleher-Liechty14},
\end{itemize}
and much more. Well-known examples of structural deformations of Toeplitz and Hankel determinants are Toeplitz+Hankel \cite{BE4, BE1, BE3, BE2, BE, GI}, slant-Toeplitz determinants \cite{Hothesis, HoIndianaMathJournal, HoMichiganMathJournal, HoAdjoints1, HoAdjoints2, GW}, and Muttalib-Borodin determinants \cite{BGK, CharlierMuttalib-Borodin, CR, FI}.  

Often in applications involving Toeplitz and Hankel determinants, the problem may be perturbed in such a way that the Toeplitz/Hankel matrix is altered in only a few rows or columns \cite{YP, BEGIL, G23}. As a motivating example, consider the enumeration of non-intersecting random walks. That is, we consider $n$ independent simple random walks on $\Z$, $X_1(t), X_2(t), \dots, X_n(t)$. At each discrete time step, the simple random walk either increases or decreases by 1, each with equal probability. We additionally restrict the random walks in the following way.
\begin{itemize}
\item At time $t=0$ the walkers begin at the points $X_j(0) = 2x_j$ with $x_1<x_2<\cdots<x_n$ and $x_j\in \Z$ for all $j$.
\item At the terminal time $2T\in 2\N$ the walkers end at the points  $X_j(2T) = 2y_j$ with $y_1<y_2<\dots<y_n$ and $y_j\in \Z$ for all $j$.
\item The paths do not intersect, i.e., $X_j(t) \ne X_k(t)$ for $j\ne k$ and $t=0,1,\dots, 2T$.
\end{itemize}
The Lindstr\"om--Gessel--Viennot (LGV) formula  \cite{lindstrom73, Gessel-Viennot85} states that the number of such collections of paths is given as a determinant
\[
\det\left[\#(2x_j,2y_k;2T)\right]_{j,k=1}^n,
\]
where $\#(2x_j, 2y_k;T)$ is the number of paths beginning at $2x_j$ and ending at $2y_k$ after $2T$ time-steps. Since this is a binomial coefficient, the number of collections of non-intersecting paths beginning at $2x_1<\cdots<2x_n$ and ending at $2y_1<\cdots<2y_n$ is
\eq\label{eq:NIRM_enumerate}
\det\left[ \begin{pmatrix} 2T \\ T+y_k - x_j \end{pmatrix}\right]_{j,k=1}^n.
\eeq
The matrix can be written as
\begin{equation}\label{NIRM_phi}
\begin{pmatrix} 2T \\ T+y_k - x_j \end{pmatrix}_{j,k=1}^n=\begin{pmatrix}\int_{\T} \phi(\ze)\ze^{-(y_k - x_j)} \frac{\dd \ze}{2 \pi \ic \ze}\end{pmatrix}_{j,k=1}^n, \qquad \phi(\ze) = \ze^{-T}\left(\ze+1\right)^{2T} = \left(\ze+1/\ze+2\right)^T,
\end{equation}
i.e., $\phi(\ze)$ is the generating function for the number of paths of a random walker with fixed displacement after $2T$ time steps. If the  $x_j=y_j$ are equally spaced then \eqref{eq:NIRM_enumerate} is a Toeplitz determinant.   If most  of the starting/ending points are equally spaced but a few are not, then the Toeplitz matrix may be altered in a few rows/columns. For example if $x_j = j$ for $j=1, \dots, n$,  $y_k = k$ for $k=1,2,\dots, n-m$, and $y_k>n-m$ are arbitrary for $k=n-m+1, \dots, n$, we get
the a determinant of the form
	\begin{equation}\label{bordered_main_example}
\det \begin{pmatrix}
			\phi_0& \phi_{1} & \cdots & \phi_{n-m-1} & \phi_{y_{n-m+1}-1}& \phi_{y_{n-m+2}-1} & \cdots & \phi_{y_{n}-1}  \\
			\phi_{-1}& \phi_0 & \cdots  & \phi_{n-m-2} & \phi_{y_{n-m+1}-2} & \phi_{y_{n-m+2}-2} & \cdots &  \phi_{y_{n}-2} \\
			\vdots & \vdots & \vdots & \vdots & \vdots & \cdots & \cdots & \vdots \\
			\phi_{-n+1} & \phi_{-n+2} &   \cdots  & \phi_{-m} & \phi_{y_{n-m+1}-n} & \phi_{y_{n-m+2}-n}  & \cdots &\phi_{y_{n}-n} 
		\end{pmatrix}, 
	\end{equation}
which has a Toeplitz structure in the first $n-m$ columns, and the remaining $m$ columns consist of descending Fourier coefficients. More generally we can consider the top $m$ starting points to be arbitrary as well. That is, if we take $x_j=y_j= j$ for $j=1,2,\dots,n-m$ and $x_j>n-m$ and $y_j>n-m$ to be arbitrary for $j=n-m+1,\dots,n$, then \eqref{eq:NIRM_enumerate} becomes
\begin{equation}\label{eq:framed_main_example}
\det \begin{pmatrix}
			\phi_0& \phi_{1} & \cdots & \phi_{n-m-1} & \phi_{y_{n-m+1}-1}& \phi_{y_{n-m+2}-1} & \cdots & \phi_{y_{n}-1}  \\
			\phi_{-1}& \phi_0 & \cdots  & \phi_{n-m-2} & \phi_{y_{n-m+1}-2} & \phi_{y_{n-m+2}-2} & \cdots &  \phi_{y_{n}-2} \\
			\vdots & \vdots & \vdots & \vdots & \vdots & \cdots & \cdots & \vdots \\
			\phi_{1-(n-m)} & \phi_{2-(n-m)} &   \cdots  & \phi_{0} & \phi_{y_{n-m+1}-(n-m)} & \phi_{y_{n-m+2}-(n-m)}  & \cdots &\phi_{y_{n}-(n-m)} \\
			\phi_{1-x_{n-m+1}} & \phi_{2-x_{n-m+1}} &   \cdots  & \phi_{n-m-x_{n-m+1}} & \phi_{y_{n-m+1}-x_{n-m+1}} & \phi_{y_{n-m+2}-x_{n-m+1}}  & \cdots & \phi_{y_{n}-x_{n-m+1}}  \\
					\vdots & \vdots & \vdots & \vdots & \vdots & \cdots & \cdots & \vdots \\
			\phi_{1-x_{n}} & \phi_{2-x_{n}} &   \cdots  & \phi_{n-m-x_{n}} & \phi_{y_{n-m+1}-x_{n}} & \phi_{y_{n-m+2}-x_{n}}  & \cdots & \phi_{y_{n}-x_{n}}  \\	
		\end{pmatrix}, 
	\end{equation}
	which has a Toeplitz structure in the top-left $(n-m)\times(n-m)$ block, descending Fourier coefficients in the first $n-m$ entries of each of the last $m$ columns, and ascending Fourier coefficients in each of the first $n-m$ entries of the last $m$ rows. The bottom-right block of size $m\times m$ has no particular structure.
	
	\begin{remark}
		\normalfont	While it is not the topic of the current paper, we mention that the above application can also give rise to the so-called \textit{slant}-Toeplitz structures. Indeed, let $p$ and $q$ be relatively prime positive integers and consider $x_j=pj$ and $y_k=qk$. Then, in view of \eqref{NIRM_phi}, we have 
		\begin{equation}
			\begin{pmatrix} 2T \\ T +y_k - x_j \end{pmatrix}_{j,k=1}^n=\begin{pmatrix}\int_{\T} \phi(\ze)\ze^{-(qk - pj)} \frac{\dd \ze}{2 \pi \ic \ze}\end{pmatrix}_{j,k=1}^n,
		\end{equation}
		which is a slant-Toeplitz matrix. For example, the $2j-k$ and $j-2k$ matrices (with offset $0$) look like:
		\[ \begin{pmatrix}
			\blue{\phi_{0}} & \phi_{-1} & \phi_{-2} & \cdots & \phi_{-n+1} \\
			\phi_{2} & \phi_{1} & \blue{\phi_{0}} & \cdots & \phi_{-n+3} \\
			\phi_{4} & \phi_{3} & \phi_{2} & \cdots & \phi_{-n+5} \\
			\vdots & \vdots & \ddots & \vdots & \vdots \\
			\phi_{2n-2} & \phi_{2n-3} & \phi_{2n-4} & \cdots & \phi_{n-1}
		\end{pmatrix} \qandq \begin{pmatrix}
			\blue{\phi_{0}} & \phi_{-2} & \phi_{-4} & \cdots & \phi_{-2n+2}  \\
			\phi_{1} & \phi_{-1} & \phi_{-3} & \cdots & \phi_{-2n+3}  \\
			\phi_{2} & \blue{\phi_{0}} & \phi_{-2} & \cdots & \phi_{-2n+4} \\
			\vdots & \vdots & \ddots & \vdots & \vdots  \\
			\phi_{n-1} & \phi_{n-3} & \phi_{n-5} & \cdots & \phi_{-n+1} 
		\end{pmatrix}.  \]	The slant-Toeplitz operators have been studied in a series of works by Mark C. Ho \cite{Hothesis, HoIndianaMathJournal, HoMichiganMathJournal, HoAdjoints1, HoAdjoints2} and the underlying system of orthogonal polynomials in the cases $(p,q)=(1,2)$ and $(p,q)=(2,1)$ have been studied in \cite{GW}. 
	\end{remark}
	
We generalize \eqref{bordered_main_example} and \eqref{eq:framed_main_example} as follows.
\begin{definition}\label{def:bordered_Toeplitz} \normalfont
Fix $\phi \in L^1(\T)$ as well as $\vec{\boldsymbol\psi}_m = (\psi^{(1)}, \dots, \psi^{(m)})$, a vector of $m$ functions in $L^1(\T)$. For $n>m$ we define the {\em (multi-)bordered Toeplitz determinant} of size $n\times n$ to be
\begin{equation}\label{btd}
D^B_n[\phi;\vec{\boldsymbol\psi}_m] := \det \begin{pmatrix}
	\phi_0& \phi_{1} & \cdots & \phi_{n-m-1} & \psi^{(1)}_{n-1} & \cdots & \psi^{(m)}_{n-1} \\
	\phi_{-1}& \phi_0 & \cdots  & \phi_{n-m-2} & \psi^{(1)}_{n-2} & \cdots & \psi^{(m)}_{n-2} \\
	\vdots & \vdots & \vdots & \vdots & \vdots & \cdots & \vdots \\
	\phi_{-n+1} & \phi_{-n+2} &   \cdots  & \phi_{-m} & \psi^{(1)}_{0} & \cdots & \psi^{(m)}_{0}
\end{pmatrix}.
\end{equation}
More generally, we will use the same definition and notation when the functions $\phi$ and $\psi^{(j)}$ are replaced with measures and the Fourier coefficients are replaced by \eqref{eq:FC_measure}.
\end{definition}
\begin{definition}\label{def:framed_Toeplitz} \normalfont
Fix $\phi \in L^1(\T)$ as well as $\vec{\boldsymbol\psi}_m = (\psi^{(1)}, \dots, \psi^{(m)})$ and $\vec{\boldsymbol\eta}_m = (\eta^{(1)}, \dots, \eta^{(m)})$, two vectors of $m$ functions in $L^1(\T)$, as well as an $m\times m$ matrix $A=(a_{jk})_{j,k=1}^m$. For $n>m$ we define the {\em (multi-)framed Toeplitz determinant} of size $n\times n$ to be
\begin{equation}\label{multi-framed}
D^F_n[\phi;\vec{\boldsymbol\psi}_m;\vec{\boldsymbol\eta}_m;A] := \det \begin{pmatrix}
	\phi_0& \phi_{1} & \cdots & \phi_{n-m-1} & \psi^{(1)}_{n-m-1} & \cdots & \psi^{(m)}_{n-m-1} \\
	\phi_{-1}& \phi_0 & \cdots  & \phi_{n-m-2} & \psi^{(1)}_{n-m-2} & \cdots & \psi^{(m)}_{n-m-2} \\
	\vdots & \vdots & \vdots & \vdots & \vdots & \cdots & \vdots \\
	\phi_{-(n-m-1)} & \phi_{-(n-m-1)+1} &   \cdots  & \phi_{0} & \psi^{(1)}_{0} & \cdots & \psi^{(m)}_{0} \\
	\eta^{(1)}_{0} & \eta^{(1)}_{1} &   \cdots  & \eta^{(1)}_{n-m-1} & a_{11}& \cdots & a_{1m} \\
	\vdots & \vdots & \vdots & \vdots & \vdots & \cdots & \vdots \\
	\eta^{(m)}_{0} & \eta^{(m)}_{1} &   \cdots  & \eta^{(m)}_{n-m-1} & a_{m1}& \cdots & a_{mm} \\
\end{pmatrix}.
\end{equation}

As with Definition \ref{def:bordered_Toeplitz}, we will use the same definition and notation when the functions $\phi$, $\psi^{(j)}$, and $\eta^{(j)}$ are replaced with measures and the Fourier coefficients are replaced by \eqref{eq:FC_measure}.
\end{definition}
In this notation, \eqref{bordered_main_example} becomes the bordered Toeplitz determinant 
\[
D^B_n[\phi;\boldsymbol{\vec{\psi}}_m], \quad \psi^{(j)}(\ze) = \phi(\ze) \ze^{-(y_{n-m+j}-n)}.
\]
and \eqref{eq:framed_main_example} becomes the multi framed Toeplitz determinant 
\[
D^F_n[\phi;\vec{\boldsymbol\psi}_m;\vec{\boldsymbol\eta}_m;A] , \quad \psi^{(j)}(\ze) = \phi(\ze)\ze^{-(y_{n-m+j}-n)}, \quad \eta^{(j)}(\ze) =\phi(\ze)\ze^{x_{n-m+j}-1} , \quad a_{jk} = \phi_{y_{n-m+k}-x_{n-m+j}}.
\]


\begin{remark} \normalfont  A slightly different multi-framed Toeplitz determinant was defined in  \cite[Equation (1.6)]{G23}. In the current work we will deal exclusively with the multi-framed Toeplitz determinants as defined above, and will refer to them simply as framed Toeplitz determinants in the sequel.
\end{remark}

\begin{remark}\label{Remark different orders of Fourier indices} \normalfont
In the definitions of the bordered and framed Toeplitz determinants we have chosen the Fourier indices of $\psi^{(j)}$ to decrease from top to bottom in the  last $m$ columns, and in the framed Toeplitz determinant we have chosen the Fourier indices of $\eta^{(j)}$ to increase from left to right in the last $m$ rows. It is worthwhile mentioning that the different ways to place the Fourier coefficients in the rows/columns could affect the leading order term in the large size asymptotics of the determinants. Instances of this phenomenon are shown in Theorems 1.10 and 1.11 of \cite{G23}. The other choices of positioning the Fourier coefficients in some of the columns (and/or some of the rows in the case of framed determinants) can be related to the form \eqref{btd} (or to the form \eqref{multi-framed}) by simple manipulations. For example, for fixed $1 \leq i,\ell\leq n$, the multi-framed determinant in which
\begin{itemize}
	\item[]  the indices of $\psi^{(i)}$  increase from top to bottom,
	\item[] the indices of $\eta^{(\ell)}$ decrease from left to right,
	\item[] the indices of $\psi^{(j)}$  decrease from top to bottom, $j=1,\cdots,m$, $j \neq i$, and
	\item[] the indices of $\eta^{(k)}$  increase from left to right, $k=1,\cdots,m$, $k \neq \ell$,
\end{itemize}
is given by $	D^F_n[\phi;\vec{\boldsymbol\psi_*}_m;\vec{\boldsymbol\eta_*}_m;A]$ as defined in \eqref{multi-framed}, where $	\vec{\boldsymbol\psi_*}_m = (\psi^{(1)}, \dots, \psi^{(i-1)}, z^{n-m-1}\tilde{\psi}^{(i)}, \psi^{(i+1)} \ldots, \psi^{(m)}),$ $\vec{\boldsymbol\eta_*}_m = (\eta^{(1)}, \dots, \eta^{(\ell-1)}, z^{n-m-1}\tilde{\eta}^{(\ell)}, \eta^{(\ell+1)} \ldots, \psi^{(m)}),$ and $\tilde{f}$ denotes the function $z \mapsto f(z^{-1})$.

\end{remark}

By analogy we also introduce bordered and framed Hankel determinants.
\begin{definition} \normalfont
Let $\mu$ be a measure on $\R$ with finite moments, and $\vec{\boldsymbol \nu}_m=(\nu^{(1)}, \nu^{(2)}, \dots, \nu^{(m)})$ a vector of $m$ measures on $\R$, each with finite moments.
The {\em bordered Hankel determinant} of size $n\times n$ is
\begin{equation}\label{def:bordered_HMatrix}
H_n^B[\mu; \vec{\boldsymbol{\nu}}_m] :=\det \begin{pmatrix}
    \mu_0 & \mu_1 & \cdots & \mu_{n-m -1} & \nu^{(1)}_{0} & \nu^{(2)}_{0} & \cdots & \nu^{(m)}_{0} \\
    \mu_1 & \mu_2 & \cdots & \mu_{n-m } & \nu^{(1)}_{1} & \nu^{(2)}_{1} & \cdots & \nu^{(m)}_{1} \\
        \mu_2 & \mu_3 & \cdots & \mu_{n-m+1 } & \nu^{(1)}_{2} & \nu^{(2)}_{2} & \cdots & \nu^{(m)}_{2} \\
    \vdots &     \vdots & \ddots & \vdots & \vdots & \ddots & \vdots \\
    \mu_{n-1} & \mu_n & \cdots & \mu_{2n-m-2} & \nu^{(1)}_{n-1} & \nu^{(2)}_{n-1} & \cdots & \nu^{(m)}_{n-1}
    \end{pmatrix},
\end{equation}
where $\mu_k$ is the $k$th moment of the measure $\mu$, and $\nu^{(j)}_k$ is the $k$th moment of the measure $\nu^{(j)}$.
\end{definition}
\begin{definition} \normalfont
Let $\mu$ be a measure on $\R$ with finite moments,  $\vec{\boldsymbol \nu}_m=(\nu^{(1)}, \nu^{(2)}, \dots, \nu^{(m)})$ and $\vec{\boldsymbol \eta}_m=(\eta^{(1)}, \eta^{(2)}, \dots, \eta^{(m)})$ two vectors of $m$ measures on $\R$, each with finite moments, and $A= (a_{jk})_{j,k=1}^\infty$ be a matrix of size $m\times m$.
The {\em framed Hankel determinant} of size $n\times n$ is
 \begin{equation}\label{def:semi_framed_hankelM}
H_n^F[\mu;\vec{\boldsymbol{\nu}}_m,\vec{\boldsymbol{ \eta}}_m ;A]:= \det  \begin{pmatrix} 
 \mu_0 & \mu_1 & \cdots & \mu_{n-m-1} & \nu_0^{(1)} &  \cdots &  \nu_0^{(m)}      \\
    \mu_1 & \mu_2 & \cdots & \mu_{n-m} & \nu_1^{(1)} &  \cdots &  \nu_1^{(m)}      \\
   \vdots & \vdots & \ddots & \vdots & \vdots & \ddots & \vdots      \\
   \mu_{n-m-1} & \mu_{n-m} & \cdots & \mu_{2(n-m-1)} & \nu_{n-m-1}^{(1)} &  \cdots &  \nu_{n-m-1}^{(m)}  \\
\eta^{(1)}_0 &  \eta^{(1)}_1 & \cdots &  \eta^{(1)}_{n-m-1} & a_{11} & \cdots & a_{1m}  \\
   \vdots & \vdots & \ddots & \vdots &  \vdots & \ddots & \vdots  \\
 \eta^{(m)}_0 &  \eta^{(m)}_1 & \cdots &  \eta^{(m)}_{n-m-1}  & a_{m1} & \cdots & a_{mm} \\  \\
\end{pmatrix},
\end{equation}
where $\mu_k$ is the $k$th moment of the measure $\mu$, $\nu^{(j)}_k$ is the $k$th moment of the measure $\nu^{(j)}$, and $\eta^{(j)}_k$ is the $k$th moment of the measure $\eta^{(j)}$. 
\end{definition}
\begin{remark}\normalfont
One could generalize the framed determinants so that the top-, bottom-, left-, and right-most $m$ rows/columns differ from the Toeplitz/Hankel structure. In the Hankel case we could consider determinants of the form
 \begin{multline}\label{def:framed_hankelM}
 \det\begin{pmatrix} a_{11} & \cdots & a_{1m} & \sigma^{(1)}_0 &  \sigma^{(1)}_1 & \cdots &  \sigma^{(1)}_{n-2m-1} & b_{11} & \cdots & b_{1m}   \\
\vdots & \ddots & \vdots &      \vdots & \vdots & \ddots & \vdots &  \vdots & \ddots & \vdots  \\
a_{m1} & \cdots & a_{mm} &   \sigma^{(m)}_0 &  \sigma^{(m)}_1 & \cdots &  \sigma^{(m)}_{n-2m-1}  & b_{m1} & \cdots & b_{mm}  \\
 \lambda_0^{(1)}     &   \cdots & \lambda_{0}^{(m)} & \mu_0 & \mu_1 & \cdots & \mu_{n-2m-1} & \nu_0^{(1)} &  \cdots &  \nu_0^{(m)}      \\
 \lambda_1^{(1)}     &   \cdots & \lambda_{1}^{(m)} &       \mu_1 & \mu_2 & \cdots & \mu_{n-2m} & \nu_1^{(1)} &  \cdots &  \nu_1^{(m)}  \\
\vdots & \ddots & \vdots &      \vdots & \vdots & \ddots & \vdots & \vdots & \ddots & \vdots   \\
 \lambda_{n-2m-1}^{(1)}     &   \cdots & \lambda_{n-2m-1}^{(m)} &      \mu_{n-2m-1} & \mu_{n-2m} & \cdots & \mu_{2(n-2m-1)} & \nu_{n-2m-1}^{(1)} &  \cdots &  \nu_{n-2m-1}^{(m)}\\
d_{11} & \cdots & d_{1m} & \eta^{(1)}_0 &  \eta^{(1)}_1 & \cdots &  \eta^{(1)}_{n-2m-1} & c_{11} & \cdots & c_{1m}   \\
\vdots & \ddots & \vdots &      \vdots & \vdots & \ddots & \vdots &  \vdots & \ddots & \vdots  \\
d_{m1} & \cdots & d_{mm} &   \eta^{(m)}_0 &  \eta^{(m)}_1 & \cdots &  \eta^{(m)}_{n-2m-1}  & c_{m1} & \cdots & c_{mm}  \\
\end{pmatrix}.
\end{multline}
It is clear that the above determinant reduces to the one of the fom \eqref{def:semi_framed_hankelM} by row and column operations, so we restrict ourselves to matrices of the form \eqref{def:semi_framed_hankelM} and \eqref{multi-framed}.
\end{remark}

The first major appearance of bordered Toeplitz determinants \eqref{btd} in the mathematical physics literature was in the work \cite{YP} of Au-Yang and Perk in 1987 where it was established that the  \textit{next-to-diagonal} two-point correlation function in the 2D Ising model is given  by \begin{equation}\label{BT&NTD}
	\langle \sigma_{0,0}\sigma_{n-1,n} \rangle = D^B_n[\widehat{\phi}; \widehat{\psi}],
\end{equation}
with 
\begin{equation}\label{hat psi}
	\widehat{\phi}(z) = \sqrt{\frac{1-k^{-1}z^{-1}}{1-k^{-1}z}} \qandq	\widehat{\psi}(z)= \frac{C_v z\widehat{\phi}(z)+C_h}{S_v(z-c_*)},
\end{equation} where $k, C_v, C_h, S_v,$ and $c_*$ are all physical parameters of the model and $\sigma_{i,j} \in \{-1,1\}$ represents the magnetic dipole moment of the atomic spins at position $(i,j)$. The framed Toeplitz determinants have also already appeared in the calculations of entanglement entropy for disjoint subsystems in the XX spin chain \cite{JK, BGIKMMV} (also see the introduction of \cite{G23}). 

In this paper we collect and review results which relate bordered and framed Toeplitz/Hankel determinants to systems of orthogonal polynomials. The orthogonal polynomial formulas for these structured determinants enables one to study their large-size asymptotics, since for the orthogonal polynomials, one has the associated Riemann-Hilbert problem first discovered by A. S. Fokas, A. R. Its, and A. V. Kitaev in 1992 \cite{FIK}. More precisely, the asymptotic analysis of the Riemann--Hilbert problem via the Deift--Zhou method \cite{DZ} allows for a detailed and rigorous asymptotic analysis, which in turn, allows for a successful asymptotic analysis for the framed and bordered Toeplitz/Hankel determinants. In the subsections below, we first give a brief review of the well known OPRL\footnote{Orthogonal polynomials on the real line.} and BOPUC\footnote{Bi-orthogonal polynomials on the unit circle.} characterizations of the (pure)  Hankel and Toeplitz determinants, respectively, before  presenting orthogonal polynomial characterizations for the bordered and framed Toeplitz/Hankel determinants. 

\subsection{Toeplitz/Hankel determinants and orthogonal polynomials}\label{Sec OP characterizations pure}

Let us start with the OPRL characterization for Hankel determinants. Let $\mathscr{I} \subset \R$ and $\mu$ be a measure with finite moments on $\mathscr{I}$. Consider the associated set of monic orthogonal polynomials\index{orthogonal polynomials!orthogonal polynomials on the line} $\{P_n(z)\}^{\infty}_{n=0}$, $\deg P_n(z)=n$ ,  satisfying the orthogonality conditions
\begin{equation}\label{Hankel OPs conditions}
	\int_{\mathscr{I}} P_n(x)x^k \dd \mu(x) = h_n \de_{nk}, \qquad k=0,1,\cdots,n,
\end{equation}
which is equivalent to 
\begin{equation}\label{Hankel OPs conditions 1}
	\int_{\mathscr{I}} P_n(x)P_k(x)  \dd \mu(x) = h_n \de_{nk}, \qquad k,n \in \Z_{\geq 0}.
\end{equation}
If $\det \boldsymbol{H}_n[\mu] \neq 0$, the polynomial $P_n(z)$ has the following determinantal representation
\begin{equation}\label{Pn det rep}
	P_n(x) = \frac{1}{\det \boldsymbol{H}_n[\mu]} \det \begin{pmatrix}
		\mu_0 & \mu_1 & \cdots & \mu_{n-1} & \mu_{n} \\
		\mu_1 & \mu_{2} & \cdots  & \mu_{n} & \mu_{n+1} \\
		\vdots & \vdots & \reflectbox{$\ddots$} & \vdots & \vdots \\
		\mu_{n-1} & \mu_{n} & \cdots  & \mu_{2n-2} & \mu_{2n-1} \\
		1 & x & \cdots & x^{n-1} & x^n
	\end{pmatrix}.
\end{equation}
Indeed the right-hand-side of \eqref{Pn det rep} is a polynomial in $x$ of degree $n$ and it is easy to see that the integral of $x^k$ times the right-hand-side of \eqref{Pn det rep} vanishes for $k=0,1,2,\dots, n-1$ since two rows of the matrix coincide. Thus it is a constant times $P_{n}(x)$ and the leading coefficient is  seen to be one by expanding the determinant in the last column.
The formula \eqref{Pn det rep} yields
\begin{equation}\label{norms Hankel OPs} h_n= \frac{\det \boldsymbol{H}_{n+1}[\mu]}{\det \boldsymbol{H}_n[\mu]}, \qquad \mbox{and thus} \qquad \det \boldsymbol{H}_n[\mu] = \prod_{j=0}^{n-1} h_j.
\end{equation}

If the measure $\mu$ has a H\" older continuous density $\dd\mu(x) = w(x)\,\dd x$, then it is due to Fokas, Its, and Kitaev \cite{FIK} that the following matrix-valued function which is built from the orthogonal polynomials and the Cauchy transforms of the weight $w$ multiplied by the orthogonal polynomials
\begin{equation}\label{Y OP solution}
	Y(z;n)=\begin{pmatrix}
		P_n(z) & \di \frac{1}{2\pi \ii}	 \int_{\mathscr{I}} \frac{ P_n(x)w(x) }{x-z}\dd x \\
		-\di \frac{2\pi \ii}{h_{n-1}}P_{n-1}(z) & -\di \frac{1}{ h_{n-1}}  \int_{\mathscr{I}} \frac{P_{n-1}(x) w(x) }{x-z}\dd x
	\end{pmatrix}, \qquad n \geq 1,    
\end{equation}
satisfies the following Riemann--Hilbert problem (RHP): 
\begin{itemize}
	\item \textbf{RH-Y1}\label{RH-Y1} \qquad  $Y(\cdot;n): \C \setminus \mathscr{I} \to \C^{2 \times 2}$ is analytic.
	\item \textbf{RH-Y2}\label{RH-Y2} \qquad  The limits of $Y(z;n)$ as $z$ tends to $x \in \mathscr{I} $ from the upper and lower half plane exist, which are denoted by $Y_{\pm}(x;n)$ respectively. Moreover, the functions $x \mapsto Y_{\pm}(x;n)$ are continuous on $\mathscr{I}$ and are related by the \textit{jump condition}
	\begin{equation}\label{Y_Jump1}
		Y_+(x;n)=Y_-(x;n) \begin{pmatrix}
			1 & w(x) \\
			0 & 1
		\end{pmatrix},
		\qquad x \in \mathscr{I}.
	\end{equation} 
	\item \textbf{RH-Y3}\label{RH-Y3} \qquad  As $z \to \infty$,  \begin{equation}\label{Y_Asymptotics_infty3}
		Y(z;n)=\big( I + O(z^{-1}) \big) z^{n \sigma_3},
	\end{equation}
	where $\sigma_3=\begin{pmatrix}
		1 & 0 \\
		0 & -1 \\
	\end{pmatrix}$ is the third Pauli matrix.
\end{itemize}
A successful large-$n$ asymptotic analysis of this RHP yields the asymptotics of the norms $h_n$'s, and thus that of $\det \boldsymbol{H}_n[w]$ via \eqref{norms Hankel OPs} (see e.g. \cite{KMcLVAV,Krasovsky,C,CG} for this analysis for various choices of $w$). For discrete measures $\mu$, there is a similar characterization in terms of a discrete Riemann--Hilbert problem \cite{BKMM, Bleher-Liechty11, Bleher-Liechty14}.

Now, let us briefly discuss the BOPUC characterization of Toeplitz determinants. Given a finite measure $\varphi$ of $\T$, one can consider the associated sets of  bi-orthonormal polynomials\index{orthogonal polynomials} $\{Q_n(z)\}^{\infty}_{n=0}$ and  $\{\widehat{Q}_n(z)\}^{\infty}_{n=0}$, where $Q_n(z)=\ka_nz^n+l_nz^{n-1}+\cdots$,  and $\widehat{Q}_n(z)=\widehat{\ka}_nz^n+\widehat{l}_nz^{n-1}+\cdots$ satisfy the bi-orthonormality\index{orthogonal polynomials!bi-orthogonal polynomials on the unit circle} conditions
\begin{equation}
	\int_{\T} Q_n(z)\widehat{Q}_k(z^{-1})\dd\varphi(z) = \de_{nk}.
\end{equation}
Similar to \eqref{Pn det rep}, if $D_n[\varphi], D_{n+1}[\varphi] \neq 0$, these polynomials have determinantal representations given by
\begin{equation}\label{Toeplitz OP 1}
	Q_n(z)= \frac{1}{\sqrt{D_n[\varphi] D_{n+1}[\varphi]}} \det \begin{pmatrix}
		\varphi_0 & \varphi_{-1} & \cdots & \varphi_{-n} \\
		\varphi_1 & \varphi_{0} & \cdots & \varphi_{-n+1} \\
		\vdots & \vdots & \ddots & \vdots \\
		\varphi_{n-1} & \varphi_{n-2} & \cdots & \varphi_{-1} \\
		1 & z & \cdots & z^n
	\end{pmatrix},
\end{equation}
and 
\begin{equation}\label{Toeplitz OP 2}
	\widehat{Q}_n(z) = \frac{1}{\sqrt{D_n[\varphi] D_{n+1}[\varphi]}} \det \begin{pmatrix}
		\varphi_0 & \varphi_{-1} & \cdots & \varphi_{-n+1} & 1 \\
		\varphi_1 & \varphi_{0} & \cdots & \varphi_{-n+2} & z \\
		\vdots & \vdots & \ddots & \vdots & \vdots \\
		\varphi_{n} & \varphi_{n-1} & \cdots  & \varphi_{1} & z^n
	\end{pmatrix},
\end{equation}
from which it is clear that \begin{equation}\label{lc}
	\ka_n = \widehat{\ka}_n = \sqrt{\frac{\det T_{n}[\varphi]}{\det T_{n+1}[\varphi]}}.
\end{equation}
Now let us consider the case $\dd\varphi(z)= \frac{\phi(z)\,\dd z}{2\pi \ii z}$ for a H\"older continuous function $\phi(z)$, and define the $2\times 2$ matrix-valued function
\begin{equation}\label{Toeplitz-OP-solution}
	X(z; n):=\begin{pmatrix}
		\ka_n^{-1} Q_n(z) & \di \ka^{-1}_n \int_{\T} \frac{Q_n(\ze)}{(\ze-z)} \frac{\phi(\ze)\dd\ze}{2\pi \ii \ze^n} \\
		-\ka_{n-1}z^{n-1}\widehat{Q}_{n-1}(z^{-1}) & \di -\ka_{n-1} \int_{\T}  \frac{\widehat{Q}_{n-1}(\ze^{-1})}{(\ze-z)} \frac{\phi(\ze)\dd\ze}{2\pi \ii \ze}
	\end{pmatrix}.
\end{equation}
It is due to J. Baik, P. Deift and K. Johansson \cite{BDJ} (inspired by the RHP for OPRL \cite{FIK}) that the function $X$ defined above satisfies the following associated Riemann--Hilbert problem:
\begin{itemize}
	\item  \textbf{RH-X1}\label{RH-X1} \qquad $X:\C\setminus \T \to \C^{2\times2}$ is analytic,
	\item \textbf{RH-X2}\label{RH-X2} \qquad  The limits of $X(\xi;n)$ as $\xi$ tends to $z \in \T $ from the inside and outside of the unit circle exist, which are denoted by $X_{\pm}(z;n)$ respectively. Moreover, the functions $z \mapsto X_{\pm}(z;n)$ are continuous on $\T$ and are related by the \textit{jump condition}
	\begin{equation}
		X_+(z;n)=X_-(z;n)\begin{pmatrix}
			1 & z^{-n}\phi(z) \\
			0 & 1
		\end{pmatrix}, \qquad  z \in \T,
	\end{equation}
	\item \textbf{RH-X3}\label{RH-X3} \qquad  As $z \to \infty$
	\begin{equation}\label{X asymp}
		X(z;n)=\left(\di I+\frac{ \overset{\infty}{X}_1(n)}{z}+\frac{\overset{\infty}{X}_2(n)}{z^2} + O(z^{-3})\right) z^{n \sigma_3}.
	\end{equation}
\end{itemize}
For the purpose of calculations in Section \ref{Sec-Asymptotics}, we present the asymptotic analysis of this RHP in Appendix \ref{Appendices} for Szeg{\H o}-type symbols\footnote{We will occasionally refer to a symbol as \textit{Szeg{\H o}-type}, if a) it is smooth and nonzero on the unit circle, b) has no winding number, and c) admits
	an analytic continuation in a neighborhood of the unit circle. For a Szeg{\H o}-type symbol $\phi(z)$, we refer to the piecewise analytic function $\al(z)$ satisfying the jump condition $\al_+(z)=\al_-(z)\phi(z)$ and the large-$z$ asymptotic condition $\al(z)=1+O(1/z)$, as the \textit{Szeg{\H o} function} associated with $\phi$ (see e.g. \eqref{phi mother} and \eqref{al* non-int application}).}. For such asymptotic analysis in the case of Fisher--Hartwig symbols we refer the interested reader to \cite{DIK, blackstone2023toeplitz}. 
	
%
%

\subsection{Bordered Toeplitz/Hankel determinants and orthogonal polynomials}\label{Sec OP characterizations bordered}

We start with the extension of \eqref{Pn det rep} to several variables. Let ${ H}_n^B[\mu; z_1, \dots, z_m]$ be the determinant
\eq
H_n^B[\mu; z_1, \dots, z_m]: = \det \begin{pmatrix}
	\mu_0 & \mu_1 & \cdots & \mu_{n-m-1} & 1 & 1 & \cdots & 1 \\
	\mu_1 & \mu_2 & \cdots & \mu_{n-m} & z_1 & z_2 & \cdots & z_m \\
	\mu_2 & \mu_3 & \cdots & \mu_{n-m+1} & z_1^2 & z_2^2 & \cdots & z_m^2 \\
	\vdots &     \vdots & \ddots & \vdots & \vdots & \vdots & \ddots & \vdots  \\
	\mu_{n-1} & \mu_n & \cdots & \mu_{2n-m-2} & z_1^{n-1} & z_2^{n-1} & \cdots & z_{m}^{n-1}\end{pmatrix}.
	\eeq

This determinant may also be expressed in terms of orthogonal polynomials.
\begin{proposition}\label{eq:bordered_det}
	The determinant ${H}_n[\mu; z_1, \dots, z_m]$ is given as
	\begin{equation}\label{eq:OPs_det}
		{ H}_n^B[\mu; z_1, \dots, z_m] = H_{n-m}^B[\mu] \det 
		\begin{pmatrix}
			P_{n-m}(z_1) & \cdots & P_{n-m}(z_m) \\
			\vdots & \ddots & \vdots \\
			P_{n-1}(z_1) & \cdots & P_{n-1}(z_m)
		\end{pmatrix}.
	\end{equation}
\end{proposition}
The above proposition is well known in the literature (see, e.g., \cite[Equation (5.5)]{Colomo-Pronko08}) though we were unable to find a proof written explicitly anywhere. The formula \eqref{eq:OPs_det} is rather similar to a formula of Christoffel (\cite[Theorem 2.5]{SzegoOP}, and Proposition \ref{eq:bordered_det} can be proved based on that formula. In this paper we give a different proof based on the Dodgson Condensation identity. This proof is presented in Section \ref{sec:proofs of OP props}. Using an identical argument we can prove the analogous result for the system of bi-orthogonal polynomials on the unit circle.

\begin{proposition}\label{Prop: multi-bordered Toeplitz and BOPUC}
	Let $T_n[\varphi;z_1,\cdots,z_m]$ be the following $n \times n$ multi-bordered Toeplitz determinant
	\begin{equation}
		T_n[\varphi;z_1,\cdots,z_m]:= \det \begin{pmatrix}
			\varphi_0 & \varphi_1 & \cdots & \varphi_{n-m-1} & 1 & 1 & \cdots & 1 \\
			\varphi_{-1} & \varphi_0 & \cdots & \varphi_{n-m-2} & z_1 & z_2 & \cdots & z_m \\
			\vdots &     \vdots & \ddots & \vdots & \vdots & \vdots & \ddots & \vdots  \\
			\varphi_{-n+1} & \varphi_{-n+2} & \cdots & \varphi_{-m} & z_1^{n-1} & z_2^{n-1} & \cdots & z_{m}^{n-1}
		\end{pmatrix}.
	\end{equation}
	It holds that
	\begin{equation}\label{eq:OPs_det Qn}
		{T}_n[\varphi; z_1, \dots, z_m] = \sqrt{D_{n-m}[\varphi]D_{n-m+1}[\varphi]} \det 
		\begin{pmatrix}
			Q_{n-m}(z_1) & \cdots & Q_{n-m}(z_m) \\
			\vdots & \ddots & \vdots \\
			Q_{n-1}(z_1) & \cdots & Q_{n-1}(z_m)
		\end{pmatrix},
	\end{equation}
	and
	\begin{equation}\label{eq:OPs_det QnHat}
		{T}_n[\tilde{\varphi}; z_1, \dots, z_m] = \sqrt{D_{n-m}[\varphi]D_{n-m+1}[\varphi]} \det 
		\begin{pmatrix}
			\widehat{Q}_{n-m}(z_1) & \cdots & \widehat{Q}_{n-m}(z_m) \\
			\vdots & \ddots & \vdots \\
			\widehat{Q}_{n-1}(z_1) & \cdots & \widehat{Q}_{n-1}(z_m)
		\end{pmatrix},
	\end{equation}
	where $Q_k(z)$ and $\widehat{Q}_k(z)$ are the biorthogonal polynomials on the unit circle given by \eqref{Toeplitz OP 1} and \eqref{Toeplitz OP 2}, and the measures $\varphi$ and $\tilde \varphi$ are related by $\dd\tilde{\varphi}(z):=\dd\varphi(z^{-1})$.
\end{proposition}

Using Propositions \ref{eq:bordered_det} and \ref{Prop: multi-bordered Toeplitz and BOPUC}, it is simple to give a similar formula for any bordered Hankel or Toeplitz determinant as a (multiple) integral.
\begin{corollary}\label{prop:bordered_det_gen}
	The bordered determinants ${ H}_n^B[\mu; \vec{\boldsymbol{\nu}}_m]$ and $D^B_n[\phi; \vec{\boldsymbol{\psi}}_m]$ are given by
	\begin{equation}\label{eq:bordered_det_gen}
		{ H}_n^B[\mu; \boldsymbol{\nu}_m] = H_{n-m}[\mu]\int_\R \cdots \int_\R \det 
		\begin{pmatrix}
			P_{n-m}(z_1) & \cdots & P_{n-m}(z_m) \\
			\vdots & \ddots & \vdots \\
			P_{n-1}(z_1) & \cdots & P_{n-1}(z_m)
		\end{pmatrix} \dd\nu^{(1)}(z_1) \cdots \dd\nu^{(m)}(z_m),
	\end{equation}
	
	\begin{equation}\label{eq:T_bordered_det_gen1}
		D^B_n[\varphi; \vec{\boldsymbol{\psi}}_m] = \sqrt{D_{n-m}[\varphi]D_{n-m+1}[\varphi]} \int_\T \cdots \int_\T  \det 
		\begin{pmatrix}
			Q_{n-m}(z_1) & \cdots & Q_{n-m}(z_m) \\
			\vdots & \ddots & \vdots \\
			Q_{n-1}(z_1) & \cdots & Q_{n-1}(z_m)
		\end{pmatrix} \prod_{j=1}^m z^{-n+1}_j \dd\psi^{(j)}(z_j).
	\end{equation}
\end{corollary}
\begin{proof}
	By multilinearity of the determinant, 
	\begin{multline}
		{ H}_n^B[\mu; \boldsymbol{\nu}_m] = \det \begin{pmatrix}
			\mu_0 & \mu_1 & \cdots & \mu_{n-m-1} & \int_\R \dd\nu^{(1)}(z) & \int_\R \dd\nu^{(2)}(z) & \cdots & \int_\R \dd\nu^{(m)}(z) \\
			\mu_1 & \mu_2 & \cdots & \mu_{n-m} & \int_\R z \dd\nu^{(1)}(z )& \int_\R z \dd\nu^{(2)}(z) & \cdots & \int_\R z  \dd\nu^{(m)}(z) \\
			\mu_2 & \mu_3 & \cdots & \mu_{n-m+1} & \int_\R z^2 \dd\nu^{(1)}(z) & \int_\R z^2 \dd\nu^{(2)}(z) & \cdots & \int_\R z^2  \dd\nu^{(m)}(z) \\
			\vdots &     \vdots & \ddots & \vdots & \vdots & \ddots & \vdots \\
			\mu_{n-1} & \mu_n & \cdots & \mu_{2n-m-2} & \int_\R z^{n-1} \dd\nu^{(1)}(z) & \int_\R z^{n-1} \dd\nu^{(2)}(z) & \cdots & \int_\R z^{n-1}  \dd\nu^{(m)}(z)
		\end{pmatrix} \\
		=\int_\R\cdots \int_\R \begin{pmatrix}
			\mu_0 & \mu_1 & \cdots & \mu_{n-m-1} & 1 & 1 & \cdots & 1 \\
			\mu_1 & \mu_2 & \cdots & \mu_{n-m} & z_1 & z_2 & \cdots & z_m \\
			\mu_2 & \mu_3 & \cdots & \mu_{n-m+1} & z_1^2 & z_2^2 & \cdots & z_m^2 \\
			\vdots &     \vdots & \ddots & \vdots & \vdots & \vdots & \ddots & \vdots  \\
			\mu_{n-1} & \mu_n & \cdots & \mu_{2n-m-2} & z_1^{n-1} & z_2^{n-1} & \cdots & z_{m}^{n-1}
		\end{pmatrix} \dd\nu^{(1)}(z_1)  \dd\nu^{(2)}(z_2) \cdots  \dd\nu^{(m)}(z_m), 
	\end{multline}
	which is \eqref{eq:bordered_det_gen} using \eqref{eq:OPs_det}. Equation \eqref{eq:T_bordered_det_gen1} follows in a similar way.
\end{proof}

\begin{remark}
\normalfont{The proof of Corollary \ref{prop:bordered_det_gen} is strictly algebraic and makes no assumptions on the border measures $\nu^{(j)}$ and $\psi^{(j)}$. The corollary holds quite broadly, including for signed and complex measures $\nu^{(j)}$ and $\psi^{(j)}$. These will appear in some of the examples in subsequent sections.}
\end{remark}

\subsection{Framed Toeplitz/Hankel determinants and orthogonal polynomials}

It was shown in \cite{G23} (based on an idea introduced in \cite{GW}) that the single-framed Toeplitz determinants are closely connected to the reproducing kernel of BOPUC. In this section we extend those results to multi-framed Toeplitz determinants and  present their analogues for (multi-)framed Hankel determinants as well. The formulas established in this section for an $m$-framed Toeplitz/Hankel determinant involve an $m \times m$ determinant using the reproducing kernel of the OPRL or BOPUC. For $m=1$ such formulas were written in Theorem 1.9 of \cite{G23}, and a similar approach establishes analogous formulas for the single-framed Hankel determinants. However for $m \geq 2$, unlike what we show in this work, no determinantal formulas were shown in \cite{G23}. In fact, for $m \geq 2$, there exists an alternative (and perhaps less efficient) approach to write an $m$-framed Toeplitz/Hankel determinant in terms of the reproducing kernels of orthogonal polynomials. This alternative representation is achieved by expressing the $m$-framed determinant in terms of several single-framed determinants, utilizing successive applications of proper Dodgson Condensation identities (see Section \ref{sec 5.1}). Consequently, as a result of Theorem 1.9 in \cite{G23}, the $m$-framed determinant can be expressed in terms of the reproducing kernels of orthogonal polynomials. This idea is displayed in section 3.3 of \cite{G23} and will not be pursued in this section.

The reproducing kernel for the OPRL \eqref{Hankel OPs conditions} is defined as 
\begin{equation}\label{Rep Ker OPRL}
	K_n(x,y) = \sum_{k=0}^n \frac{P_k(x) P_k(y)}{h_k}.
\end{equation}
For an $m\times m$ matrix $A = (a_{jk})_{j,k=1}^m$, and variables $x_1, \dots, x_m$, and $y_1, \dots, y_m$, define the determinant  \\ $H_n^F(\mu; x_1, \dots, x_m; y_1, \dots, y_m; A) $ as
\[
H_n^F[\mu; x_1, \dots, x_m; y_1, \dots, y_m; A] := \det \begin{pmatrix} \mu_0 & \mu_1 & \cdots & \mu_{n-m-1} & 1 & \cdots & 1 \\
	\mu_1 & \mu_2 & \cdots & \mu_{n-m} & x_1 & \cdots & x_m \\
	\mu_2 & \mu_3 & \cdots & \mu_{n-m+1} & x_1^2 & \cdots & x_m^2 \\
	\vdots & \vdots & \ddots & \vdots & \vdots & \ddots & \vdots\\
	\mu_{n-m-1} & \mu_{n-m} & \cdots & \mu_{2(n-m-1)} & x_1^{n-m-1} & \cdots & x_m^{n-m-1} \\
	1 & y_1 & \cdots & y_1^{n-m-1} & a_{11} & \dots & a_{1m} \\
	\vdots & \vdots & \ddots & \vdots & \vdots & \ddots & \vdots\\
	1 & y_m & \cdots & y_m^{n-m-1} & a_{m1} & \dots & a_{mm}
\end{pmatrix}.
\]
\begin{proposition}\label{eq:semi-framed_det_CD}  Let $A=(a_{jk})_{j,k=1}^m$. We have
	\begin{equation}\label{eq:thick_semifremed_CD}
		H_n^F[\mu; x_1, \dots, x_m; y_1, \dots, y_m; A] =H_{n-m}[\mu] \det\left[A - \left(K_{n-m-1}(x_k, y_j)\right)_{j,k=1}^m\right],
	\end{equation}
	where	$K_{n}(x,y) $
	is the reproducing kernel of the system of orthogonal polynomials on the real line given by \eqref{Rep Ker OPRL}.
	\end{proposition}
		The above result, which will be proved in Section \ref{Sec 5.2}, implies a general  formula for framed Hankel determinants.
	
	\begin{corollary}\label{Cor:semi-framed_det_gen} Let $A=(a_{jk})_{j,k=1}^m$. We have
		\begin{equation}\label{eq:s_framed_det_gen}
			H_n^F(\mu; \nu^{(1)}, \dots, \nu^{(m)}; \eta^{(1)}, \dots, \eta^{(m)}; A) = H_{n-m}[\mu] \det\left[ A - \left(\int_\R\int_\R K_{n-m-1}(x_k, y_j)\dd\nu^{(k)}(x_k)\dd\eta^{(j)}(y_j)\right)_{j,k=1}^m\right],
		\end{equation}
			where	$K_{n}(x,y) $
		is the reproducing kernel of the system of orthogonal polynomials on the real line given by \eqref{Rep Ker OPRL}.
	\end{corollary}
	\begin{proof}
		The proof is nearly identical to that of \eqref{eq:bordered_det_gen}.
	\end{proof}
	\begin{remark}\normalfont
		If the powers of $x$ and/or $y$ are reversed in order, similar results follow as corollaries. For $m=1$ we record the rest of the four different orders of the powers of $x$ and $y$ below:
		\begin{multline}
			\det \begin{pmatrix} \mu_0 & \mu_1 & \cdots & \mu_{n-2} & x^{n-2} \\
				\mu_1 & \mu_2 & \cdots & \mu_{n-1} & x^{n-3} \\
				\mu_2 & \mu_3 & \cdots & \mu_{n} & x^{n-4} \\
				\vdots & \vdots & \ddots & \vdots & \vdots\\
				\mu_{n-2} & \mu_{n-1} & \cdots & \mu_{2n-4} & 1 \\
				1 & y & \cdots & y^{n-2} & a
			\end{pmatrix} = { H}_n^F[\mu; 1/x,y; a/x^{n-2}]x^{n-2} \\
			=H_{n-1}[\mu](a-x^{n-2}K_{n-2}(1/x,y)),
		\end{multline}
		\begin{multline}
			\det \begin{pmatrix} \mu_0 & \mu_1 & \cdots & \mu_{n-2} & 1 \\
				\mu_1 & \mu_2 & \cdots & \mu_{n-1} & x^2 \\
				\mu_2 & \mu_3 & \cdots & \mu_{n} & x^3 \\
				\vdots & \vdots & \ddots & \vdots & \vdots\\
				\mu_{n-2} & \mu_{n-1} & \cdots & \mu_{2n-4} & x^{n-2} \\
				y^{n-2} & y^{n-3} & \cdots &   1 & a
			\end{pmatrix} = {H}_n^F[\mu; x,1/y; a/y^{n-2}]y^{n-2} \\
			=H_{n-1}[\mu](a-y^{n-2}K_{n-2}(x,1/y)),
		\end{multline}
		\begin{multline}
			\det \begin{pmatrix} \mu_0 & \mu_1 & \cdots & \mu_{n-2} & x^{n-2} \\
				\mu_1 & \mu_2 & \cdots & \mu_{n-1} & x^{n-3} \\
				\mu_2 & \mu_3 & \cdots & \mu_{n} & x^{n-4} \\
				\vdots & \vdots & \ddots & \vdots & \vdots\\
				\mu_{n-2} & \mu_{n-1} & \cdots & \mu_{2n-4} & 1 \\
				y^{n-2} & y^{n-3} & \cdots &   1 & a
			\end{pmatrix} = { H}_n^F[\mu; 1/x,1/y; a/(xy)^{n-2}](xy)^{n-2} \\
			=H_{n-1}[\mu](a-(xy)^{n-2}K_{n-2}(1/x,1/y)).
		\end{multline}
	\end{remark}

	\smallskip
	
Analogously in the Toeplitz case, define the determinant
		\[
	D_n^F[\varphi; x_1, \dots, x_m; y_1, \dots, y_m; A] = \begin{pmatrix} \varphi_0 & \varphi_{-1} & \cdots & \varphi_{-n+m+1} & 1 & \cdots & 1 \\
			\varphi_1 & \varphi_0 & \cdots & \varphi_{-n+m+2} & x_1 & \cdots & x_m \\
			\vdots & \vdots & \ddots & \vdots & \vdots & \ddots & \vdots\\
			\varphi_{n-m-1} & \varphi_{n-m-2} & \cdots & \varphi_{0} & x_1^{n-m-1} & \cdots & x_m^{n-m-1} \\
			1 & y_1 & \cdots & y_1^{n-m-1} & a_{11} & \dots & a_{1m} \\
			\vdots & \vdots & \ddots & \vdots & \vdots & \ddots & \vdots\\
			1 & y_m & \cdots & y_m^{n-m-1} & a_{m1} & \dots & a_{mm}
		\end{pmatrix}.
		\]
		The following results can be proven almost identically as Proposition \ref{eq:semi-framed_det_CD}, and respectively generalize Theorems 3.5 and 1.9 of \cite{G23}.
	\begin{proposition}\label{eq:semi-framed_det_CDaa}  Let $A=(a_{jk})_{j,k=1}^m$. We have
		\begin{equation}\label{eq:thick_semifremed_CD1}
			 D^F_n[\varphi; x_1, \dots, x_m; y_1, \dots, y_m; A] =D_{n-m}[\varphi] \det\left[A - \left(\mathscr{K}_{n-m-1}(x_k, y_j)\right)_{j,k=1}^m\right],
		\end{equation}
		where	\begin{equation}\label{BOPUC Rep Ker}
			\mathscr{K}_{n}(x,y) := \sum_{j=0}^{n} Q_{j}(y)\widehat{Q}_{j}(x),
		\end{equation}
		is the reproducing kernel of the system of biorthogonal polynomials on the unit circle:
			\begin{equation}\label{bi-orthogonality intro}
			\int_{\T} Q_j(\ze)\widehat{Q}_k(\ze^{-1})\dd\varphi(\ze)= \de_{jk}, \qquad j,k \in \N \cup \{0\}.
		\end{equation}
		\end{proposition}
		\begin{corollary}\label{eq:semi-framed_det_CDbb} 
			Let $A=(a_{jk})_{j,k=1}^m$. We have
			\begin{equation}\label{eq:thick_semifremed_CD2}
				\begin{split}
					&  D_n^F[\phi; \psi^{(1)}, \dots, \psi^{(m)}; \eta^{(1)}, \dots, \eta^{(m)}; A] \\ &  =D_{n-m}[\phi] \det\left[A - \left( \int_{\T} \int_{\T} \mathscr{K}_{n-m-1}(x^{-1}_k, y^{-1}_j) \dd\psi^{(k)}(x_k) \dd\eta^{(j)}(y_j) \right)_{j,k=1}^m\right],
				\end{split} 
			\end{equation}
			where	$\mathscr{K}_{n}(x,y) $
			is the reproducing kernel of the system of biorthogonal polynomials on the unit circle given by \eqref{BOPUC Rep Ker}.
		\end{corollary}

\begin{remark}\normalfont
As with Corollary \ref{prop:bordered_det_gen}, Corollaries \ref{Cor:semi-framed_det_gen} and \ref{eq:semi-framed_det_CDbb} hold when the border measures $\nu^{(j)}$, $\eta^{(j)}$, and $\psi^{(j)}$ are signed or complex measures.
\end{remark}

The formulas of this subsection and the previous one for bordered and framed Toeplitz/Hankel determinants pave the way for their large-size asymptotic analysis, due to the Riemann-Hilbert characterizations for OPRL and BOPUC. Using such orthogonal polynomial representations, a few strong Szeg{\H o} limit theorems of the form
	\begin{equation}\label{SSLT type}
		G^n[\phi]E[\phi]\left(H[\phi;\boldsymbol{\psi};\boldsymbol{\eta}]+O(e^{-cn})\right), \qquad  \ c>0,
	\end{equation}
	  were proven for certain bordered/framed Toeplitz determinants with Szeg{\H o}-type symbols $\phi$ in \cite{BEGIL} and \cite{G23}. In Section \ref{Sec-Asymptotics}, we present a concise overview of the strong Szegő limit theorem, accompanied by precise definitions of $G[\phi]$ and $E[\phi]$ as given in equation \eqref{G and E}\footnote{For a more comprehensive exploration of this theorem, we direct readers to \cite{DIK1}.	The expression for $H[\phi;\boldsymbol{\psi};\boldsymbol{\eta}]$ (see \eqref{SSLT type}) varies depending on whether bordered or framed Toeplitz determinants are under consideration. We refer the interested reader to \cite{BEGIL} and \cite{G23} for detailed descriptions of $H$ in various settings.}.  Although the primary focus of this paper is not the computation of asymptotics using these orthogonal polynomial characterizations, we include some calculations in Section \ref{Sec-Asymptotics} to underscore their relevance.

\subsection{Applications}

In Sections \ref{sec:NIBE}, \ref{section Nonintersecting paths}, and \ref{sec 6-vertex} below, we present several examples of bordered and framed Toeplitz/Hankel determinants which appear in various integrable models in probability. 
\begin{itemize}
\item Section \ref{sec:NIBE} considers the maximal height of an ensemble of non-intersecting Brownian excursions, previously studied in \cite{Forrester-Majumdar-Schehr11, Liechty12} in the case of all particles beginning and ending at the origin, in which case the distribution of the maximal height is given by a Hankel determinant with a discrete weight. In Section \ref{sec:NIBE} we consider the case in which a few of the starting/ending points may be non-zero, in which case the  distribution of the maximal height is described by a bordered or framed Hankel determinant of a discrete weight. Using Corollaries \ref{prop:bordered_det_gen} and \ref{eq:s_framed_det_gen} we describe the corresponding formulation in terms of orthogonal polynomials and give some idea of how the asymptotic analysis would proceed and describe expected asymptotic results.
\item Section \ref{section Nonintersecting paths} considers the width of several ensembles of non-intersecting  paths:  Brownian bridges, discrete time random walks, and continuous time random walks. Each of these ensembles was  previously studied in \cite{Baik-Liu14}, in which the authors showed that for the simplest initial conditions on paths, the distribution of the width is the ratio of two Hankel determinants for non-intersecting Brownian bridges, and the ratio of two Toeplitz determinants for the non-intersecting discrete or continuous time random walks. We once again consider the case in which a few particles have starting/ending positions different from those of \cite{Baik-Liu14}, and formulate the distribution in terms of bordered and framed Toeplitz/Hankel determinants, and consequently in terms of orthogonal polynomials using Corollaries \ref{prop:bordered_det_gen} and \ref{Cor:semi-framed_det_gen}.
\item Section \ref{sec 6-vertex} considers the six-vertex model with domain wall boundary conditions (DWBC), one of the most important models in two-dimensional statistical physics. The partition function of the model may be formulated as a Hankel determinant \cite{Zinn_Justin00, Bleher-Liechty14}, and a certain variation known as the {\it partially inhomogeneous partition function} may be expressed as a bordered or framed Hankel determinant. The partially inhomogeneous partition function has a probabilistic interpretation as the generating function for certain statistics near the boundary of a state of the model  \cite{Bogoliubov-Pronko-Zvonarev02, Colomo-Pronko08, Colomo-Pronko10, Colomo-Pronko-ZinnJustin10, Gorin-Liechty23}, and for the bordered version of the Hankel determinant, an asymptotic analysis was recently performed in \cite{Gorin-Liechty23}. We briefly review this result and present the framed version as well.
\end{itemize}
 The formulas derived in Sections \ref{sec:NIBE}, \ref{section Nonintersecting paths}, and \ref{sec 6-vertex} all invite detailed asymptotic analysis in upcoming works. We do not give a complete asymptotic analysis of the examples mentioned above, but we do give some partial results and analysis to  highlight the capacity of the orthogonal polynomial characterizations. We choose to study the asymptotics of the probability of non-intersection for an ensemble of continuous time random walks as the number of random walkers approach infinity. This is the quantity which appears in the denominator of \eqref{eq:W_bordered_prop_D} and \eqref{eq:W_bordered_prop_C}. See Section \ref{Sec Discrete and continuous time random walks} for the definition of a continuous time random walk as the difference of two independent Poisson processes of rate $\lambda = 1/2$. We prove 
\begin{theorem}\label{thm asymp bordered}
	Consider $n+2$ independent continuous time random walks $X_1(t), \dots, X_{n+2}(t)$ which begin at locations $x_1<x_2<\cdots<x_{n+2}$ at time $t=0$. For a given $y_1<y_2<\dots<y_{n+2}$, let $\mathbb{P}_{n+2}(x_1,\dots, x_{n+2}; y_1,\dots,y_{n+2};T)$ be the probability that $X_1(t), \dots, X_{n+2}(t)$ satisfy $X_j(T)=y_j$ for $j=1,\dots,n+2$, and $X_1(t), \dots, X_{n+2}(t)$ do not intersect for $0\le t\le T$. Assume  $x_j=j$ for $1\leq j \leq n+2$, $y_k=k$ for $1\leq k\leq n$, while both $y_{n+1}$ and $y_{n+2}$ are arbitrary except for the restriction $n<y_{n+1}<y_{n+2}$. If $T$, $\be_2:=y_{n+2}-n$, and $\be_1:=y_{n+1}-n$ are bounded as $n\to\infty$, then $\mathbb{P}_{n+2}(x_1,\dots, x_{n+2}; y_1,\dots,y_{n+2};T)$  as $n \to \infty$ has the asymptotics
	\begin{equation}\label{Pn+2 asymp}
	\mathbb{P}_{n+2}(x_1,\dots, x_{n+2}; y_1,\dots,y_{n+2};T)=	\frac{e^{-(n+2)T}\left(\frac{T}{2}\right)^{\be_{1}+\be_{2}-3}e^{\frac{T^2}{4}}}{(\be_{1}-1)!(\be_{2}-1)!}(\be_2-\be_1) \left(1+O\left(\frac{1}{n^2}\right)\right).
	\end{equation}
\end{theorem}
\begin{theorem}\label{thm asymp framed}
 	Let $\mathbb{P}_{n+1}(x_1,\dots, x_{n+1}; y_1,\dots,y_{n+1};T)$ be as defined in Theorem \ref{thm asymp bordered}. Suppose that $x_j=j$ for $1\leq j\leq n$, $y_k=k$ for $1\leq k\leq n$, while both $x_{n+1}$ and $y_{n+1}$ are arbitrary. If $T$, $\be_1:=y_{n+1}-n$, and $\ga_1:=x_{n+1}-n$ are bounded as $n\to\infty$, then $\mathbb{P}_{n+1}(x_1,\dots, x_{n+1}; y_1,\dots,y_{n+1};T)$ as $n \to \infty$ has the asymptotics
	\begin{equation}\label{top-path-different-asymp}
		e^{-(n+1)T}e^{\frac{T^2}{4}} \left[ I_{y_{n+1}-x_{n+1}}(T) + \left(\frac{T}{2}\right)^{\ga_{1}+\be_{1}-2} \sum_{k=0}^{k_0} \frac{\left(\frac{T}{2}\right)^{-2k}}{(\ga_{1}-1-k)!(\be_{1}-1-k)!} \right] \left(1+O(e^{-c n})\right),
	\end{equation}
	where $I_\nu(\cdot)$ is the modified Bessel function of order $\nu$, $k_0:=\min\{\ga_1-1,\be_1-1\}$, and $c$ is any positive constant.
\end{theorem} 
The proofs of Theorems \ref{thm asymp bordered} and \ref{thm asymp framed} are presented in Section \ref{Sec-Asymptotics}. These calculations are performed under the simplifying assumptions that the total time $T$ and the length of the gaps between consecutive starting and ending points of the random walks are bounded. It is very interesting to study the situations in which  these quantities grow at various rates with $n$, the number of random walks.
\begin{remark}
	\normalfont  It may be worth mentioning that when $x_j=j$ and $y_k=k$ for all $1\leq j,k\leq n$, then one has to analyze the (pure) Toeplitz determinant with symbol $e^{\frac{T}{2}(z+z^{-1}-2)}$ and the strong Szeg{\H o} theorem implies  \begin{equation}
			\mathbb{P}_{n}(1,2,\cdots,n;1,2,\cdots,n;T) = e^{-nT}e^{\frac{T^2}{4}}\left(1+O(e^{-c n})\right), \qquad c>0.
		\end{equation}
		Also, comparing with Theorem \ref{thm asymp bordered}, if we have $n+1$ continuous time random walks where only the top path has a gapped ending point $y_{n+1}>n+1$, then one has to asymptotically analyze the associated single bordered Toeplitz determinant, and as shown in Theorem \ref{thm 7.2}, \[ \mathbb{P}_{n+1}(1,2,\cdots,n+1;1,2,\cdots,n,y_{n+1};T) = \frac{e^{-(n+1)T}\left(\frac{T}{2}\right)^{\be_{1}-1}e^{\frac{T^2}{4}}}{(\be_{1}-1)!}\left(1+O(e^{-c n})\right), \qquad c>0, \]
		where $\be_1:= y_{n+1}-n$. Comparing this with \eqref{Pn+2 asymp} invites one to expect the asymptotic formula 
			\begin{equation}\label{Pn+m asymp}
			\mathbb{P}_{n+m}(1,\dots, n+m; 1,\dots,n,y_{n+1},\ldots,y_{n+m};T)=	\frac{e^{\frac{T^2}{4}}e^{-(n+m)T} \prod_{j=1}^{m} \left(\frac{T}{2}\right)^{  \be_{j}-j}}{\prod_{j=1}^{m}(\be_{j}-j)!}\prod_{1\le j<k\le m}(\be_k-\be_j) \left(1+o(1)\right),
		\end{equation}
with $\be_j:= y_{n+j}-j$, although we do not propose it as a conjecture yet.
\end{remark}

\subsection{Outline for the rest of the paper}
Sections \ref{sec:NIBE}, \ref{section Nonintersecting paths}, and \ref{sec 6-vertex} deal with the applications described above. Propositions \ref{eq:bordered_det}, \ref{Prop: multi-bordered Toeplitz and BOPUC}, \ref{eq:semi-framed_det_CD}, and \ref{eq:semi-framed_det_CDaa} are then proven in Section \ref{sec:proofs of OP props}. The bordered and framed Hankel determinant formulas appearing in Sections \ref{sec:NIBE} and \ref{section Nonintersecting paths} are proven in Section \ref{proofs of props}. Finally in Section \ref{Sec-Asymptotics} we prove the asymptotic results of Theorems \ref{thm asymp bordered} and \ref{thm asymp framed}.

\section{Height of non-intersecting Brownian excursions with wanderers}\label{sec:NIBE}

Consider $n$ Brownian excursions $X_1(t), \dots, X_n(t)$, i.e., Brownian bridges on $\R_+$ with an absorbing wall at zero. For $0\le m <n$ and vectors $\vec\al = (\al_1, \dots, \al_m)$, $\vec\be = (\be_1, \dots, \be_m)$, satisfying $\al_1 \ge \al_2 \ge \dots \ge \al_m \ge 0$ and $\be_1\ge \be_2 \ge \dots \ge \be_m\ge 0$, we will condition the Brownian excursions as follows (also see Figure \ref{BE-wanderers}):
\begin{itemize}
\item At time $t=0$ we start the paths at $X_1(0)=X_2(0) = \cdots = X_{n-m}(0) = 0$, and $X_n(0) = \al_1, X_{n-1}(0) = \al_2, \cdots X_{n-m+1}(0) = \al_m$.
\item At time $t=1$ the paths end at the points $X_1(1) = X_{2}(1) = \dots = X_{n-m}(1) = 0$, and $X_n(1) =\be_1, X_{n-1}(1) = \be_2, \dots, X_{n-m+1}(1)= \be_m.$
\item The particles do not intersect at times $0<t<1$, and are ordered as $X_1(t) <X_2 (t) < \dots < X_n(t)$.
\end{itemize}
We denote  the conditional probability described above as $\mathbb{P}_m^{\vec\al, \vec\be}$, and refer to the model as {\it $n$ non-intersecting Brownian excursions with wanderers on the time interval $[0,1]$}. Our primary interest in this section is in the random variable 
\[
\mathcal{M}_n := \max_{t\in (0,1)} X_n(t).
\]

\begin{center}
 \begin{figure}[h]
\begin{center}
   \scalebox{0.32}{\includegraphics{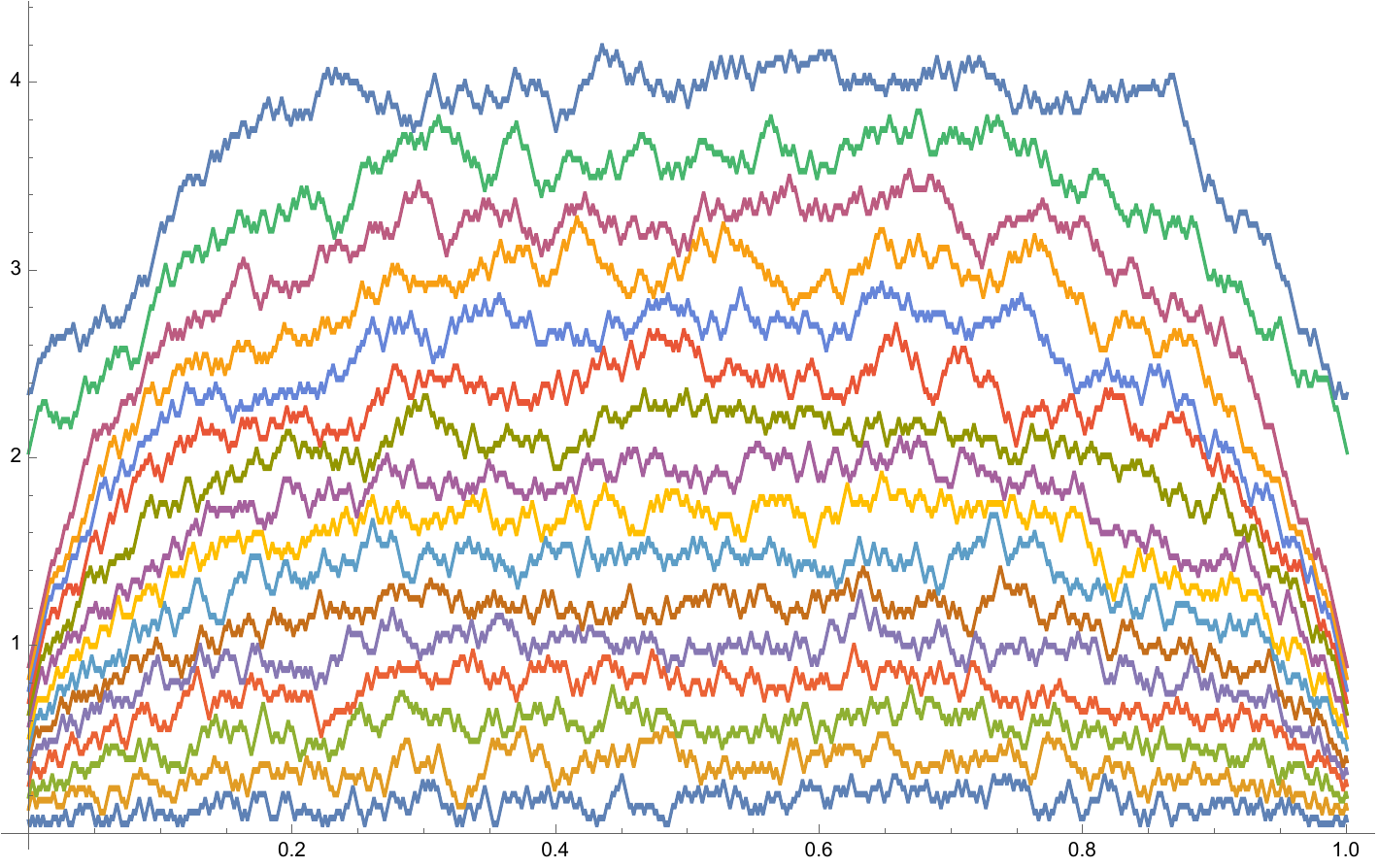}\hspace{1cm}\includegraphics{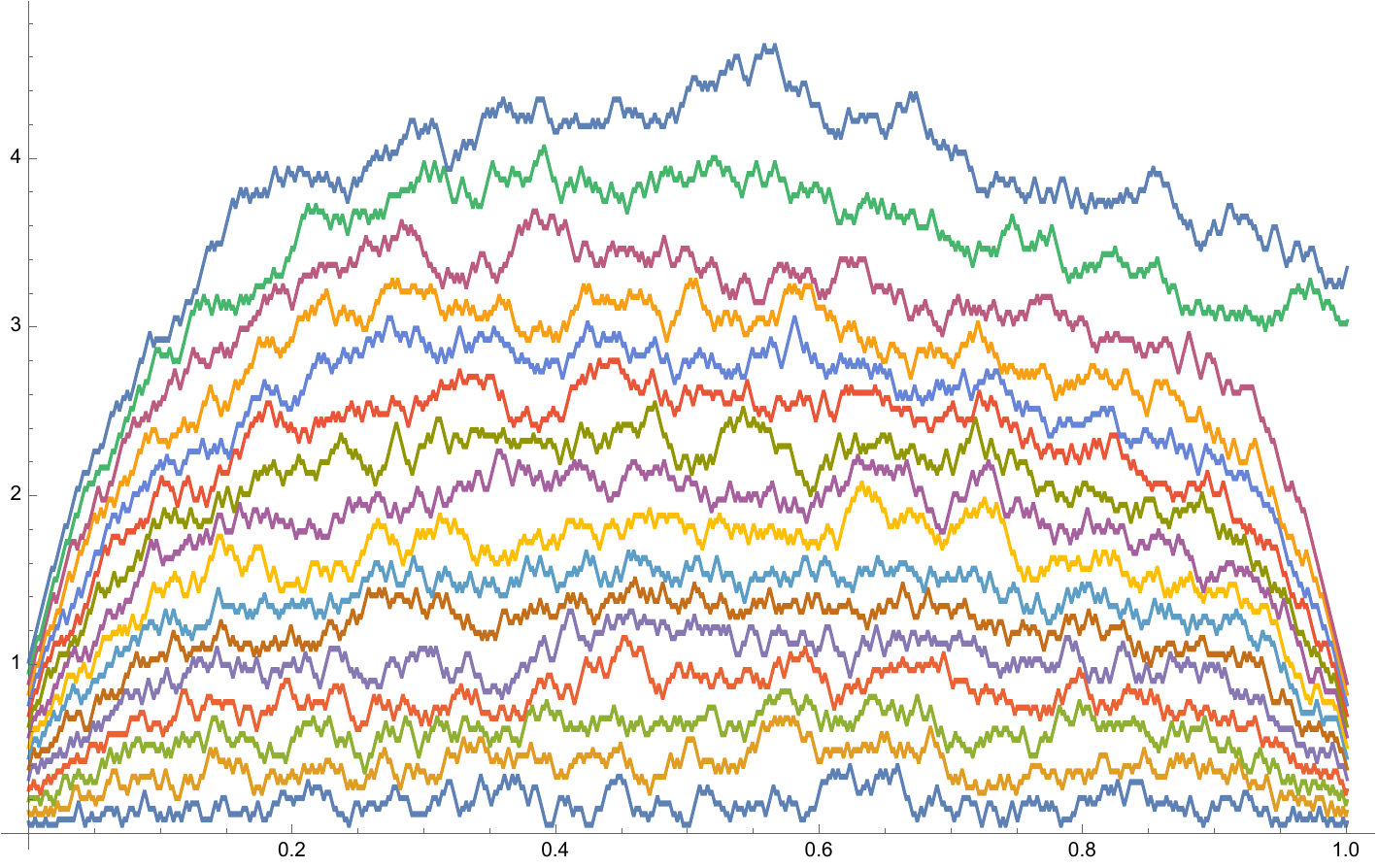}}
\end{center}
        \caption[]{Two simulations of non-intersecting Brownian excursions with wanderers with $n=16$ and $m=2$. In the figure on the left, $
        \al_1, \al_2, \be_1$, and $\be_2$ are all positive and the distribution of $\mathcal{M}_n$ is described in \eqref{eq:framed_M_dist} by a framed Hankel determinant. In the figure on the right, $
        \al_1= \al_2 =0$, and the distribution of $\mathcal{M}_n$ is described in \eqref{eq:bordered_M_dist} by a bordered Hankel determinant.}\label{BE-wanderers}
    \end{figure}
\end{center}

In the paper \cite{Liechty12} (see also \cite{Forrester-Majumdar-Schehr11}) it was shown that the in the case $m=0$, i.e., all Brownian particles begin and end at 0, distribution of the random variable $\mathcal{M}_n$ is expressed as a Hankel determinant with respect to a discrete measure. Here we describe the case $m>0$, in which most of the particles begin and end at 0, but a few have positive starting/ending points. In this case the discrete Hankel determinant is replaced by a bordered or framed Hankel determinant. To describe these determinants introduce the discrete measure supported on perfect squares\footnote{While it is perhaps more natural to replace $h\mapsto h^2$ in the definition of $\mu$ so that the measure is supported on the positive integers, we choose to write $\mu$ as supported on the perfect squares so that the relevant Hankel matrices are written with all moments of $\mu$ rather than just the even ones.}
\eq\label{def:mu_NIBE}
\mu=\sum_{h= 1, 4, 9, 16, \dots} h e^{-\frac{\pi^2 h}{2M^2}}\de_h, 
\eeq
as well as the signed measure depending on a parameter $\al$, 
\eq\label{def:nu_NIBE}
\nu(\al)  = \frac{1}{\al}\sum_{h= 1, 4, 9, 16, \dots}  \left(\sqrt{h} e^{-\frac{\pi^2h}{2M^2}}\sin\left(\frac{\sqrt{h} \pi\al}{M}\right)\right)\de_h,
\eeq
and the constants depending on two parameters
\eq\label{def:s_NIBE}
s(\al,\be) = \frac{1}{\al\be}\sum_{h=1,4,9,16,\dots} e^{-\frac{\pi^2h}{2M^2}}\sin\left(\frac{\sqrt{h} \pi\al}{M}\right)\sin\left(\frac{\sqrt{h} \pi\be}{M}\right) = \frac{1}{\al\be}\sum_{h=1}^\infty e^{-\frac{\pi^2h^2}{2M^2}}\sin\left(\frac{h \pi\al}{M}\right)\sin\left(\frac{h\pi\be}{M}\right).
\eeq
We have the following proposition describing the distribution of $\mathcal{M}_n$ in terms of a framed or bordered Hankel determinant.
\begin{proposition}\label{prop:NIBE_bordered_framed}
Fix integers $n$ and $m$ with $0\le m <n$ as well as vectors $\vec\al$ and $\vec \be$ satisfying  $\al_1 > \al_2 > \dots > \al_m > 0$ and $\be_1> \be_2 > \dots > \be_m> 0$. The distribution of the maximum height $\mathcal{M}_n$ of $X_n(t)$ is given by the formula
\begin{multline}\label{eq:framed_M_dist}
\mathbb{P}_m^{\vec\al, \vec\be}(\mathcal{M}_n < M) =\frac{\displaystyle{2^{n/2} \pi^{2(n-m)^2+n/2}\prod_{j=1}^m e^{\frac{1}{2}(\al_j^2+\be_j^2)}}}{\displaystyle{M^{2(n-m)^2+n}\det\left[\frac{\sinh(\al_{m-j+1}\be_{m-k+1})}{\al_{m-j+1}\be_{m-k+1}} -\sum_{\ell = 0}^{n-m-1} \frac{(\al_{m-j+1}\be_{m-k+1})^{2\ell}}{(2\ell+1)!}\right]_{j,k=1}^m\prod_{j=1}^{n-m} (2j-1)!}}  \\
\times 
H_n^F(\mu; \nu(\be_m), \nu(\be_{m-1}), \dots, \nu(\be_1); \nu(\al_m), \nu(\al_{m-1}), \dots, \nu(\al_1); S) ,
\end{multline} for all $M \ge \max\{\al_1, \be_1\},$
where the framed Hankel determinant $H_n^F(\mu; {\boldsymbol \nu}_m; {\boldsymbol \eta}_m; S)$ is as defined in \eqref{def:semi_framed_hankelM}, the measures $\mu$ and $\nu(\al)$ are defined in \eqref{def:mu_NIBE} and \eqref{def:nu_NIBE}, and the $m\times m$ matrix $S$ is given by
\[
S = \bigg(s(\al_{m-j+1}, \be_{m-k+1})\bigg)_{j,k=1}^m,
\]
with $s(\al,\be)$ defined in \eqref{def:s_NIBE}.

In the case all particles begin at zero, but the top $m$ particles end at $0<\be_m<\be_{m-1}<\cdots<\be_1$, the distribution of the maximum height $\mathcal{M}_n$ is given by the formula
\begin{multline}\label{eq:bordered_M_dist}
\mathbb{P}_m^{\vec 0, \vec \be}( \mathcal{M}_n < M ) =
 (-1)^{m(n-1)}\frac{2^{n/2} \pi^{n^2+(n-m)^2+n/2}}{M^{n^2+(n-m)^2+n}\De({\vec \be}^2)}\prod_{j=1}^m \frac{e^{\frac{1}{2}\be_j^2}}{\be_j^{2(n-m)}}\prod_{j=1}^{n-m} \left(\frac{1}{(2j-1)!}\right)  \\
\times  H_n^B[\mu; \nu(\be_m), \dots, \nu(\be_1)], \end{multline}
where $\De(\vec \beta^2) = \prod_{1\le j<k\le m} (\be_k^2-\be_j^2)$ and the bordered Hankel determinant $H_n^B[\mu; \vec{\boldsymbol \nu}_m]$ is as defined in \eqref{def:bordered_HMatrix}.
\end{proposition}
\begin{remark}\label{rem:distinct_ab}
The above Proposition is stated for distinct $\al_1,\dots, \al_{m}$ and $\be_1,\dots, \be_{m}$. If any of the $\al_j$'s or $\be_j$'s coincide, then the framed or bordered Hankel determinants in \eqref{eq:framed_M_dist} and \eqref{eq:bordered_M_dist} vanish, as do the $m\times m$ determinants in the denominators of \eqref{eq:framed_M_dist} and \eqref{eq:bordered_M_dist}. In this case those formulas may be understood using l'H\^{o}pitals rule.
\end{remark}

In the case $m=0$, i.e., all starting and ending points of the non-intersecting Brownian excursions are at 0, it was proven in \cite{Liechty12} that for all $x\in \R$, 
\[
\lim_{n\to\infty} \mathbb{P}_0(\mathcal{M}_n < \sqrt{2n} - 2^{-7/6} n^{1/6} x) = F_1(2^{2/3} x),
\]
where $F_1(x)$ is the cumulative distribution function for the Tracy--Widom GOE (TW-GOE) distribution which describes the location of the largest eigenvalue in the Gaussian orthogonal ensemble of real symmetric random matrices \cite{Tracy-Widom96}. This result is quite expected since the top path $X_n(t)$ converges to the Airy$_2$ process minus a parabola in the proper scaling limit\footnote{It is well known that the top path in $n$ non-intersecting Brownian {\it bridges} conditioned to have all starting and ending points at zero converges in the proper scaling limit to the Airy$_2$ process minus a parabola. For non-intersecting Brownian excursions, the convergence of the top path to the Airy$_2$ process minus a parabola follows from combining a finite dimensional convergence implied by \cite[Section 5.3]{Tracy-Widom07} with the results of \cite{Corwin-Hammond11}.}, and it is well known that that the distribution of the maximum of the Airy$_2$ process minus a parabola is given by TW-GOE (up to a constant rescaling factor) \cite{Johansson03, Corwin-Quastel-Remenik13}.

When $n-m$ starting and ending points are at zero, but $m$ of them are positive so that $X_{n-j+1}(0) = \al_j$ and $X_{n-j+1}(1) = \be_j$ for $j=1,\dots, m$, then a different limiting process may emerge as the scaling limit of the top path. In particular, if $m$ is fixed and $\al_j$ and $\be_j$ are scaled with $n$ as 
\eq\label{eq:Airy_wanderers_scaling}
\al_j = \sqrt{2n} - (2n)^{1/6} a_j, \quad \be_j = \sqrt{2n} - (2n)^{1/6} b_j,
\eeq
then the top path $X_n(t)$ should converge\footnote{The works \cite{Adler-Delepine-van_Moerbeke09} and \cite{Adler-Ferrari-van_Moerbeke10} proved this convergence for non-intersecting Brownian bridges, not non-intersecting Brownian excursions.} in a certain scaling limit as $n\to\infty$ to a perturbation of the Airy$_2$ process known as the {\it Airy process with wanderers}, again minus a parabola. The Airy process with wanderers  depends on the parameters $a_1, \dots, a_m$ and $b_1,\dots, b_m$, and we denote it as $\mathcal{A}_m^{(\vec a, \vec b)}(t)$, where $\vec a = (a_1, \dots, a_m)$ and $\vec b = (b_1,\dots, b_m)$.  This first appeared in the paper \cite{Borodin-Peche08} in the context of last passage percolation with boundary sources, and then in \cite{Adler-Delepine-van_Moerbeke09, Adler-Ferrari-van_Moerbeke10} as a scaling limit of non-intersecting paths. More recently in \cite{Liechty-Nguyen-Remenik22}, several formulas were obtained for the distribution of the supremum of $\mathcal{A}_m^{(\vec a, \vec b)}(t)-t^2$ for $0\le a_1\le \dots\le a_m$ and $0\le b_1\le\dots\le b_m$\footnote{The condition that all $a_j$'s and $b_j$'s are positive is essential to get a nontrivial limit. Otherwise the probability converges to 0.}. We denote this distribution function as $F_m^{(\vec a, \vec b)}(x)$:
\eq
F_m^{(\vec a, \vec b)}(x) := \mathbb{P}\left(\sup_{t\in \R} \left(\mathcal{A}_m^{(\vec a, \vec b)}(t)-t^2\right)<x\right).
\eeq
In \cite{Liechty-Nguyen-Remenik22}, $F_m^{(\vec a, \vec b)}(x)$ is expressed as a Fredholm determinant (Theorem 1) and in terms of a Lax pair for the Painlev\'{e} II equation (Theorem 2). Both formulas have well defined limits as the parameters $a_j$ and/or $b_j$ go to $+\infty$. In view of the scaling \eqref{eq:Airy_wanderers_scaling}, this corresponds to pushing the starting and/or ending points of the non-intersecting paths to 0. In particular we have
\[
\lim_{\substack {\vec \al \to \vec 0 \\ \vec \be \to \vec 0}} F_m^{(\vec a, \vec b)}(x)=:F_m^{(\infty, \infty)}(x) = F_1(2^{2/3} x),
\]
where $F_1(x)$ is the TW-GOE distribution function. We can also take the parameters $a_j$ to $+\infty$ but leave $\vec b$ finite, which corresponds to the starting points of the non-intersecting paths going to 0 but the top $m$ ending points still positive. We denote this limit as 
\[
F_m^{(\infty, \vec b)}(x) := \lim_{\vec \al \to \vec 0 } F_m^{(\vec a, \vec b)}(x).
\]
In light of the expected convergence of $X_n(t)$ to $\mathcal{A}_m^{(\vec a, \vec b)}(t)$ in the proper scaling limit, we therefore expect that for $\al_j$ and $\be_j$ scaled as in \eqref{eq:Airy_wanderers_scaling} with $a_j>0$ and $b_j>0$ for $j=1,\dots,m$,
\eq\label{eq:limits_of_NIME}
\begin{aligned}
\lim_{n\to\infty} \mathbb{P}^{\vec \al, \vec \be}_m(\mathcal{M}_n < \sqrt{2n} - 2^{-7/6} n^{1/6} x) &= F_m^{(\vec a, \vec b)}(x), \\
\lim_{n\to\infty} \mathbb{P}^{\vec 0, \vec \be}_m(\mathcal{M}_n < \sqrt{2n} - 2^{-7/6} n^{1/6} x) &= F_m^{(\infty, \vec b)}(x).
\end{aligned}
\eeq

Somewhat surprisingly, \cite[Corollary 1]{Liechty-Nguyen-Remenik22} states that the two limiting functions described above are essentially the same. More precisely it states
\[
 F_m^{(\vec a, \vec b)}(x) =  F_{2m}^{(\infty, \langle \vec a, \vec b\rangle)}(x),
 \]
 where $\langle \vec a, \vec b\rangle$ is the reordering of the combined elements of $\vec a$ and $\vec b$ in non-decreasing order. Assuming \eqref{eq:limits_of_NIME}, this would imply that $\mathbb{P}^{\vec \al, \vec \be}_m(\mathcal{M}_n < \sqrt{2n} - 2^{-7/6} n^{1/6} x)$, expressed in \eqref{eq:framed_M_dist} in terms of a framed Hankel determinant, and  and $\mathbb{P}^{\vec 0, ( \vec \al, \vec \be )}_{2m}(\mathcal{M}_n < \sqrt{2n} - 2^{-7/6} n^{1/6} x)$, expressed in \eqref{eq:bordered_M_dist} as a bordered Hankel determinant, have the same scaling limit as $n\to\infty$ when $\vec \al$ and $\vec \be$ are scaled as \eqref{eq:Airy_wanderers_scaling} with $a_j,b_j>0$, and  $( \vec \al, \vec \be )$ is the reordering of the combined elements of $\vec a$ and $\vec b$ in non-increasing order.
 
 In order to evaluate the formulas \eqref{eq:framed_M_dist} and \eqref{eq:bordered_M_dist} as $n\to\infty$, it is useful to use Propositions \ref{eq:bordered_det_gen} and \ref{eq:s_framed_det_gen} to express them in terms of orthogonal polynomials. Then one may hope to evaluate them asymptotically as $n\to\infty$ provided asymptotics of the orthogonal polynomials are known. To apply the Propositions \ref{eq:bordered_det_gen} and \ref{eq:s_framed_det_gen}, we first introduce the system of monic orthogonal polynomials $\{P_k(x)=x^k+\cdots\}_{k=0}^\infty$ with respect to the discrete measure $\mu$ given in \eqref{def:mu_NIBE}:
\begin{equation}
\sum_{x=1,4,9,16,\dots} P_k(x)P_j(x) xe^{-\frac{\pi^2}{2M^2} x} = h_k \de_{jk},
\end{equation}
or equivalently
\begin{equation}\label{eq:NIBE_OP_def}
\sum_{x=1}^\infty P_k(x^2)P_j(x^2) x^2 e^{-\frac{\pi^2}{2M^2} x^2} = h_k \de_{jk}.
\end{equation}
Applying Corollary \ref{eq:semi-framed_det_CDbb} to \eqref{eq:framed_M_dist} gives
\begin{multline}\label{eq:NIBE_framed_OPs}
\mathbb{P}_m^{\vec\al, \vec\be}(\mathcal{M}_n < M) = \frac{2^{n/2} \pi^{2(n-m)^2+n/2}\prod_{j=1}^m e^{\frac{1}{2}(\al_j^2+\be_j^2)}H_{n-m}[\mu] }{\displaystyle{M^{2(n-m)^2+n}\det\left[\frac{\sinh(\al_{m-j+1}\be_{m-k+1})}{\al_{m-j+1}\be_{m-k+1}} -\sum_{\ell = 0}^{n-m-1} \frac{(\al_{m-j+1}\be_{m-k+1})^{2\ell}}{(2\ell+1)!}\right]_{j,k=1}^m\prod_{j=1}^{n-m} (2j-1)!}}   \\
\times \det\left[ S - \left(\sum_{x,y=1}^\infty K_{n-m-1}(x^2, y^2)xye^{-\frac{\pi^2(x_j+y_k^2)}{2M^2}}\sin\left(\frac{x \pi \be_k}{M}\right)\sin\left(\frac{x \pi \al_j}{M}\right)\right)_{j,k=1}^m\right],
\end{multline}
where
\[
K_{n-m-1}(x, y)=\sum_{k=0}^{n-m-1} \frac{P_k(x)P_k(y)}{h_k}, \quad S = \left( \sum_{h=1}^\infty \frac{e^{-\frac{\pi^2h^2}{2M^2}}\sin\left(\frac{h \pi\al_{m-j+1}}{M}\right)\sin\left(\frac{h\pi\be_{m-k+1}}{M}\right)}{\al_{m-j+1}\be_{m-k+1}}\right)_{j,k=1}^m.
\]
Applying Corollary  \ref{prop:bordered_det_gen} to \eqref{eq:bordered_M_dist} gives
\begin{multline}\label{eq:NIBE_bordered_OPs}
\mathbb{P}_m^{\vec 0, \vec \be}( \mathcal{M}_n < M ) =  (-1)^{n(n-1)/2+(n-m)(n-m-1)/2}\frac{2^{n/2} \pi^{n^2+(n-m)^2+n/2}}{M^{n^2+(n-m)^2+n}\De({\vec \be}^2)}\prod_{j=1}^m \frac{e^{\frac{1}{2}\be_j^2}}{\be_j^{2(n-m)}} \prod_{j=1}^{n-m} \left(\frac{1}{(2j-1)!}\right) \\
\times H_{n-m}[\mu]\det\left[\sum_{x=1}^\infty \left(x e^{-\frac{\pi^2x_j^2}{2M^2}}\sin\left(\frac{x \pi \be_j}{M}\right)\right) P_{n-m+k-1}(x^2)\right]_{j,k=1}^m.
\end{multline}

\medskip 

To evaluate \eqref{eq:NIBE_framed_OPs} and \eqref{eq:bordered_det_gen} asymptotically as $n\to\infty$, one first needs large-degree asymptotic formulas for the orthogonal polynomials \eqref{eq:NIBE_OP_def}. In general this can be accomplished by the Riemann--Hilbert method \cite{DKMVZ, BKMM, Bleher-Liechty14}. Indeed, for the orthogonal polynomials \eqref{eq:NIBE_OP_def}, the relevant Riemann--Hilbert analysis was performed in \cite{Liechty12}, although explicit asymptotic formulas for the orthogonal polynomials do not appear there. Nonetheless, one can use the steepest descent analysis of the Riemann--Hilbert problem from \cite{Liechty12} to extract asymptotic formulas for the orthogonal polynomials \eqref{eq:NIBE_OP_def} and then insert the asymptotic formulas for the orthogonal polynomials into \eqref{eq:NIBE_framed_OPs} and \eqref{eq:bordered_det_gen}. This is the approach outlined in \cite[Appendix A]{Liechty-Nguyen-Remenik22} to prove \eqref{eq:limits_of_NIME} with $m=2$ and $\vec\al=\vec 0$, though few details are given there. Here we repeat the remark from \cite[Appendix A]{Liechty-Nguyen-Remenik22} that larger values of $m$ require finer asymptotic formulas for the orthogonal polynomials, thus the analysis becomes increasingly difficult for larger values of $m$.

\section{Width of non-intersecting processes}\label{section Nonintersecting paths}

The paper \cite{Baik-Liu14} considers the width of various ensembles of non-intersecting paths: non-intersecting Brownian bridges, discrete random walks, and continuous time random walks. For the simplest initial conditions on the random paths, the distribution of the width of the ensemble is described as the ratio of two Hankel determinants for the non-intersecting Browinan bridges, and the ratio of two Toeplitz determinants for the non-intersecting random walks. In this section we describe how the initial conditions may be perturbed slightly so that the Toeplitz/Hankel determinants from \cite{Baik-Liu14} are replaced with bordered or framed Toeplitz/Hankel determinants.

\subsection{Non-intersecting Brownian bridges with wanderers}
Consider $n$ independent Brownian motions \\ $X_1(t), \dots, X_n(t)$. For integers $m_1 \ge 0$ and $m_2\ge 0$ satisfying $m_1+m_2=m<n$. define the vectors  $\vec\al = (\al_1, \dots, \al_m)$, $\vec\be = (\be_1, \dots, \be_m)$, satisfying $\al_1 \le \al_2 \le \cdots \le \al_{m_1} < 0 < \al_{m_1+1} \le \cdots \le \al_{m_1+m_2}$ and $\be_1 \le \be_2 \le \cdots \le \be_{m_1} < 0 < \be_{m_1+1} \le \cdots \le \be_{m_1+m_2}$. Similar to the non-intersecting Brownian excursions in Section \ref{sec:NIBE}, we will condition the Brownian motions as follows (also see Figure \ref{BB-wanderers1}):
\begin{itemize}
\item At time $t=0$ we start the paths at $X_{m_1+1}(0)=X_{m_1+2}(0) = \dots = X_{n-m_2}(0) = 0$, as well as $X_1(0) = \al_1, X_{2}(0) = \al_2, \dots, X_{m_1}(0) = \al_{m_1},$ and $X_{n-m_2+1}(0)=\al_{m_1+1}, X_{n-m_2+2}(0)=\al_{m_1+2}, \dots, X_n(0) = \al_{m_1+m_2}.$
\item At time $t=1$ the paths end at the points $X_{m_1+1}(1)=X_{m_1+2}(1) = \dots = X_{n-m_2}(1) = 0$, as well as $X_1(1) = \be_1, X_{2}(1) = \be_2, \dots, X_{m_1}(1) = \be_{m_1},$ and $X_{n-m_2+1}(1)=\be_{m_1+1}, X_{n-m_2+2}(1)=\be_{m_1+2}, \dots, X_n(1) = \be_{m_1+m_2}.$
\item The particles not to intersect at times $0<t<1$, and are ordered as $X_1(t) <X_2 (t) < \dots < X_n(t)$.
\end{itemize}
We refer to this model as {\it non-intersecting Brownian bridges with wanderers} and denote the conditional probability described above as $\mathbb{P}_{m_1,m_2}^{\vec\al,\vec\be}$.

\begin{center}
 \begin{figure}[h]
\begin{center}
   \scalebox{0.25}{\includegraphics{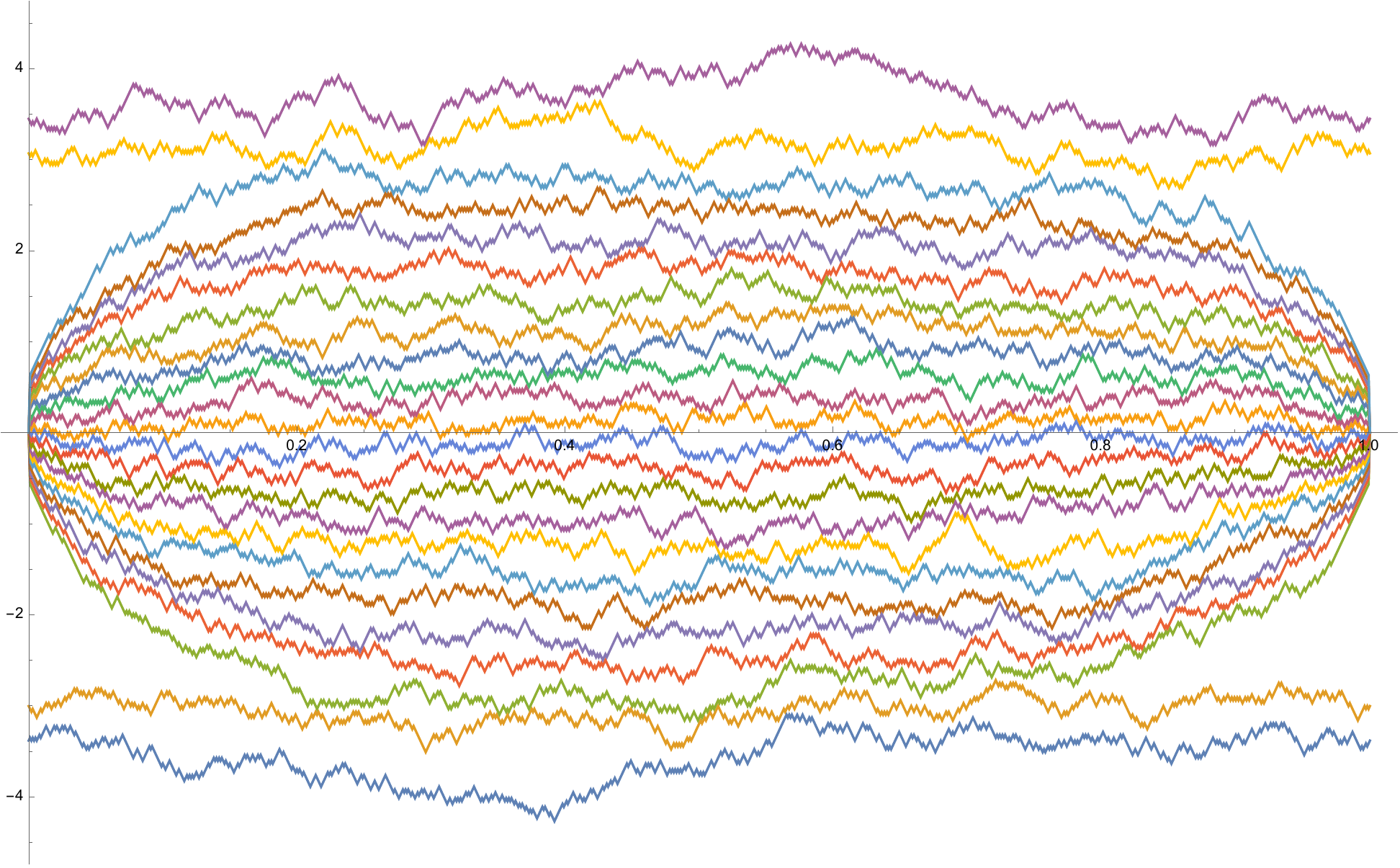}\hspace{1cm}\includegraphics{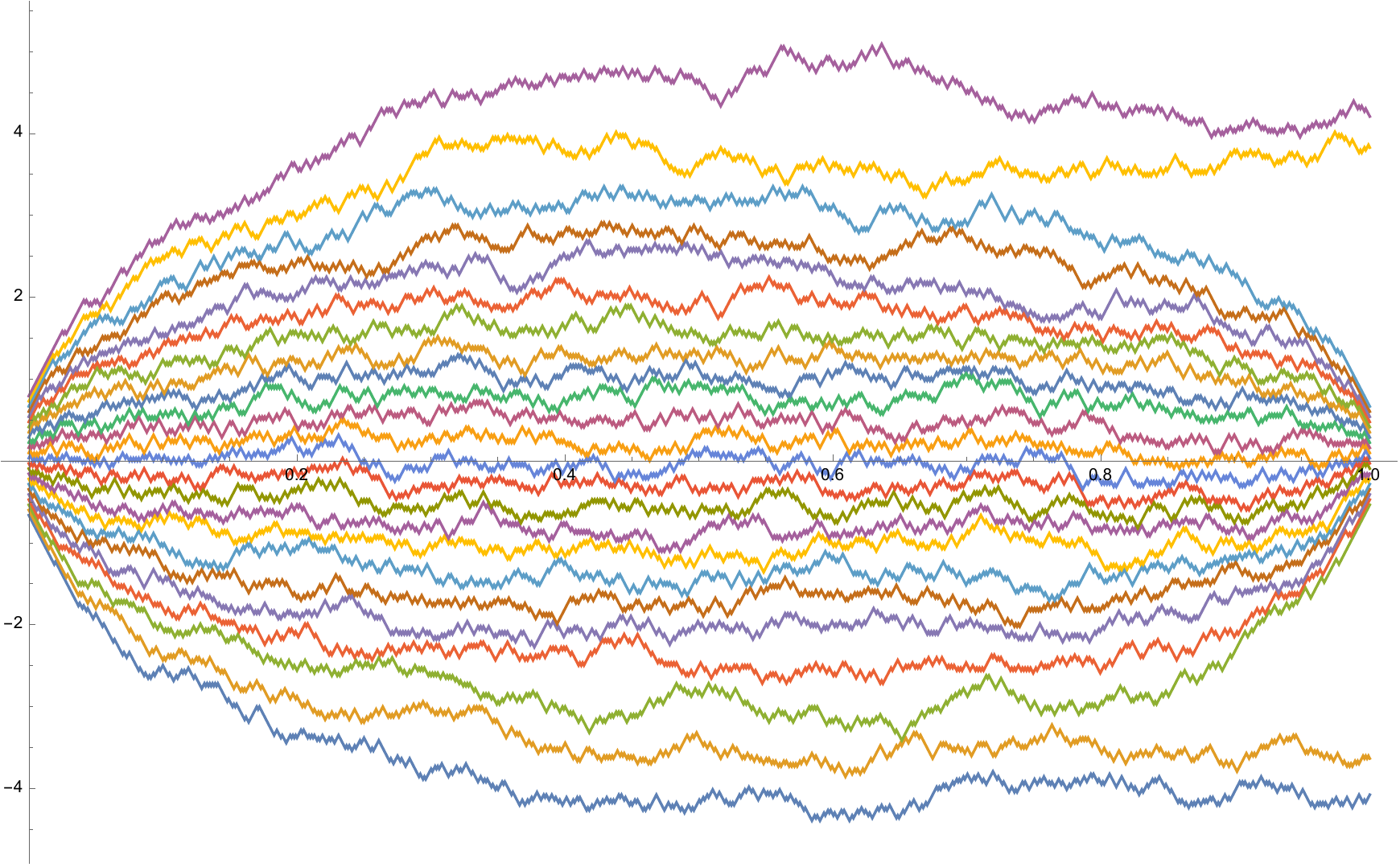}}
\end{center}
        \caption[]{Two simulations of non-intersecting Brownian bridges with wanderers with $n=24$ and $m_1=m_2=2$. In the figure on the left, $
        \al_1, \al_2, \al_3, \al_4, \be_1, \be_2, \be_3$, and $\be_4$ are all non-zero and the distribution of $W_n$ is described in \eqref{eq:framed_W_dist} in terms of framed Hankel determinants. In the figure on the right, $
        \al_1= \al_2 =\al_3 = \al_4=0$, and the distribution of $W_n$ is described in \eqref{eq:bordered_W_dist} in terms of bordered Hankel determinants.}\label{BB-wanderers1}
    \end{figure}
\end{center}

We define the width of the ensemble as 
\[
\mathscr{W}_n := \max_{0\le t \le 1}\{X_n(t) - X_1(t)\}.
\] 
For the case $m_1=m_2=0$, in which all particles begin and end at 0, the distribution of $\mathscr{W}_n$ was presented in \cite{Baik-Liu14} as the integral of the ratio of two Hankel determinants with respect to Gaussian (or Gaussian-like) measures, one continuous and one discrete. Here we present the analogous formulas in terms of bordered and/or framed Hankel determinants when $m_1$ or $m_2$ is nonzero. In order to present the formulas we first define the continuous measure $\mu$ and discrete measure $\mu^{M,\tau}$ as
\eq\label{def:mu_NIBB}
\dd\mu(x) = e^{-x^2/2}\dd x, \quad \mu^{M,\tau} = \frac{2\pi}{M}\sum_{x\in \mathcal{D}_{M,\tau}} e^{-x^2/2}\de_{x},
\eeq
where
\eq\label{def:discrete_set}
\mathcal{D}_{M,\tau} = \left\{ \frac{2\pi (k-\tau)}{M}   \ : \ k\in \Z\right\}.
\eeq
We also introduce the complex measures $\mu_{\be}$ and $\mu^{M,\tau}_{\be}$ as
\eq\label{def:mu_beta_NIBB}
\dd\mu_{\be}(x) = e^{-x^2/2-\ii\be x}\dd y, \quad \mu^{M,\tau}_{\be} = \frac{2\pi}{M}\sum_{x\in \mathcal{D}_{M,\tau}} e^{-x^2/2-\ii\be x}\de_{x},
\eeq
We then have the following proposition, which relies on the Karlin--McGregor formula \cite{Karlin-McGregor59} for non-intersecting processes. We present a proof in Section \ref{proofs of props}.

\begin{proposition}\label{prop:NIBB_bordered_framed}
Fix integers $n$, $m_1$, and $m_2$ with $m_1\ge 0$, $m_2\ge 0$, and  $m=m_1+ m_2 <n$ as well as vectors $\vec\al=(\al_1,\dots,\al_m)$ and $\vec \be=(\be_1,\dots,\be_m)$ satisfying  $\al_1 < \al_2 < \cdots < \al_{m_1} < 0 < \al_{m_1+1} < \cdots < \al_{m_1+m_2}$ and $\be_1 < \be_2 < \cdots < \be_{m_1} < 0 < \be_{m_1+1} < \cdots < \be_{m_1+m_2}$. The distribution of the width of the non-intersecting Brownian bridges with wanderers is given by    by the formula
\eq\label{eq:framed_W_dist}
\mathbb{P}_{m_1,m_2}^{\al,\be}(\mathscr{W}_n < M) = \int_0^{1} \frac{H_n^F[\mu^{M,\tau};\mu^{M,\tau}_{\be_{1}},\mu^{M,\tau}_{\be_{2}},\dots,\mu^{M,\tau}_{\be_{m}}; \mu^{M,\tau}_{(-\al_{1})},\dots,\mu^{M,\tau}_{(-\al_{m})}; A^{M,\tau}]}{H_n^F[\mu;\mu_{\be_{1}}\mu_{\be_{2}},\dots, \mu_{\be_{m}}; \mu_{(-\al_{1})},\dots,\mu_{(-\al_{m})}; A]}\dd\tau,
\eeq
 for all $M \ge \max\{\al_m-\al_1, \be_m-\be_1\},$
where the framed Hankel determinant $H_n^F(\mu; \vec{\boldsymbol \nu}_m; \vec{\boldsymbol \eta}_m; A)$ is as defined in \eqref{def:semi_framed_hankelM},  and the $m\times m$ matrix $A$ is given by
\eq\label{eq:const_matrix_NIBB}
\begin{aligned}
A&=\left( \int_\R e^{-y^2/2 +\ii y(\al_{k}-\be_{j})}\,\dd y\right)_{j,k=1}^m=(2\pi)^{m/2}\left( e^{- (\al_{k}-\be_{j})^2/2}\right)_{j,k=1}^m \\
 A^{M,\tau}&=\left( \frac{2\pi}{M}\sum_{y\in \mathcal{D}_{M,\tau}}  e^{-y^2/2 +\ii y(\al_{k}-\be_{j})}\right)_{j,k=1}^m. \\
\end{aligned}
\eeq
In the case all particles begin at zero, but the bottom $m_1$ particles end at $\be_1<\be_{2}<\cdots<\be_{m_1}<0$ and the top $m_2$ particles end at $0<\be_{m_1+1}<\dots<\be_m$, the distribution of the $\mathscr{W}_n$ is given by the formula
\eq\label{eq:bordered_W_dist}
\mathbb{P}_{m_1,m_2}^{\vec 0,\vec\be}(\mathscr{W}_n < M) = \int_0^{1} \frac{H_n^B[\mu^{M,\tau};\mu^{M,\tau}_{\be_1},\dots, \mu^{M,\tau}_{\be_{m}} ]}{H_n^B[\mu; \mu_{\be_1},\dots, \mu_{\be_{m}} ]}\dd\tau,
\eeq
where the bordered Hankel determinant $H_n^B[\mu; \vec{\boldsymbol \nu}_m]$ is as defined in \eqref{def:bordered_HMatrix}.
\end{proposition}
\begin{remark}
As in Remark \ref{rem:distinct_ab} following Proposition \ref{prop:NIBE_bordered_framed}, the above formulas are also valid when two or more of the $\al_j$'s or $\be_j$'s coincide, after applyng l'H\^{o}pital's rule.
\end{remark}

In the paper \cite{Baik-Liu14}, it was shown that for the case $m_1=m_2=0$, so that all Brownian particles begin and end at 0, for any $x\in \R$, 
\eq\label{eq:F2_limit}
\lim_{n\to\infty} \mathbb{P}_0\left(\mathscr{W}_n \le 2\sqrt{n}+2^{-2/3} n^{-1/6}x\right) = F_2(x),
\eeq
where $F_2(x)$ is the Tracy--Widom GUE (TW-GUE) distribution function which describes the location of the largest eigenvalue in the Gaussian unitary ensemble of complex Hermitian random matrices \cite{Tracy-Widom94}. Since the top and bottom curves of the ensemble each converge in the proper scaling limit to the Airy$_2$ process minus a parabola, the result \eqref{eq:F2_limit} has the interpretation that \cite[Equation (19)]{Baik-Liu14}
\[
\mathbb{P}\left(2^{-1/3}\sup_{t\in \R}\left(\mathcal{A}^{(1)}(t) + \mathcal{A}^{(2)}(t) - 2t^2\right) \le x\right)= F_2(x),
\]
where $\mathcal{A}^{(2)}$ and $\mathcal{A}^{(2)}$ are two independent copies of the Airy$_2$ process. In other words the supremum of the sum of two independent Airy$_2$ processes minus a parabola is given by the Tracy--Widom GUE distribution. In the setting of non-intersecting Brownian motions with wanderers we expect a similar result for Airy processes with wanderers. Namely, for fixed $m_1$ and $m_2$ if we take the scalings
\eq\label{eq:NIBB_M_scaling}
M = 2\sqrt{n}+2^{-2/3} n^{-1/6}x, 
\eeq
and
\begin{align}
 \al_j &= -2\sqrt{n} +2^{2/3} n^{1/6} a_j, \quad \be_j = -2\sqrt{n} +2^{2/3} n^{1/6} b_j, \quad j=1,\dots, m_1, \\
 \al_j &= 2\sqrt{n} -2^{2/3} n^{1/6} a_j, \quad \be_j = 2\sqrt{n} -2^{2/3} n^{1/6} b_j, \quad j=m_1+1,\dots, m_1+m_2, 
\end{align}
with all $a_j$'s and $b_j$'s positive, then we expect \eqref{eq:framed_W_dist} to converge to the distribution of the supremum of the sum of to Airy processes with wanderers minus a parabola. That is, we expect
\[
\lim_{n\to\infty}\mathbb{P}_{m_1,m_2}^{\al,\be}(\mathscr{W}_n < M) =\mathbb{P}\left(2^{-1/3}\sup_{t\in \R}\left(\mathcal{A}^{(\vec a_{m_1}, \vec b_{m_1})}(t) + \mathcal{A}^{(\vec a_{m_2}, \vec b_{m_2})}(t) - 2t^2\right) \le x\right),
\]
where $\vec a_{m_1} = (a_1, \dots, a_{m_1})$, $\vec a_{m_2} = (a_{m_1+1}, \dots, a_{m_1+m_2})$, $\vec b_{m_1} = (b_1, \dots, b_{m_1})$, $\vec b_{m_2} = (b_{m_1+1}, \dots, b_{m_1+m_2})$, and the Airy processes with wanderers $\mathcal{A}^{(\vec a_{m_1}, \vec b_{m_1})}$ and $\mathcal{A}^{(\vec a_{m_2}, \vec b_{m_2})}(t) $ are assumed to be independent. Similarly we expect \eqref{eq:bordered_W_dist} to converge to the distribution of 
\[
2^{-1/3}\sup_{t\in \R}\left(\mathcal{A}(t) + \mathcal{A}^{(\vec a_{m}, \vec b_{m})}(t) - 2t^2\right),
\]
where $\mathcal{A}$ is the usual Airy$_2$ process, and is independent from the Airy process with wanderers $\mathcal{A}^{(\vec a_{m}, \vec b_{m})}$.

With an eye towards asymptotic analysis, we can write \eqref{eq:framed_W_dist} and \eqref{eq:bordered_W_dist} in terms of orthogonal polynomials using Corollaries  \ref{prop:bordered_det_gen} and  \ref{eq:semi-framed_det_CDbb}. 
Introduce the systems of monic orthogonal polynomials $\{P_k(x)\}_{k=0}^\infty$, $P_k(x) = x^k+\cdots$, and $\{P_k^{M, \tau}(x)\}_{k=0}^\infty$, $P_k(x) = x^k+\cdots$, via the orthogonality conditions
\eq\label{eq:OPs_def_NIBB}
\int_\R P_j(x)P_k(x)e^{-x^2/2}\dd x = h_k\de_{jk}, \quad  \frac{2\pi}{M}\sum_{x\in \mathcal{D}_{M,\tau}} P_j^{M,\tau}(x)P_k^{M,\tau}(x)e^{-x^2/2} = h_k^{M,\tau}\de_{jk}.
\eeq

Then using Corollary \ref{eq:semi-framed_det_CDbb} in  \eqref{eq:framed_W_dist}  gives  (with $m=m_1+m_2$),
 \begin{multline}\label{eq:framed_NIBB_OPs}
 \mathbb{P}_{m_1,m_2}^{\vec\al,\vec\be}(\mathscr{W}_n < M) = \ds\int_0^{1} \frac{H_{n-m}[\mu^{M,\tau}]}{H_{n-m}[\mu]} \\
\times \frac{\det\left[A^{M,\tau} - \left(\left(\frac{2\pi}{M}\right)^2 \ds\sum_{y_j\in \mathcal{D}_{M,\tau}}\ds\sum_{x_k\in \mathcal{D}_{M,\tau}} K_{n-m-1}^{M,\tau}(x_k, y_j) e^{-x_k^2/2+\ii x_k \al_k - y_j^2/2-\ii y_j \be_j}
\right)_{j,k=1}^{m}\right]}{\det\left[A- \left(\ds\int_\R\ds\int_\R K_{n-m-1}(x_k, y_j) e^{-x_k^2/2+\ii x_k \al_k - y_j^2/2-\ii y_j \be_j}\,\dd x_k\,\dd y_j
\right)_{j,k=1}^{m}\right]}\,\dd \tau,
\end{multline}
where
\[
K_{n-m-1}(x,y) = \sum_{k=0}^{n-m-1} \frac{P_k(x)P_k(y)}{h_k}, \quad K_{n-m-1}^{M,\tau}(x,y) = \sum_{k=0}^{n-m-1} \frac{P_k^{M,\tau}(x)P_k^{M,\tau}(y)}{h_k^{M,\tau}},
\]
and using Corollary \ref{prop:bordered_det_gen} in \eqref{eq:bordered_W_dist}  gives (with $m=m_1+m_2$),
\begin{multline}\label{eq:NIBB_OP_B_gen}
\mathbb{P}_{m_1,m_2}^{\vec 0,\be}(\mathscr{W}_n < M) \ds  \int_0^{1}\frac{H_{n-m}[\mu^{M,\tau}]}{H_{n-m}[\mu]} \\ \times  \frac{\left(\frac{2\pi}{M}\right)^m\ds \sum_{x_1, x_2, \dots, x_m \in \mathcal{D}_{M,\tau} }  \det 
\begin{pmatrix}
P_{n-m}^{M,\tau}(x_1) & \cdots & P_{n-m}^{M,\tau}(x_m) \\
\vdots & \ddots & \vdots \\
P_{n-1}^{M,\tau}(x_1) & \cdots & P_{n-1}^{M,\tau}(x_m)
\end{pmatrix} \ds\prod_{j=1}^{m} e^{-x_j^2/2-\ii \be_{j}}  }{\ds\int_\R \cdots\ds \int_\R \det 
\begin{pmatrix}
P_{n-m}(x_1) & \cdots & P_{n-m}(x_m) \\
\vdots & \ddots & \vdots \\
P_{n-1}(x_1) & \cdots & P_{n-1}(x_m)
\end{pmatrix} \ds\prod_{j=1}^{m} \left(e^{-x_j^2/2-\ii \be_{j}} \dd x_j\right)}\dd\tau.
\end{multline}

It was shown in \cite{Baik-Liu14} that in the scaling \eqref{eq:NIBB_M_scaling},
\[
\lim_{n\to\infty} \frac{H_{n-m}[\mu^{M,\tau}]}{H_{n-m}[\mu]} =F_2(x),
\]
where this limit is independent of $\tau$. To evaluate the rest of \eqref{eq:framed_NIBB_OPs}, \eqref{eq:NIBB_OP_B_gen} as $n\to\infty$, one needs to know the large degree behavior of the orthogonal polynomials $P_k(z)$ and $P_j^{M,\tau}(z)$, preferably throughout the complex plane. The polynomials $\{P_j(x)\}_{j=0}^\infty$ are nothing but the Hermite polynomials, for which asymptotics are well known. For the polynomials $\{P_j^{M,\tau}(z)\}_{j=0}^\infty$, asymptotic formulas may be obtained by the Riemann--Hilbert method \cite{BKMM, Bleher-Liechty11, Bleher-Liechty14}. For this system of orthogonal polynomials with the scaling \eqref{eq:NIBB_M_scaling}, the nonlinear steepest analysis of the corresponding discrete Riemann--Hilbert problem was carried out in \cite{Liechty12} (see also \cite{Liechty-Wang16}). Although asymptotic formulas for $P_n^{M,\tau}(z)$ do not appear directly in that paper, they follow from the analysis which appears there. If one has asymptotic formulas for $P_n^{M,\tau}(z)$ as $n\to\infty$, they could be inserted into \eqref{eq:framed_NIBB_OPs}, \eqref{eq:NIBB_OP_B_gen}. Then, at least for small values of $m$, the sums appearing in the numerators of \eqref{eq:framed_NIBB_OPs}, \eqref{eq:NIBB_OP_B_gen} could be evaluated by first replacing the sum by a complex integral using the residue theorem, and then evaluating the integral using the method of steepest descent. This approach was used in \cite[Section 5]{Liechty-Wang16} to asymptotically calculate a sum very similar to the one in the numerator of \eqref{eq:framed_NIBB_OPs}.

\subsection{Discrete and continuous time random walks}\label{Sec Discrete and continuous time random walks}
In the paper \cite{Baik-Liu14}, the authors also consider the width of the non-intersecting random walks described in the introduction. That is, we consider $n$ independent simple random walks on $\Z$, $X_1(t), X_2(t), \dots, X_n(t)$ which are conditioned to begin at the points $X_j(0)=2x_j$ and end at the points $X_j(2T) =2 y_j$ after a fixed time $2T \in 2\N$, with $x_1<x_2<\dots<x_n$ and $y_1<y_2<\dots<y_n$. The {\it width} of this ensemble of non-intersecting paths as
\[
W_n(2T) = \max_{0\le t \le 2T}\{ X_n(t) - X_1(t)\}.
\]

As discussed in the introduction, the probability that the paths begin at the points $2x_1, \dots, 2x_n$ and end at the points $2y_1, \dots, 2y_n$ at time $2T$ without intersecting for $0\le t\le 2T$ is
\eq\label{NIRM_phi_prob}
\det\left[\phi_{y_k-x_j}\right]_{j,k=1}^n, \quad \phi(\ze) = \ze^{-T}\left(\frac{\ze+1}{2}\right)^{2T},
\eeq
where the extra factor of $1/2$ in $\phi(\ze)$ as compared to \eqref{NIRM_phi} is due to the fact that \eqref{NIRM_phi_prob} gives a probability whereas \eqref{NIRM_phi} is counting paths.

A variation of the above set-up is to consider continuous time random walks. Let $N_1(t)$,  $N_2(t)$ be independent Poisson processes with rate $\lambda = 1/2$, and let $X(t)= N_1(t) - N_2(t)$ be their difference. The probability generating function for the transition probabilities for $N_1(t)$ (and $N_2(t)$) is
\[
p_{N_1}(z; t) = \sum_{k=0}^\infty \mathbb{P}(N_1(t) = k) z^k = \sum_{k=0}^\infty e^{-\frac{t}{2}} \frac{1}{k!}\left(\frac{t}{2}\right)^k z^k = e^{\frac{t}{2}(z-1)}.
\]
The probability generating function for $X(t) = N_1(t)- N_2(t)$ is therefore
\begin{equation}\label{P_T}
	p_{X}(z;t) = p_{N_1}(z;t)p_{N_2}(1/z; t) = e^{\frac{t}{2}(z+1/z-2)}.
\end{equation}
We refer to the process $X(t) = N_1(t) - N_2(t)$ as a continuous time random walk;  it increases or decreases by 1 after exponentially distributed waiting times.
We can consider $n$ continuous time random walks $X_1(t), \dots, X_n(t)$  beginning at locations $x_1<x_2<\cdots<x_n$ at time $t=0$ and conditioned not to intersect for $0\le t \le T$. Then, similar to \eqref{NIRM_phi_prob}, by the Karlin--McGregor/LGV formula, the transition probability for the vector $(X_1(t), \dots, X_n(t))$ to be at the location $(y_1, \dots, y_n)$ at time $t$ is
\begin{equation}\label{NIRM_phi_continuous}
	e^{-nT} \det\left(\phi_{y_k-x_j}\right), \qquad \phi(\ze) =  e^{\frac{T}{2}(\ze+1/\ze)}.
\end{equation}

\begin{center}
 \begin{figure}[h]
\begin{center}
   \scalebox{0.57}{\includegraphics{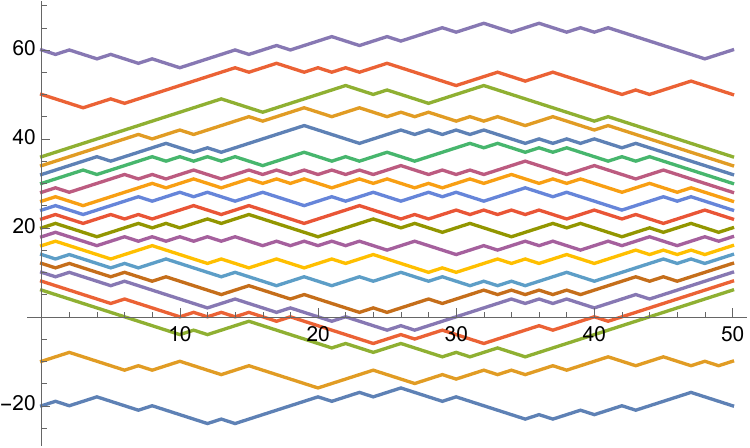}\hspace{1cm}\includegraphics{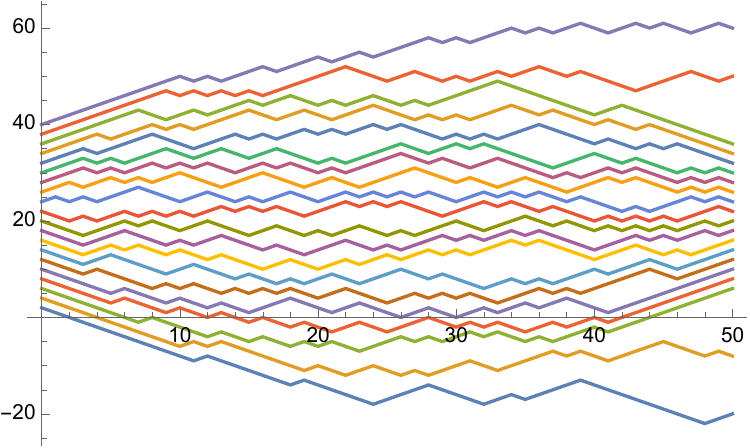}}
\end{center}
        \caption[]{Two simulations of non-intersecting discrete time random walks with wanderers with $n=20$, $2T=50$, and $m_1=m_2=2$. In the figure on the left, the starting and ending points of the walkers  are equispaced except for the lowest and highest pair, and the distribution of $W_n$ is described in \eqref{eq:W_framed_prop_D} in terms of framed Hankel determinants. In the figure on the right, the starting points are all equispaced but the ending points are not and the distribution of $W_n$ is described in \eqref{eq:W_bordered_prop_D} in terms of bordered Hankel determinants.}\label{DT-wanderers}
    \end{figure}
\end{center}

\begin{center}
 \begin{figure}[h]
\begin{center}
   \scalebox{0.22}{\includegraphics{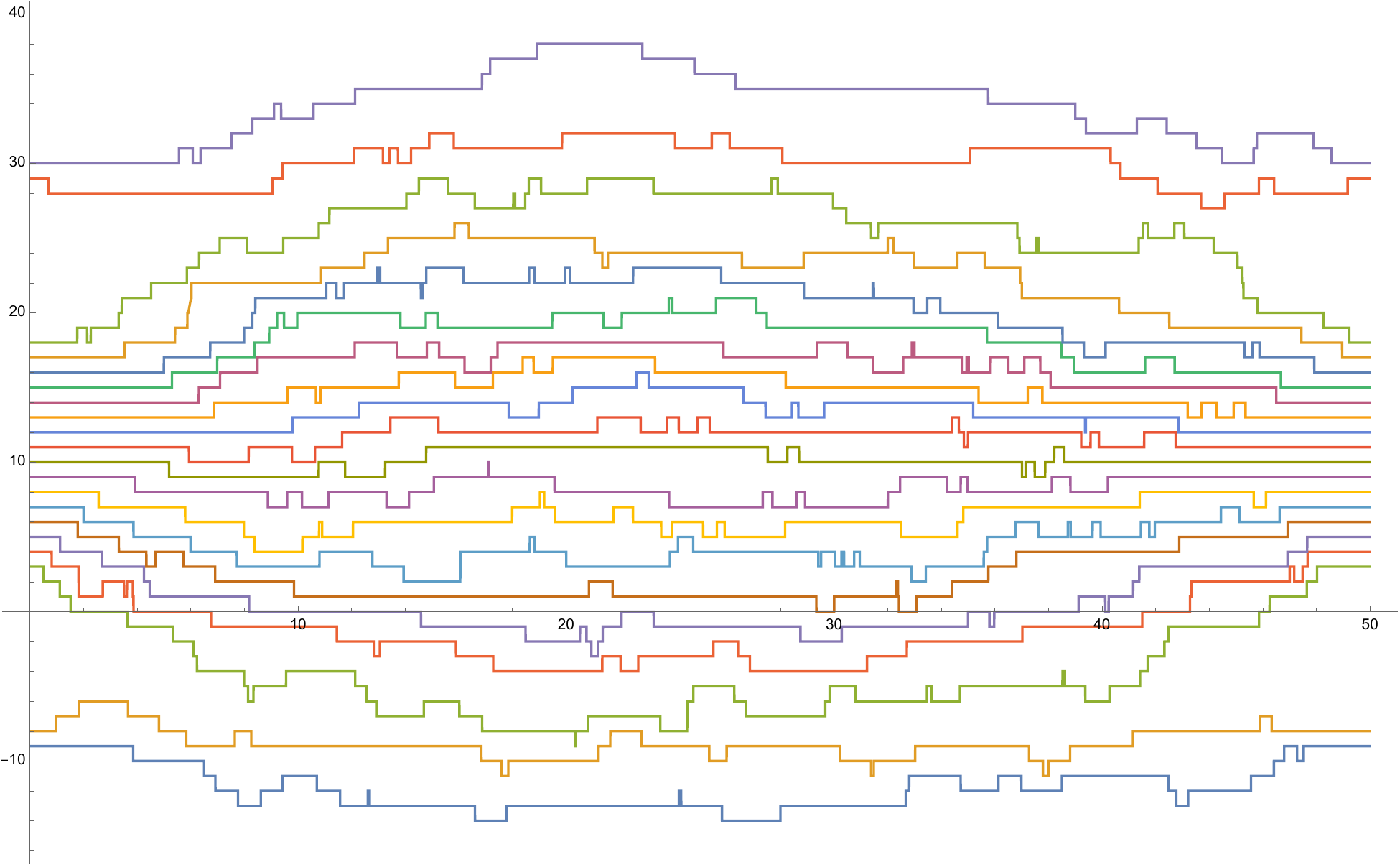}\hspace{1cm}\includegraphics{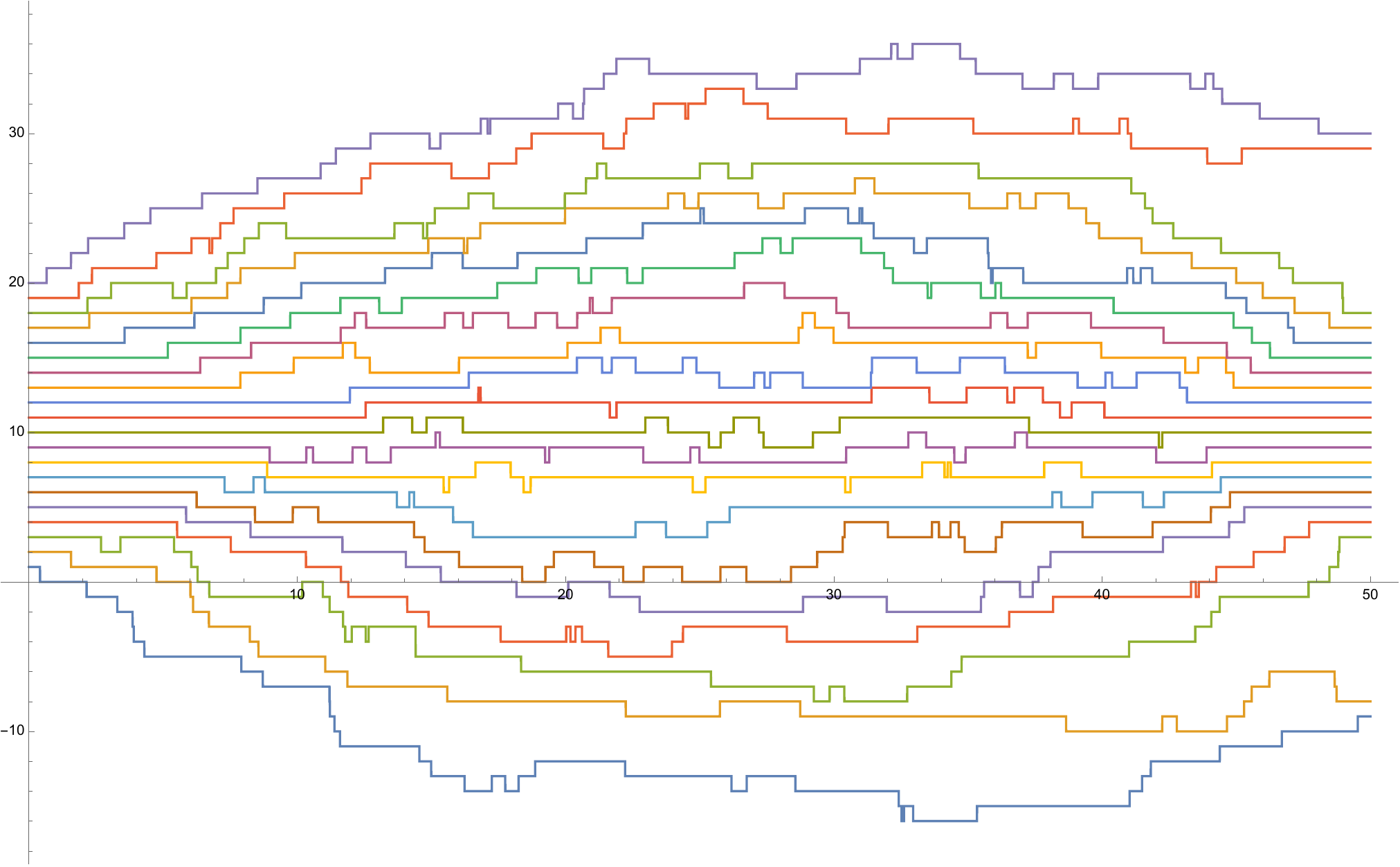}}
\end{center}
        \caption[]{Two simulations of non-intersecting continuous time random walks with wanderers with $n=20$, $T=50$, and $m_1=m_2=2$. In the figure on the left, the starting and ending points of the walkers are equispaced except for the lowest and highest pair, and the distribution of $W_n$ is described in \eqref{eq:W_framed_prop_C} in terms of framed Hankel determinants. In the figure on the right, the starting points are all equispaced but the ending points are not and the distribution of $W_n$ is described in \eqref{eq:W_bordered_prop_C} in terms of bordered Hankel determinants.}\label{CT-wanderers}
    \end{figure}
\end{center}

The width of non-intersecting ensembles of both discrete and continuous random walks were discussed in \cite{Baik-Liu14}, where it was shown that they can be expressed in terms of a discrete analogue of determinants described above. 
For a positive integer $M\in \Z_{>0}$ and a point $s\in \T$, define the discrete set $\mathcal{D}_{M,s}$ as\footnote{Note that the set $\mathcal{D}_{M,s}$ defined here is different from the set $\mathcal{D}_{M,\tau}$ defined in \eqref{def:discrete_set}, though they play similar roles.}
\[
\mathcal{D}_{M,s} = \{ z\in \T  \ : \  z^M = s\} ,
\]
and for a function $f\in L^1(\T)$ define the discrete measure
\eq\label{def:discretization}
f^{M,s}:=\frac{1}{M} \sum_{z\in \mathcal{D}_{M,s}} f(z) \de_z.
\eeq
Note that as $M\to\infty$,  the discrete measure $f^{M,s}$ converges to $\frac{f(z)\dd z}{2\pi \ii z}$.
\begin{proposition}\label{prop:NIRW_width}
    Let $x_1<x_2<\dots<x_n$ and  $y_1<y_2<\cdots < y_n$ be integers, and consider the ensemble of $n$ non-intersecting discrete time random walks beginning at points $X_j(0) = 2x_j$ for $j=1,\dots,n$, conditioned to end at points $X_k(2T) = 2y_k$ for $j=1,\dots, n$. The width of this ensemble of non-intersecting paths is given as
\begin{equation}\label{W_n_CDF_discreteRW}
    \mathbb{P}(W_n(2T) < 2M) = \frac{1}{\det \left(\phi_{y_k-x_j}\right)_{j,k=1}^n}\int_{|s| = 1}\det \left(\phi^{M,s}_{y_k-x_j}\right)_{j,k=1}^n\frac{\dd s}{2\pi i s}, \quad \phi(\ze) = \ze^{-T}\left(\ze+1\right)^{2T}.
\end{equation}
Similarly, in the ensemble of $n$ non-intersecting continuous time random walks beginning at points $X_j(0) = x_j$ for $j=1,\dots,n$ and conditioned to end at points $X_k(T) = y_k$ for $k=1,\dots, n$, the width of the ensemble is given as
\begin{equation}\label{W_n_CDF_contRW}
    \mathbb{P}(W_n(T) < M) = \frac{1}{\det \left(\phi_{y_k-x_j}\right)_{j,k=1}^n}\int_{|s| = 1}\det\left(\phi^{M,s}_{y_k-x_j}\right)_{j,k=1}^n \frac{\dd s}{2\pi i s}, \quad \phi(\ze) = e^{\frac{T}{2}(\ze+1/\ze)}.
\end{equation}

\end{proposition}
\begin{proof}
For the pure Toeplitz case $x_j = y_j = j, j=1,2,\dots, n$, the formulas \eqref{W_n_CDF_discreteRW} and \eqref{W_n_CDF_contRW} are exactly \cite[Proposition 4.2]{Baik-Liu14} and \cite[Proposition 4.1]{Baik-Liu14}, respectively. The proofs there do not depend on the Toeplitz structure of the determinant, and our proof proceeds in exactly the same way.
\end{proof}

Similar to the previous section, we can choose the starting/ending points so that the determinants in \eqref{W_n_CDF_discreteRW} and \eqref{W_n_CDF_contRW} have bordered or framed structure.  In addition to the symbol $\phi$ and its discretization $\phi^{M,s}$, define the symbols $\psi[\be]$ and $\eta[\al]$ as
\eq
\psi[\be](\ze):=\phi(\ze) \ze^{n-m_2-\be}, \quad \eta[\al](\ze):=\phi(\ze) \ze^{-(m_1+1-\al)},
\eeq 
and their discretizations
\[
\psi[\be]^{M,s}:=\sum_{z\in \mathcal{D}_{M,s}} \psi[\be](z)\de_z, \quad \eta[\al]^{M,s}:=\sum_{z\in \mathcal{D}_{M,s}} \eta[\be](z)\de_z.
\]
Also let $A$ and $A^{M,s}$ be the $m\times m$ matrices
\[
A=\begin{pmatrix} \phi_{\be_k-\al_j}\end{pmatrix}_{j,k=1}^m, \quad A^{M,s}=\begin{pmatrix} \phi^{M,s}_{\be_k-\al_j}\end{pmatrix}_{j,k=1}^m.
\]
 We then have the following propositions which follow immediately from Proposition \ref{prop:NIRW_width} and elementary row operations.

\begin{proposition}
Fix non-negative integers $m_1$ and $m_2$ such that $m_1+m_2<n$, and integers $\al_1<\al_2<\dots<\al_{m_1}<m_1<n-m_2<\al_{m_1+1}<\dots<\al_{m_1+m_2}$ and $\be_1<\be_2<\dots<\be_{m_1}<m_1<n-m_2<\be_{m_1+1}<\dots<\be_{m_1+m_2}$. Fix the numbers $x_1, \dots, x_n$ and $y_1,\dots, y_n$ as follows: $x_j =y_j = j$ for $j=m_1+1, \dots, n-m_2$; $x_j=\al_j$ and $y_j=\be_j$ for $j=1,\dots,m_1$ ; and $x_{n-m_2+j}= \al_j$ and $y_{n-m_2+j}j= \be_j$ for $j=m_1+1,\dots, m_1+m_2$.

In the ensemble of discrete time simple random walks with starting points $X_j(0) = 2x_j$ and ending points $X_j(2T) = 2y_j$, the distribution of the width \eqref{W_n_CDF_discreteRW}  has the form (with $m=m_1+m_2$)
\begin{equation}\label{eq:W_framed_prop_D}
    \mathbb{P}(W_n(2T) < 2M) =\ds\int_{|s| = 1} \frac{D_n^F[\phi^{M,s}; \psi[\be_1]^{M,s}, \dots , \psi[\be_m]^{M,s}; \eta[\al_1]^{M,s}, \dots, \eta[\al_m]^{M,s}; A^{M,s}]}{D_n^F[\phi; \psi[\be_1], \dots , \psi[\be_m]; \eta[\al_1], \dots, \eta[\al_m]; A]}\frac{\dd s}{2\pi \ii s},
\end{equation}
 where $\phi$ is as in  \eqref{W_n_CDF_discreteRW} .
 
 Similarly in the ensemble of continuous time random walks with starting points $X_j(0) = x_j$ and ending points $X_j(T) = y_j$, the distribution of the width \eqref{W_n_CDF_contRW} has the form (with $m=m_1+m_2$)
\begin{equation}\label{eq:W_framed_prop_C}
    \mathbb{P}(W_n(T) < M) =\ds\int_{|s| = 1} \frac{D_n^F[\phi^{M,s}; \psi[\be_1]^{M,s}, \dots , \psi[\be_m]^{M,s}; \eta[\al_1]^{M,s}, \dots, \eta[\al_m]^{M,s}; A^{M,s}]}{D_n^F[\phi; \psi[\be_1], \dots , \psi[\be_m]; \eta[\al_1], \dots, \eta[\al_m]; A]}\frac{\dd s}{2\pi \ii s},
\end{equation}
 where $\phi$ is as in  \eqref{W_n_CDF_contRW}.
 \end{proposition}
 If the starting points are all consecutive, the framed determinants above reduce to bordered ones.
  \begin{proposition}
Fix non-negative integers $m_1$ and $m_2$ such that $m_1+m_2<n$, and integers $\be_1<\be_2<\dots<\be_{m_1}<m_1<n-m_2<\be_{m_1+1}<\dots<\be_{m_1+m_2}$.
Fix the numbers $x_1, \dots, x_n$ and $y_1,\dots, y_n$ as follows: $x_j =j$ for $j=1,\dots, n$; and $y_j = j$ for $j=m_1+1, \dots, n-m_2$;  $y_j=\be_j$ for $j=1,\dots,m_1$ ;  and $y_{n-m_2+j}= \be_j$ for $j=m_1+1,\dots, m_1+m_2$.

In the ensemble of discrete time simple random walks with starting points $X_j(0) = 2x_j$ and ending points $X_j(2T) = 2y_j$, the distribution of the width \eqref{W_n_CDF_discreteRW}  has the form (with $m=m_1+m_2$)
\begin{equation}\label{eq:W_bordered_prop_D}
    \mathbb{P}(W_n(2T) < 2M) =\ds\int_{|s| = 1} \frac{D_n^B[\phi^{M,s}; \psi[\be_1]^{M,s}, \dots , \psi[\be_m]^{M,s}]}{D_n^B[\phi; \psi[\be_1], \dots , \psi[\be_m] ]}\frac{\dd s}{2\pi \ii s},
\end{equation}
 where $\phi$ is as in  \eqref{W_n_CDF_discreteRW} .
 
 In the ensemble of continuous time simple random walks with starting points $X_j(0) = x_j$ and ending points $X_j(T) = y_j$, the distribution of the width \eqref{W_n_CDF_contRW}  has the form (with $m=m_1+m_2$)
\begin{equation}\label{eq:W_bordered_prop_C}
    \mathbb{P}(W_n(T) < M) =\ds\int_{|s| = 1} \frac{D_n^B[\phi^{M,s}; \psi[\be_1]^{M,s}, \dots , \psi[\be_m]^{M,s}]}{D_n^B[\phi; \psi[\be_1], \dots , \psi[\be_m] ]}\frac{\dd s}{2\pi \ii s},
\end{equation}
 where $\phi$ is as in  \eqref{W_n_CDF_contRW}.
\end{proposition}
\noindent Two simulations, related respectively to \eqref{eq:W_framed_prop_D} and \eqref{eq:W_bordered_prop_D}, for the discrete random walks with $m_1=m_2=2$ are displayed in Figure \ref{DT-wanderers}. Similarly for continuous time random walks, two simulations related respectively to \eqref{eq:W_framed_prop_C} and \eqref{eq:W_bordered_prop_C} are shown in Figure \ref{CT-wanderers}.

As in the previous subsection on the width of non-intersecting Brownian bridges, we expect \eqref{eq:W_framed_prop_D}, \eqref{eq:W_framed_prop_C}, \eqref{eq:W_bordered_prop_D} and \eqref{eq:W_bordered_prop_C} to have scaling limits related to the supremum of the sum of two Airy processes with wanderers minus a parabola. We do not describe that limit in detail, but we do give the orthogonal polynomials which may be relevant for asymptotic analysis.

For $\phi$ given by either \eqref{W_n_CDF_discreteRW} or \eqref{W_n_CDF_contRW}, introduce the bi-orthogonal system of polynomials  $Q_k(z)=\ka_kz^k+\cdots$,  and $\widehat{Q}_n(z)=\widehat{\ka}_kz^k+\cdots$ via the bi-orthonormality conditions
	\begin{equation}\label{eq:rw_contOPs}
	\int_{\T} Q_j(\ze)\widehat{Q}_k(\ze^{-1})\phi(\ze)\frac{\dd z}{2\pi \ii \ze}= \de_{jk}, \qquad j,k \in \Z_{\ge 0},
	\end{equation}
	as well as bi-orthogonal system of polynomials  $Q_k^{M,s}(z)=\ka_k^{M,s}z^k+\cdots$,  and $\widehat{Q}^{M,s}_k(z)=\widehat{\ka}_k^{M,s}z^k+\cdots$ 
	defined via
\begin{equation}\label{eq:rw_discOPs}
	\frac{1}{M}\sum_{z^M=s} Q_j^{M,s}(z)\widehat{Q}_k^{M,s}(z^{-1})\phi(z)= \de_{jk}, \qquad j,k \in \Z_{\ge 0}. 
	\end{equation}
Then applying Corollary \ref{prop:bordered_det_gen} gives that the right-hand-side of \eqref{eq:W_bordered_prop_D} and \eqref{eq:W_bordered_prop_C} is,
\begin{multline}\label{W_n_CDF_RW_bordered}
 \ds\int_{|s| = 1} \sqrt{\frac{D_{n-m}[\phi^{M,s}]D_{n-m+1}[\phi^{M,s}]}{D_{n-m}[\phi]D_{n-m+1}[\phi]}} \\
\times    \frac{\ds\sum_{\{z_1^M, z_2^M, \dots, z_m^M = s\}}\begin{pmatrix}
Q^{M,s}_{n-m}(z_1) & \cdots & Q^{M,s}_{n-m}(z_m) \\
\vdots & \ddots & \vdots \\
Q_{n-1}^{M,s}(z_1) & \cdots & Q^{M,s}_{n-1}(z_m)
\end{pmatrix} \prod_{j=1}^{m} z_j^{1-m_2-\be_j}\phi_j(z_j)\dd z_j}{\ds\int\cdots\int_{\T^m}\begin{pmatrix}
Q_{n-m}(z_1) & \cdots & Q_{n-m}(z_m) \\
\vdots & \ddots & \vdots \\
Q_{n-1}(z_1) & \cdots & Q_{n-1}(z_m)
\end{pmatrix} \prod_{j=1}^{m} z^{1-m_2-\be_j}_j  \frac{\phi_j(z_j)\dd z_j}{2\pi \ii z_j}}\frac{\dd s}{2\pi \ii s}.
\end{multline}
Applying Corollary \ref{eq:semi-framed_det_CDbb} gives that the right-hand-side of \eqref{eq:W_framed_prop_D} and \eqref{eq:W_framed_prop_C} is,
\begin{multline}\label{W_n_CDF_RW_framed}
 \ds\int_{|s| = 1} \frac{D_{n-m}[\phi^{M,s}]}{D_{n-m}[\phi]} \frac{\det\left[A^{M,s} - \left(\sum_{z,w \in \mathcal{D}_{M,s}}\mathscr{K}^{M,s}_{n-m-1}(z^{-1},w^{-1}) \phi(z)\phi(w) z^{n-m_2-\be_k} w^{-(m_1+1-\al_j)} \right)_{j,k=1}^m\right]}{
\det\left[A - \left(\int_\T\int_\T\mathscr{K}_{n-m-1}(z^{-1},w^{-1}) \frac{\phi(z)\phi(w) z^{n-m_2-\be_k} w^{-(m_1+1-\al_j)}}{(2\pi \ii)^2 zw} \right)_{j,k=1}^m\right]
},
\end{multline}
where
\[
\mathscr{K}_{n-m-1}(z,w) := \sum_{j=0}^{n-m-1} Q_{j}(w)\widehat{Q}_{j}(z), \qquad \mathscr{K}^{M,s}_{n-m-1}(z,e) := \sum_{j=0}^{n-m-1} Q_{j}^{M,s}(w)\widehat{Q}_{j}^{M,s}(z).
\]

For $\phi$ given by \eqref{NIRM_phi_continuous}, a steepest descent analysis of the Riemann--Hilbert problem for the orthogonal system \eqref{eq:rw_contOPs} was carried out in \cite{BDJ}; for the discrete version of the same symbol, a steepest descent analysis of the discrete Riemann--Hilbert problem for the orthogonal system \eqref{eq:rw_contOPs} was carried out in \cite{Baik-Jenkins13}. To evaluate \eqref{W_n_CDF_RW_bordered}, one could first extract asymptotics of $Q_n(z)$ from the analysis of \cite{BDJ}, \cite{Baik-Jenkins13} and insert those asymptotics into \eqref{W_n_CDF_RW_bordered}. We do not do this in the general case in this work, but we give partial results in some special cases. For the case $m=2$, the asymptotic result for the denominator of \eqref{W_n_CDF_RW_bordered} is presented in Theorem \ref{thm asymp bordered}. Also, for $m=1$ the asymptotic result for the denominator of \eqref{W_n_CDF_RW_framed} is given in Theorem \ref{thm asymp framed}. The asymptotic analysis leading to these results is presented in Section \ref{Sec-Asymptotics}.

\section{The six-vertex model with DWBC}\label{sec 6-vertex}

The six-vertex model is stated on a square lattice of size $n\times n$, with arrows drawn on edges satisfying {\it ice rule}: at each vertex, exactly two adjacent arrows point inward and two point outward. There are six possible configurations of arrows at each vertex, and we label them as shown in Figure \ref{arrows}.

\begin{center}
 \begin{figure}[h]
\begin{center}
   \scalebox{0.32}{\includegraphics{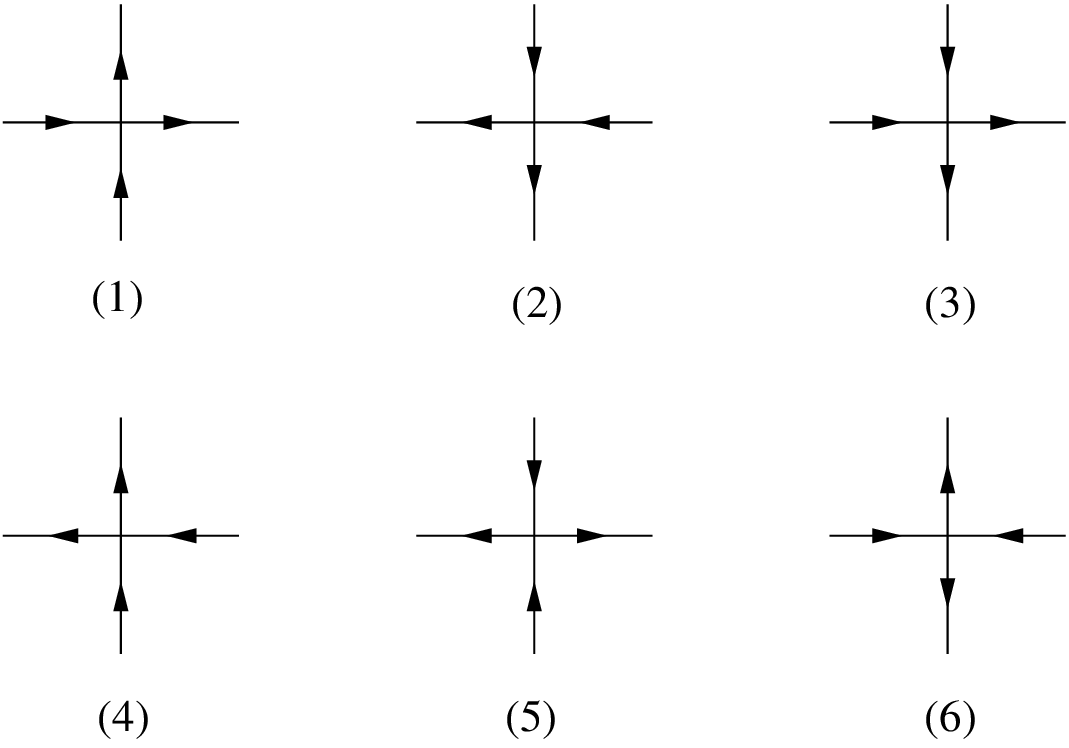}}
\end{center}
        \caption[The six arrow configurations allowed at a vertex]{The six arrow configurations allowed at a vertex.}\label{arrows}
    \end{figure}
\end{center}
	
We can give a weight $w_i>0$ to each vertex type $i = 1,2,\dots, 6$, and define the weight of an arrow configuration $\sigma$ as
\[
w_n(\sg) = \prod_{i=1}^6 w_i^{N_i(\sg)},
\]
where $N_i(\sg)$ is the number of vertices of type $i$ in the configuration $\sg$. The Gibbs measure is then given by the normalized weight of a configuration
\eq\label{eq:6v_mu}
\mu_n(\sg) = \frac{w_n(\sg)}{Z_n}, \quad Z_n = \sum_{\sg} w(\sg).
\eeq

The quantity $Z_n$ is the partition function for the model, and the sum defining it is over all allowable configurations $\sg$ on the $n\times n$ lattice, which may be restricted by some boundary conditions as well as by the ice-rule. For certain boundary conditions known as {\it Domain Wall Boundary Conditions} (DWBC), the partition function may be expressed as a Hankel determinant.  In DWBC all arrows on the left and right boundary point outward, and all arrows on the top and bottom boundaries point inward. See Figure \ref{fig:5by5} for an example of an arrow configuration with DWBC on the $5\times 5$ lattice. The determinantal formula for the partition function of the DWBC six-vertex model is known as the {\it Izergin--Korepin formula}.

For fixed boundary conditions like DWBC, the Gibbs measure of the six-vertex model is invariant under reversal of arrows (see e.g., \cite[Chapter 5, Section 2]{Bleher-Liechty14}), so we may give the weight $a>0$ to the vertices of type 1 and 2; the weight $b>0$ to the vertices of type 3 and 4; and the weight $c>0$ to the vertices of type 5 and 6. Then the phase diagram for the model is as follows.
\begin{itemize}
\item {\bf Ferroelectric phase}. If $a>b+c$ (resp. $b>a+c$) then vertices of Type 1 and 2 (resp. Type 3 and 4) are dominant. In the ferroelectric phase it is typical to parametrize the weights as
\begin{equation}\label{pf4}
a\equiv a(t, \ga)=\sinh(t-\ga), \quad
b\equiv b(t, \ga)=\sinh(\ga+t), \quad
c\equiv c(\ga)=\sinh(2\ga), \quad
0<|\ga|<t.
\end{equation}

\item {\bf Disordered phase}. If the quantities $a, b$, and $c$ form a triangle, then there is no dominant weight. 
In the disordered phase it is typical to parametrize the weights as
\begin{equation}\label{pf6}
a\equiv a(t, \ga)=\sin(\ga-t), \quad
b\equiv b(t, \ga)=\sin(\ga+t), \quad
c\equiv c(\ga)=\sin(2\ga), \quad
|t|<\ga<\pi/2.
\end{equation}
\item {\bf Antiferroelectric phase}. If $c>a+b$, then vertices of Type 5 and 6 are dominant. In the antiferroelectric phase it is typical to parametrize the weights as
\begin{equation}\label{pf5}
a\equiv a(t, \ga)=\sinh(\ga-t), \quad
b\equiv b(t, \ga)=\sinh(\ga+t), \quad
c\equiv c(\ga)=\sinh(2\ga), \quad
|t|<\ga.
\end{equation}
\end{itemize}

\begin{figure}
\begin{center}
\begin{tikzpicture}[scale=1]
\draw[step=1cm, black, very thick] (-2,-2) grid (2,2);
\draw[-{Stealth[length=3mm, width=3mm]},very thick] (2,2) -- (2.5,2);
\draw[-{Stealth[length=3mm, width=3mm]},very thick](2,1) -- (2.5,1);
\draw[-{Stealth[length=3mm, width=3mm]},very thick] (2,0) -- (2.5,0);
\draw[-{Stealth[length=3mm, width=3mm]},very thick] (2,-1) -- (2.5,-1);
\draw[-{Stealth[length=3mm, width=3mm]},very thick] (2,-2) -- (2.5,-2);
\draw[-{Stealth[length=3mm, width=3mm]},very thick] (-2,2) -- (-2.5,2);
\draw[-{Stealth[length=3mm, width=3mm]},very thick] (-2,1) -- (-2.5,1);
\draw[-{Stealth[length=3mm, width=3mm]},very thick] (-2,0) -- (-2.5,0);
\draw[-{Stealth[length=3mm, width=3mm]},very thick] (-2,-1) -- (-2.5,-1);
\draw[-{Stealth[length=3mm, width=3mm]},very thick] (-2,-2) -- (-2.5,-2);
\draw[-{Stealth[length=3mm, width=3mm]},very thick] (2,2.5) -- (2,2.2);
\draw[very thick] (2,2.2) -- (2,2);
\draw[-{Stealth[length=3mm, width=3mm]},very thick] (1,2.5) -- (1,2.2);
\draw[very thick] (1,2.2) -- (1,2);
\draw[-{Stealth[length=3mm, width=3mm]},very thick] (0,2.5) -- (0,2.2);
\draw[very thick] (0,2.2) -- (0,2);
\draw[-{Stealth[length=3mm, width=3mm]},very thick] (-2,2.5) -- (-2,2.2);
\draw[very thick] (-2,2.2) -- (-2,2);
\draw[-{Stealth[length=3mm, width=3mm]},very thick] (-1,2.5) -- (-1,2.2);
\draw[very thick] (-1,2.2) -- (-1,2);

\draw[-{Stealth[length=3mm, width=3mm]},very thick] (2,-2.5) -- (2,-2.2);
\draw[very thick] (2,-2.2) -- (2,-2);
\draw[-{Stealth[length=3mm, width=3mm]},very thick] (1,-2.5) -- (1,-2.2);
\draw[very thick] (1,-2.2) -- (1,-2);
\draw[-{Stealth[length=3mm, width=3mm]},very thick] (0,-2.5) -- (0,-2.2);
\draw[very thick] (0,-2.2) -- (0,-2);
\draw[-{Stealth[length=3mm, width=3mm]},very thick] (-2,-2.5) -- (-2,-2.2);
\draw[very thick] (-2,-2.2) -- (-2,-2);
\draw[-{Stealth[length=3mm, width=3mm]},very thick] (-1,-2.5) -- (-1,-2.2);
\draw[very thick] (-1,-2.2) -- (-1,-2);

\draw[{Stealth[length=3mm, width=3mm]}-,very thick] (-1.6,2) -- (-1.5,2) ;
\draw[-{Stealth[length=3mm, width=3mm]},very thick] (-0.5,2) -- (-0.4,2) ;
\draw[-{Stealth[length=3mm, width=3mm]},very thick] (0.5,2) -- (0.6,2) ;
\draw[-{Stealth[length=3mm, width=3mm]},very thick] (1.5,2) -- (1.6,2) ;

\draw[-{Stealth[length=3mm, width=3mm]},very thick] (-2,1.5) -- (-2,1.4) ;
\draw[{Stealth[length=3mm, width=3mm]}-,very thick] (-1,1.6) -- (-1,1.5) ;
\draw[-{Stealth[length=3mm, width=3mm]},very thick] (0,1.5) -- (0,1.4) ;
\draw[-{Stealth[length=3mm, width=3mm]},very thick] (1,1.5) -- (1,1.4) ;
\draw[-{Stealth[length=3mm, width=3mm]},very thick] (2,1.5) -- (2,1.4) ;

\draw[-{Stealth[length=3mm, width=3mm]},very thick] (-2,0.5) -- (-2,0.4) ;
\draw[-{Stealth[length=3mm, width=3mm]},very thick] (-2,-0.5) -- (-2,-0.4) ;
\draw[-{Stealth[length=3mm, width=3mm]},very thick] (-2,-1.5) -- (-2,-1.4) ;

\draw[{Stealth[length=3mm, width=3mm]}-,very thick] (-1.6,-2) -- (-1.5,-2) ;
\draw[-{Stealth[length=3mm, width=3mm]},very thick] (-0.5,-2) -- (-0.4,-2) ;
\draw[-{Stealth[length=3mm, width=3mm]},very thick] (0.5,-2) -- (0.6,-2) ;
\draw[-{Stealth[length=3mm, width=3mm]},very thick] (1.5,-2) -- (1.6,-2) ;

\draw[-{Stealth[length=3mm, width=3mm]},very thick] (2,0.5) -- (2,0.4) ;
\draw[-{Stealth[length=3mm, width=3mm]},very thick] (2,-0.5) -- (2,-0.4) ;
\draw[-{Stealth[length=3mm, width=3mm]},very thick] (2,-1.5) -- (2,-1.4) ;

\draw[{Stealth[length=3mm, width=3mm]}-,very thick] (-1,-1.6) -- (-1,-1.5) ;
\draw[-{Stealth[length=3mm, width=3mm]},very thick] (0,-1.5) -- (0,-1.4) ;
\draw[-{Stealth[length=3mm, width=3mm]},very thick] (1,-1.5) -- (1,-1.4) ;

\draw[{Stealth[length=3mm, width=3mm]}-,very thick] (-1.6,-1) -- (-1.5,-1) ;
\draw[{Stealth[length=3mm, width=3mm]}-,very thick] (-1.4,0) -- (-1.5,0) ;
\draw[{Stealth[length=3mm, width=3mm]}-,very thick] (-1.6,1) -- (-1.5,1) ;

\draw[{Stealth[length=3mm, width=3mm]}-,very thick] (1.6,-1) -- (1.5,-1) ;
\draw[{Stealth[length=3mm, width=3mm]}-,very thick] (1.4,0) -- (1.5,0) ;
\draw[{Stealth[length=3mm, width=3mm]}-,very thick] (1.6,1) -- (1.5,1) ;

\draw[-{Stealth[length=3mm, width=3mm]},very thick] (1,-0.5) -- (1,-0.6) ;
\draw[-{Stealth[length=3mm, width=3mm]},very thick] (1,0.5) -- (1,0.4) ;

\draw[-{Stealth[length=3mm, width=3mm]},very thick] (0.5,1) -- (0.6,1) ;
\draw[-{Stealth[length=3mm, width=3mm]},very thick] (0.5,0) -- (0.4,0) ;
\draw[-{Stealth[length=3mm, width=3mm]},very thick] (0.5,-1) -- (0.4,-1) ;

\draw[-{Stealth[length=3mm, width=3mm]},very thick] (-0.5,1) -- (-0.6,1) ;

\draw[-{Stealth[length=3mm, width=3mm]},very thick] (-0.5,-1) -- (-0.6,-1) ;
\draw[-{Stealth[length=3mm, width=3mm]},very thick] (-1,-0.5) -- (-1,-0.6) ;
\draw[-{Stealth[length=3mm, width=3mm]},very thick] (0,-0.5) -- (0,-0.4) ;

\draw[-{Stealth[length=3mm, width=3mm]},very thick] (0,0.5) -- (0,0.6) ;
\draw[-{Stealth[length=3mm, width=3mm]},very thick] (-1,0.5) -- (-1,0.6) ;

\draw[-{Stealth[length=3mm, width=3mm]},very thick] (-0.5,0) -- (-0.6,0) ;

\draw[red,thick,fill=red] (-1,2) circle (1mm);
\draw[red,thick,fill=red] (-2,0) circle (1mm);
\draw[red,thick,fill=red] (-1,-2) circle (1mm);
\draw[red,thick,fill=red] (2,0) circle (1mm);
\draw[red,thick,fill=red] (0,1) circle (1mm);
\draw[red,thick,fill=red] (1,-1) circle (1mm);

\draw[blue,thick,fill=blue] (-1,0) circle (1mm);

\end{tikzpicture}
\end{center}
\caption{A six-vertex configuration on the $5\times 5$ lattice with DWBC. Type 5 vertices are marked in red, and type-6 vertices are marked in blue. The domain wall boundary conditions and the ice rule ensure that there is exactly one type-5 vertex along each boundary, and no more than $k$ type 5 vertices in the $k$th row from the right or left and the $k$th column from the top or bottom. Vertices of types 5 and 6 alternate along each row and column.}\label{fig:5by5}
\end{figure}
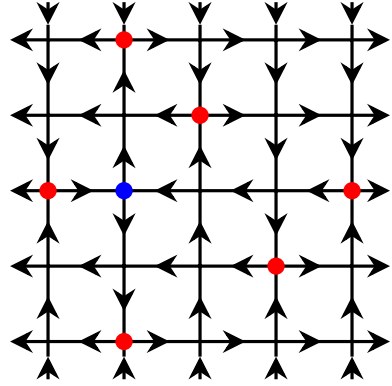

In order to describe the Izergin--Korepin formula in its most general form, it is convenient to introduce vertex weights which depend on the location of the vertex in the $n\times n$ lattice as well as the vertex type. To do so, fix $n=1,2,\dots$, and $2n+1$ complex parameters $\chi_1,\dots,\chi_n; \psi_1,\dots,\psi_n; \gamma$. For a configuration $\sigma$ of the six-vertex model in an $n\times n$ square with DWBC and one of the parametrizations \eqref{pf4} -- \eqref{pf5}, we define the weight of the vertex at $(x,y)$ to be
 $$
  \omega(x,y;\sigma)=
  \begin{cases}
    a(\psi_y-\chi_x, \ga), &\quad\text{if }\sigma(x,y)\text{ is of Type 1 or 2,}\\
    b(\psi_y-\chi_x, \ga), &\quad\text{if }\sigma(x,y)\text{ is of Type 3 or 4,}\\
    c(\ga), &\quad\text{if }\sigma(x,y)\text{ is of Type 5 or 6.}
  \end{cases}
 $$
 The inhomogeneous partition function is then defined as
 \begin{equation}
Z_n(\chi_1,\dots,\chi_n; \psi_1,\dots,\psi_n; \gamma)=\sum_{\sigma} \prod_{x=1}^n \prod_{y=1}^n \omega(x,y;\sigma),
 \end{equation}
where the sum is over all configurations satisfying DWBC.

The Izergin--Korepin formula states that for any of the parametrizations \eqref{pf4} -- \eqref{pf5},
\begin{multline}
\label{eq_IK_det}
Z_n(\chi_1,\dots,\chi_n;\, \psi_1,\dots,\psi_n;\, \gamma)\\= \frac{\prod\limits_{i,j=1}^n \bigl(a(\psi_j-\chi_i, \ga) b(\psi_j-\chi_i, \ga)\bigr)}{\prod\limits_{i<j} \bigl(b(\chi_i-\chi_j,0)b(\psi_i-\psi_j,0)\bigr)} \det \left[ \frac{c(\ga)}{a(\psi_j-\chi_i, \ga) b(\psi_j-\chi_i, \ga)} \right]_{i,j=1}^n.
\end{multline}
The partition function for the homogeneous model is then obtained by taking all parameters $\chi_1, \dots , \chi_n$ to be the same, and $\psi_1, \dots, \psi_n$ to be the same. For convenience we can take $\chi_1, \dots , \chi_n=0$ and $\psi_1, \dots, \psi_n=t$. The determinant vanishes when two of the $\psi$'s or $\chi$'s coincide, as does the denominator of the prefactor, so one must be careful to evaluate the limit as $\chi_1, \dots , \chi_n\to 0$ and $\psi_1, \dots, \psi_n\to t$. To do so it is convenient to define the function
\[
\varphi(t) = \frac{c(\ga)}{a(t, \ga) b(t, \ga)},
\]
so that the Izergin--Korepin determinant is 
\[
Z_n(\chi_1,\dots,\chi_n;\, \psi_1,\dots,\psi_n;\, \gamma)= \frac{\prod\limits_{i,j=1}^n \bigl(a(\psi_j-\chi_i, \ga) b(\psi_j-\chi_i, \ga)\bigr)}{\prod\limits_{i<j} \bigl(b(\chi_i-\chi_j,0)b(\psi_i-\psi_j,0)\bigr)} \det \left[ \varphi(\psi_j-\chi_i) \right]_{i,j=1}^n.
\]
We can take $\chi_i\to 0$ in sequential order row-by-row. By expanding $\varphi(\psi_j - \chi_i)$ as a Taylor series around $\psi_j$ and using row operations we obtain
\[
Z_n(0^n;\, \psi_1,\dots,\psi_n;\, \gamma)= \frac{\prod\limits_{j=1}^n \bigl(a(\psi_j, \ga) b(\psi_j, \ga)\bigr)^n}{\left(\prod_{j=0}^{n-1} j!\right) \prod\limits_{i<j} \bigl(b(\psi_i-\psi_j,0)} 
\det \left[ \varphi^{(i-1)}(\psi_j) \right]_{i,j=1}^n.
\]
We can then take $\psi_j\to t$ in sequential order column-by-column, again expanding $\varphi^{(i-1)}(\psi_j)$ in a Taylor series around $\psi_j = t$ and using column operations to obtain
\[
Z_n(0^n;\, t^n;\, \gamma)= \frac{ \bigl(a(t, \ga) b(t, \ga)\bigr)^{n^2}}{\left(\prod_{j=0}^{n-1} j!\right)^2 } 
\det \left[ \varphi^{(i+j-2)}(t) \right]_{i,j=1}^n.
\]
To recognize this as a Hankel determinant, one must recognize the function $\varphi(t)$ as the Laplace transform of some measure m:
\[
\varphi(t) = \int_{\R} e^{tx} {\rm dm}(x),
\]
so that 
\eq\label{eq:laplace}
\varphi^{(i+j-2)}(t) = \int_{\R} x^{i+j-2}e^{tx} {\rm dm}(x) = \mu_{i+j-2}, \quad \dd\mu(x)= e^{tx} {\rm dm}(x),
\eeq
and the homogeneous partition function becomes
\[
Z_n = Z_n(0^n;\, t^n;\, \gamma) = \frac{ \bigl(a(t, \ga) b(t, \ga)\bigr)^{n^2}}{\left(\prod_{j=0}^{n-1} j!\right)^2 } 
H_n[\mu].
\]
This characterization as a Hankel determinant was exploited in a series of works to compute the large-$n$ asymptotics of the partition function in various regions of the phase diagram \cite{Zinn_Justin00, Bleher-Fokin06, Bleher-Liechty09, Bleher-Liechty10, Bleher-Bothner12, Bleher-Liechty14}.

Variations on the above formula can be obtained by taking a partially homogeneous limit of \eqref{eq_IK_det}. Such partially inhomogeneous partition functions in fact function as generating functions for the location of the $c$-type vertices near the boundary of the $n\times n$ domain. For example, if we take $\chi_1, \dots, \chi_n \to 0$ and $\psi_1, \dots, \psi_{n-1} \to t$ as before, but leave $\psi_n$ free, we obtain
\begin{multline}\label{eq:Zn_limit_D}
Z_n(0^n;\, t^{n-1}, \psi_n;\, \gamma) = \frac{ \bigl(a(t, \ga) b(t, \ga)\bigr)^{n^2-n}\bigl(a(\psi_n, \ga) b(\psi_n, \ga)\bigr)^{n}}{(n-1)!b(\psi_n -t, 0)^{n-1}\left(\prod_{j=0}^{n-2} j!\right)^2 }  \\
\times \det \begin{pmatrix} \f(t, \ga) &  \f'(t, \ga) & \dots  & \f^{(n-2)}(t,\ga) & \f(\psi_n, \ga) \\ 
\f'(t, \ga) &  \f{''}(t,\ga) & \dots  &  \f^{(n-1)}(t, \ga) & \f'(\psi_n, \ga)
\\ \vdots & \vdots & \ddots  & \vdots & \vdots \\
\f^{(n-2)}(t, \ga) & \f^{(n-1)}(t, \ga) & \dots  & \f^{(2n-4)}(t, \ga) & \f^{(n-2)}(\psi_n, \ga) \\
\f^{(n-1)}(t, \ga) &  \f^{(n)}(t, \ga) & \dots &  \f^{(2n-3)}(t, \ga) & \f^{(n-1)}(\psi_n, \ga) \end{pmatrix}.
\end{multline}
If we write $\psi_n = t + \xi$ and use the Laplace transform \eqref{eq:laplace}, the above formula becomes
\begin{multline}
Z_n(0^n;\, t^{n-1}, \psi_n;\, \gamma) = \frac{ \bigl(a(t, \ga) b(t, \ga)\bigr)^{n^2-n}\bigl(a(\psi_n, \ga) b(\psi_n, \ga)\bigr)^{n}}{(n-1)!b(\psi_n -t, 0)^{n-1}\left(\prod_{j=0}^{n-2} j!\right)^2 }  \\
\times \det \begin{pmatrix} \mu_0 &  \mu_1 & \dots  & \mu_{n-2} & \nu_0 \\ 
\mu_1 & \mu_2 & \dots  &  \mu_n & \nu_1
\\ \vdots & \vdots & \ddots  & \vdots & \vdots \\
\mu_{n-2} & \mu_{n-1}  & \dots  & \mu_{2n-4} & \nu_{n-2} \\
\mu_{n-1} &  \mu_{n} & \dots &  \mu_{2n-3} & \nu_{n-1} \end{pmatrix}.
\end{multline}
where $\mu$ is as in \eqref{eq:laplace}, and $\nu$ is defined as
\[
\dd\nu(x) = e^{(t+\xi)x}{\rm dm}(x) = e^{\xi x} \dd\mu(x).
\]
This is the bordered Hankel determinant $H_n^B[\mu; \nu]$ and Corollary \ref{prop:bordered_det_gen} gives the formula 
\begin{multline}\label{eq:6v_bordered}
Z_n(0^n;\, t^{n-1}, \psi_n;\, \gamma) = \frac{ \bigl(a(t, \ga) b(t, \ga)\bigr)^{n^2-n}\bigl(a(t+\xi, \ga) b(t+\xi, \ga)\bigr)^{n}}{(n-1)!b(\xi, 0)^{n-1}\left(\prod_{j=0}^{n-2} j!\right)^2 } \\
\times H_{n-1}[\mu]\int_{\R} P_{n-1}(x) e^{\xi x} \dd\mu(x),  \\
\end{multline}
where the system of monic polynomials $\{P_{k}(x)\}_{k=0}^\infty$ is defined by the orthogonality condition
\eq\label{eq:6v-OPs}
\int_{\R} P_j(x) P_k(x) \dd\mu(x) = h_k \de_{jk}.
\eeq
This partially inhomogeneous partition function tracks the location of the $c$-type vertex in the last column of the six-vertex state, which we may denote $\Lambda_n$. This fact has been used throughout the past 20 years to investigate the distribution of $\Lambda_n$ \cite{Bogoliubov-Pronko-Zvonarev02} and subsequently to investigate other correlation functions in the six-vertex model \cite{Colomo-Pronko08, Colomo-Pronko10, Colomo-Pronko-ZinnJustin10}. The formula \eqref{eq:6v_bordered} in particular was analyzed asymptotically in the recent work \cite{Gorin-Liechty23} in each of the phase regions to prove Gaussian fluctuations for $\Lambda_n$ as $n\to\infty$ in the disordered and anti-ferroelectric phases, and convergence of $\Lambda_n$ to a geometric random variable in the ferroelectric phase\footnote{In fact \cite{Gorin-Liechty23} did much more. There it was shown that in the disordered and anti-ferroelectric phase regions the  type-5 vertices near the boundary converge in a scaling limit to the GUE corners process from random matrix theory \cite{Johansson-Nordenstam06} and in the ferroelectric phase region they converge to the stochastic six-vertex model \cite{borodin2016stochastic}.}. The orthogonal polynomials \eqref{eq:6v-OPs} depend on the phase region, and their asymptotics can be ascertained from the Riemann--Hilbert analysis presented in \cite{Bleher-Fokin06, Bleher-Liechty10, Bleher-Liechty14} (see also \cite[Appendix A.2]{Gorin-Liechty23} for the orthogonal polynomials in the ferroelectric phase region).

One can also jointly track the location of the $c$-type vertex in the last row and last column by taking $\chi_1, \dots, \chi_{n-1} \to 0$, $\psi_1, \dots, \psi_{n-1} \to t$,  $\psi_n = t+\xi_2$, and $\chi_n = -\xi_1$. Similar to \eqref{eq:Zn_limit_D}, one finds
\begin{multline}\label{eq:Zn_limit_D_sf}
Z_n(0^{n-1}, -\xi_1;\, t^{n-1}, t+\xi_2;\, \gamma) = \\
\frac{ \bigl[a(t, \ga) b(t, \ga)\bigr]^{(n-1)^2}\bigl[a(t+\xi_1, \ga) b(t+\xi_1, \ga)a(t+\xi_2, \ga) b(t+\xi_2, \ga)\bigr]^{n-1}a(t+\xi_1+\xi_2, \ga)b(t+\xi_1+\xi_2, \ga)}{\left[b(\xi_1, 0)b(\xi_2, 0)\right]^{n-1}\left(\prod_{j=0}^{n-2} j!\right)^2 }  \\
\times \det \begin{pmatrix} \mu_0 &  \mu_1 & \dots  & \mu_{n-2} & \nu_0 \\ 
\mu_1 &  \mu_2 & \dots  & \mu_{n-1} & \nu_1
\\ \vdots & \vdots & \ddots  & \vdots & \vdots \\
\mu_{n-2} & \mu_{n-1} & \dots  & \mu_{2n-4} & \nu_{n-2} \\
\eta_0 &  \eta_1 & \dots &  \eta_{n-2} & a \end{pmatrix},
\end{multline}
where $\mu$ is as in \eqref{eq:laplace} and 
\[
\dd\eta(x) = e^{(t+\xi_1)x}{\rm dm}(x) = e^{\xi_1 x} \dd\mu(x), \quad \dd\nu(x) = e^{(t+\xi_2)x}{\rm dm}(x) = e^{\xi_2 x} \dd\mu(x),
\]
and the corner entry $a$ is
\[
a=\f(t+\xi_1+\xi_2) = \int_\R e^{(t+\xi_1+\xi_2)x} {\rm dm}(x).
\]
This is a framed Hankel determinant, and we have
\begin{multline}
Z_n(0^{n-1}, -\xi_1;\, t^{n-1}, t+\xi_2;\, \gamma) = \\
\frac{ \bigl[a(t, \ga) b(t, \ga)\bigr]^{(n-1)^2}\bigl[a(t+\xi_1, \ga) b(t+\xi_1, \ga)a(t+\xi_2, \ga) b(t+\xi_2, \ga)\bigr]^{n-1}a(t+\xi_1+\xi_2, \ga)b(t+\xi_1+\xi_2, \ga)}{\left[b(\xi_1, 0)b(\xi_2, 0)\right]^{n-1}\left(\prod_{j=0}^{n-2} j!\right)^2 }  \\
\times H_n^F[\mu;\nu, \eta;a].
\end{multline}
Corollary \ref{eq:semi-framed_det_CDbb} then gives
\begin{multline}\label{eq:6v-doubly_refined}
Z_n(0^{n-1}, -\xi_1;\, t^{n-1}, t+\xi_2;\, \gamma) = \\
\frac{ \bigl[a(t, \ga) b(t, \ga)\bigr]^{(n-1)^2}\bigl[a(t+\xi_1, \ga) b(t+\xi_1, \ga)a(t+\xi_2, \ga) b(t+\xi_2, \ga)\bigr]^{n-1}a(t+\xi_1+\xi_2, \ga)b(t+\xi_1+\xi_2, \ga)}{\left[b(\xi_1, 0)b(\xi_2, 0)\right]^{n-1}\left(\prod_{j=0}^{n-2} j!\right)^2 }  \\
\times H_{n-1}[\mu]\left(a-\int_{\R}\int_{\R} e^{\xi_1 x+\xi_2 y}K_{n-2}(x,y)\dd\mu(x)\dd\mu(y)\right) \\
= \frac{ \bigl[a(t, \ga) b(t, \ga)\bigr]^{(n-1)^2}\bigl[a(t+\xi_1, \ga) b(t+\xi_1, \ga)a(t+\xi_2, \ga) b(t+\xi_2, \ga)\bigr]^{n-1}a(t+\xi_1+\xi_2, \ga)b(t+\xi_1+\xi_2, \ga)}{\left[b(\xi_1, 0)b(\xi_2, 0)\right]^{n-1}\left(\prod_{j=0}^{n-2} j!\right)^2 }  \\
\times H_{n-1}[\mu]\left(\int_\R e^{(\xi_1+\xi_2)x} d\mu(x)-\int_{\R}\int_{\R} e^{\xi_1 x+\xi_2 y}\sum_{k=0}^{n-2} \frac{P_k(x)P_k(y)}{h_k}\dd\mu(x)\dd\mu(y)\right).
\end{multline}
It is an open and interesting question to evaluate \eqref{eq:6v-doubly_refined} as $n\to\infty$ in order to obtain information on the joint distribution of the location of the $c$-type vertices in the first row and column.

\section{Proofs of Propositions \ref{eq:bordered_det}, \ref{Prop: multi-bordered Toeplitz and BOPUC}, \ref{eq:semi-framed_det_CD}, and \ref{eq:semi-framed_det_CDaa}}\label{sec:proofs of OP props}

\subsection{Proof of Propositions \ref{eq:bordered_det} and \ref{Prop: multi-bordered Toeplitz and BOPUC}}\label{sec 5.1} We prove Proposition \ref{eq:bordered_det} in detail. The proof of Proposition \ref{Prop: multi-bordered Toeplitz and BOPUC} is similar. 

Before proceeding let us recall the  \textit{Dodgson Condensation Identity}\footnote{Also known as the \textit{Desnanot–Jacobi} identity or the \textit{Sylvester determinant} identity.} , which we occasionally abbreviate as DCI (see \cite{Abeles,Fulmek-Kleber,Bressoud, GW} and references therein). Let $\boldsymbol{\mathscr{M}}$ be an $n \times n$ matrix. By \[  \mathscr{M} \left\lbrace \begin{matrix} j_1& j_2& \cdots & j_{\ell} \\  k_1& k_2& \cdots & k_{\ell} \end{matrix} \right\rbrace, \]
we mean the determinant of the $(n-\ell)\times(n-\ell)$ matrix obtained from $\boldsymbol{\mathscr{M}}$ by removing the rows $j_i$ and the columns $k_i$, $1\leq i \leq \ell$. Although the order of writing the row and column indices is immaterial for this definition, in this work we prefer to respect the order of indices, for example we prefer to write \[ \mathscr{M} \left\lbrace \begin{matrix} 3& 5 \\  1& 4 \end{matrix} \right\rbrace, \] and not \[ \mathscr{M} \left\lbrace \begin{matrix} 5 & 3 \\  1& 4 \end{matrix} \right\rbrace \qquad \mbox{or} \qquad \mathscr{M} \left\lbrace \begin{matrix} 3& 5 \\  4& 1 \end{matrix} \right\rbrace \qquad \mbox{or} \qquad \mathscr{M} \left\lbrace \begin{matrix} 5 & 3 \\  4& 1 \end{matrix} \right\rbrace, \]
although all of these are the same determinant. Let $j_1 < j_2$ and $k_1 < k_2$. The Dodgson Condensation identity reads
\begin{equation}\label{DODGSON}
	\mathscr{M} \cdot \mathscr{M}\left\lbrace \begin{matrix} j_1 & j_2 \\  k_1& k_2 \end{matrix} \right\rbrace = \mathscr{M}\left\lbrace \begin{matrix} j_1  \\  k_1 \end{matrix} \right\rbrace \cdot \mathscr{M}\left\lbrace \begin{matrix} j_2  \\  k_2 \end{matrix} \right\rbrace - \mathscr{M}\left\lbrace \begin{matrix} j_1  \\  k_2 \end{matrix} \right\rbrace \cdot \mathscr{M}\left\lbrace \begin{matrix} j_2  \\  k_1 \end{matrix} \right\rbrace.
\end{equation}

Now we prove Proposition \ref{eq:bordered_det}. Let $n\in \Z_{\geq 2} $ be arbitrary. The case $m=n$ is obvious by row operations. For $m<n$, we prove this Proposition by induction on $m\in \N$, the number of borders. For $m=1$, the result obviously holds due to \eqref{Pn det rep}. To prepare for proving the inductive step, we introduce the notations
	\begin{equation}\label{M1 and M2}
		\mathscr{M}_1 \equiv {H}_n[\mu; z_1, \dots, z_m], \qandq         \mathscr{M}_2 \equiv {H}_n[\mu; z_1, \dots, z_{m-1}],
	\end{equation}
	for the $n\times n$ multi-bordered Hankel determinants respectively with $m$ and $m-1$ borders of monomials. Recalling \eqref{DODGSON} we use the following two Dodgson Condensation identities:
	\begin{equation}\label{Dodgson1}
		\mathscr{M}_1 \cdot \mathscr{M}_1\left\lbrace \begin{matrix} 0 & n-1 \\  n-2 & n-1 \end{matrix} \right\rbrace = \mathscr{M}_1\left\lbrace \begin{matrix} 0  \\  n-2 \end{matrix} \right\rbrace \cdot \mathscr{M}_1\left\lbrace \begin{matrix} n-1  \\  n-1 \end{matrix} \right\rbrace - \mathscr{M}_1\left\lbrace \begin{matrix} 0  \\  n-1 \end{matrix} \right\rbrace \cdot \mathscr{M}_1\left\lbrace \begin{matrix} n-1  \\  n-2 \end{matrix} \right\rbrace,
	\end{equation}
	and
	\begin{equation}\label{Dodgson2}       	\mathscr{M}_2 \cdot \mathscr{M}_2\left\lbrace \begin{matrix} 0 & n-1 \\  n-m & n-1 \end{matrix} \right\rbrace = \mathscr{M}_2\left\lbrace \begin{matrix} 0  \\  n-m \end{matrix} \right\rbrace \cdot \mathscr{M}_2\left\lbrace \begin{matrix} n-1  \\  n-1 \end{matrix} \right\rbrace - \mathscr{M}_2\left\lbrace \begin{matrix} 0  \\  n-1 \end{matrix} \right\rbrace \cdot \mathscr{M}_2\left\lbrace \begin{matrix} n-1  \\  n-m \end{matrix} \right\rbrace.
	\end{equation}
	We study the determinants on the right hand side of these DCIs one by one. Notice that, by the induction hypothesis
	\begin{equation}\label{M1 n-1 n-1}
		\mathscr{M}_1\left\lbrace \begin{matrix} n-1  \\  n-1 \end{matrix} \right\rbrace = H_{n-1}[\mu;z_1,\cdots,z_{m-1}] = H_{n-m}[\mu] \underset{(m-1)\times(m-1)}{\det} 
		\begin{pmatrix}
			P_{n-m}(z_1) & \cdots & P_{n-m}(z_{m-2}) & P_{n-m}(z_{m-1}) \\
			\vdots & \ddots & \vdots & \vdots \\
			P_{n-2}(z_1) & \cdots & P_{n-2}(z_{m-2}) & P_{n-2}(z_{m-1})
		\end{pmatrix},
	\end{equation}
	and
	\begin{equation}\label{M1 n-1 n-2}
		\mathscr{M}_1\left\lbrace \begin{matrix} n-1  \\  n-2 \end{matrix} \right\rbrace = H_{n-1}[\mu;z_1,\cdots,z_{m-2},z_m] = H_{n-m}[\mu] \underset{(m-1)\times(m-1)}{\det} 
		\begin{pmatrix}
			P_{n-m}(z_1) & \cdots & P_{n-m}(z_{m-2}) & P_{n-m}(z_{m}) \\
			\vdots & \ddots & \vdots & \vdots \\
			P_{n-2}(z_1) & \cdots & P_{n-2}(z_{m-2}) & P_{n-2}(z_{m})
		\end{pmatrix}.
	\end{equation}
	The remaining two determinants on the right hand side of \eqref{Dodgson1} are not directly expressible in terms of $H_{n-1}$, since due to the removal of the first row, their $11$-entries are $\mu_1$ as opposed to $\mu_0$. Let us now turn our attention to \eqref{Dodgson2}. We can already observe that 
	\begin{align}
		\mathscr{M}_2\left\lbrace \begin{matrix} 0 & n-1 \\  n-m & n-1 \end{matrix} \right\rbrace & = \mathscr{M}_1\left\lbrace \begin{matrix} 0 & n-1 \\  n-2 & n-1 \end{matrix} \right\rbrace, \label{M20n-1,n-mn-1} \\
		\mathscr{M}_2\left\lbrace \begin{matrix} 0  \\  n-m  \end{matrix}  \right\rbrace \equiv \mathscr{M}_2\left\lbrace \begin{matrix} 0  \\  n-m  \end{matrix}  \right\rbrace [z_1,\cdots,z_{m-1}] & = \mathscr{M}_1\left\lbrace \begin{matrix} 0  \\   n-1 \end{matrix} \label{M10n-1}\right\rbrace, \\
		\mathscr{M}_2\left\lbrace \begin{matrix} 0  \\  n-m  \end{matrix} \right\rbrace [z_1,\cdots,z_{m-2},z_m] & = \mathscr{M}_1\left\lbrace \begin{matrix} 0  \\   n-2 \end{matrix} \label{M10n-2}\right\rbrace, \\
		\mathscr{M}_2\left\lbrace \begin{matrix} n-1  \\  n-m  \end{matrix} \right\rbrace \equiv \mathscr{M}_2\left\lbrace \begin{matrix} n-1  \\  n-m  \end{matrix} \right\rbrace[z_1,\cdots,z_{m-1}] & = \mathscr{M}_1\left\lbrace \begin{matrix} n-1  \\   n-1 \end{matrix} \right\rbrace, \qquad \mbox{and} \label{M2n-1n-ma} \\ \mathscr{M}_2\left\lbrace \begin{matrix} n-1  \\   n-m \end{matrix} \right\rbrace[z_1,\cdots,z_{m-2},z_m] & = \mathscr{M}_1\left\lbrace \begin{matrix} n-1  \\   n-2 \end{matrix} \right\rbrace. \label{M2n-1n-mb}
	\end{align}
	We can now express the two determinants $\mathscr{M}_1\left\lbrace \begin{matrix} 0  \\   n-1 \end{matrix} \right\rbrace$ and $\mathscr{M}_1\left\lbrace \begin{matrix} 0  \\   n-2 \end{matrix} \right\rbrace$ in terms of elements in the Dodgson Condensation identity \eqref{Dodgson2}. In fact, from \eqref{Dodgson2} and  we have
	\begin{equation}\label{M10n-1}
		\begin{split} \mathscr{M}_1\left\lbrace \begin{matrix} 0  \\   n-1 \end{matrix} \right\rbrace[z_1,\cdots,z_{m-1}] =
			\frac{1}{{\mathscr{M}_2\left\lbrace \begin{matrix} n-1  \\   n-1 \end{matrix} \right\rbrace[z_1,\cdots,z_{m-2}]}} & \left(
			\mathscr{M}_2[z_1,\cdots,z_{m-1}]\mathscr{M}_2\left\lbrace \begin{matrix} 0 & n-1 \\  n-m & n-1 \end{matrix} \right\rbrace[z_1,\cdots,z_{m-2}] \right.
			\\ & 
			\left.
			+ \ \mathscr{M}_2\left\lbrace \begin{matrix} 0  \\   n-1 \end{matrix} \right\rbrace[z_1,\cdots,z_{m-2}]\mathscr{M}_2\left\lbrace \begin{matrix} n-1  \\   n-m \end{matrix} \right\rbrace[z_1,\cdots,z_{m-1}] \right).
		\end{split}
	\end{equation}
	In view of \eqref{M10n-2}, for $\mathscr{M}_1\left\lbrace \begin{matrix} 0  \\   n-2\end{matrix} \right\rbrace$ we have the same expression given on the right hand side of the equation above, except that we only need to replace $z_{m-1}$ by $z_m$. Indeed,
	\begin{equation}\label{M10n-2B}
		\begin{split} \mathscr{M}_1\left\lbrace \begin{matrix} 0  \\   n-2 \end{matrix} \right\rbrace[z_1,\cdots,z_{m-2},z_m] =
			\frac{1}{{\mathscr{M}_2\left\lbrace \begin{matrix} n-1  \\   n-1 \end{matrix} \right\rbrace[z_1,\cdots,z_{m-2}]}} & \left(
			\mathscr{M}_2[z_1,\cdots,z_{m-2},z_m]\mathscr{M}_2\left\lbrace \begin{matrix} 0 & n-1 \\  n-m & n-1 \end{matrix} \right\rbrace[z_1,\cdots,z_{m-2}] \right.
			\\ & \hspace{-1cm}
			\left.
			+ \ \mathscr{M}_2\left\lbrace \begin{matrix} 0  \\   n-1 \end{matrix} \right\rbrace[z_1,\cdots,z_{m-2}]\mathscr{M}_2\left\lbrace \begin{matrix} n-1  \\   n-m \end{matrix} \right\rbrace[z_1,\cdots,z_{m-2},z_m] \right).
		\end{split}
	\end{equation}
	Using \eqref{M2n-1n-ma} and \eqref{M2n-1n-mb}, we can write these DCIs as
	\begin{equation}\label{M10n-1A}
		\begin{split} \mathscr{M}_1\left\lbrace \begin{matrix} 0  \\   n-1 \end{matrix} \right\rbrace =
			\frac{1}{{\mathscr{M}_2\left\lbrace \begin{matrix} n-1  \\   n-1 \end{matrix} \right\rbrace}} & \left(
			\mathscr{M}_2 \cdot \mathscr{M}_2\left\lbrace \begin{matrix} 0 & n-1 \\  n-m & n-1 \end{matrix} \right\rbrace 
			+ \mathscr{M}_2\left\lbrace \begin{matrix} 0  \\   n-1 \end{matrix} \right\rbrace \cdot \mathscr{M}_1\left\lbrace \begin{matrix} n-1  \\   n-1 \end{matrix} \right\rbrace \right),
		\end{split}
	\end{equation} and
	\begin{equation}\label{M10n-2A}
		\mathscr{M}_1\left\lbrace \begin{matrix} 0  \\   n-2 \end{matrix} \right\rbrace =
		\frac{1}{{\mathscr{M}_2\left\lbrace \begin{matrix} n-1  \\   n-1 \end{matrix} \right\rbrace}}  \left(
		\mathscr{M}_2[z_1,\cdots,z_{m-2},z_m]\mathscr{M}_2\left\lbrace \begin{matrix} 0 & n-1 \\  n-m & n-1 \end{matrix} \right\rbrace 
		+ \mathscr{M}_2\left\lbrace \begin{matrix} 0  \\   n-1 \end{matrix} \right\rbrace\mathscr{M}_1\left\lbrace \begin{matrix} n-1  \\   n-2 \end{matrix} \right\rbrace \right).
	\end{equation}
	Solving for $\mathscr{M}_1$ from \eqref{Dodgson1} yields
	\begin{equation}
		\mathscr{M}_1  =    \frac{\mathscr{M}_1\left\lbrace \begin{matrix} 0  \\  n-2 \end{matrix} \right\rbrace \cdot \mathscr{M}_1\left\lbrace \begin{matrix} n-1  \\  n-1 \end{matrix} \right\rbrace - \mathscr{M}_1\left\lbrace \begin{matrix} 0  \\  n-1 \end{matrix} \right\rbrace \cdot \mathscr{M}_1\left\lbrace \begin{matrix} n-1  \\  n-2 \end{matrix} \right\rbrace }{\mathscr{M}_1\left\lbrace \begin{matrix} 0 & n-1 \\  n-2 & n-1 \end{matrix} \right\rbrace}.
	\end{equation}
	Now using \eqref{M20n-1,n-mn-1}, \eqref{M10n-1A} and \eqref{M10n-2A} in the above expression we obtain
	\begin{equation}\label{M1 in terms of stuff with m-1 borders}
		\mathscr{M}_1  = \frac{\mathscr{M}_2[z_1,\cdots,z_{m-2},z_m] \cdot \mathscr{M}_1\left\lbrace \begin{matrix} n-1  \\  n-1 \end{matrix} \right\rbrace - \mathscr{M}_2[z_1,\cdots ,z_{m-2},z_{m-1}] \cdot \mathscr{M}_1\left\lbrace \begin{matrix} n-1  \\  n-2 \end{matrix} \right\rbrace }{\mathscr{M}_2\left\lbrace \begin{matrix}  n-1 \\   n-1 \end{matrix} \right\rbrace}
	\end{equation}
	after straightforward simplifications.  Now we notice that all objects on the right hand side of \eqref{M1 in terms of stuff with m-1 borders} either have $m-1$ or $m-2$ borders, and thus for each of them we can use the induction hypothesis, recalling \eqref{M1 and M2} .  To this end we already have equations \eqref{M1 n-1 n-1} and \eqref{M1 n-1 n-2} for $\mathscr{M}_1\left\lbrace \begin{matrix} n-1  \\  n-1  \end{matrix} \right\rbrace$ and $\mathscr{M}_1\left\lbrace \begin{matrix} n-1  \\  n-2  \end{matrix} \right\rbrace$, while for the other terms we can write
	\begin{equation}
		\mathscr{M}_2\left\lbrace \begin{matrix} n-1  \\  n-1  \end{matrix} \right\rbrace  = H_{n-m+1}[\mu] \det 
		\begin{pmatrix}
			P_{n-m+1}(z_1) & \cdots & P_{n-m+1}(z_{m-2}) \\
			\vdots & \ddots & \vdots \\
			P_{n-2}(z_1) & \cdots & P_{n-2}(z_{m-2})
		\end{pmatrix},
	\end{equation}
	\begin{equation}
		\mathscr{M}_2[z_1,\cdots ,z_{m-2},z_{m-1}]  = H_{n-m+1}[\mu] \det 
		\begin{pmatrix}
			P_{n-m+1}(z_1) & \cdots & P_{n-m+1}(z_{m-2}) & P_{n-m+1}(z_{m-1}) \\
			\vdots & \ddots & \vdots & \vdots \\
			P_{n-1}(z_1) & \cdots & P_{n-1}(z_{m-2}) & P_{n-1}(z_{m-1})
		\end{pmatrix},
	\end{equation}
	and
	\begin{equation}
		\mathscr{M}_2[z_1,\cdots ,z_{m-2},z_{m}] = H_{n-m+1}[\mu] \det 
		\begin{pmatrix}
			P_{n-m+1}(z_1) & \cdots & P_{n-m+1}(z_{m-2}) & P_{n-m+1}(z_{m}) \\
			\vdots & \ddots & \vdots & \vdots \\
			P_{n-1}(z_1) & \cdots & P_{n-1}(z_{m-2}) & P_{n-1}(z_{m})
		\end{pmatrix}.
	\end{equation}
	Plugging these, together with \eqref{M1 n-1 n-1} and \eqref{M1 n-1 n-2}, into \eqref{M1 in terms of stuff with m-1 borders}, we observe that the right hand side of \eqref{M1 in terms of stuff with m-1 borders} is nothing but 
	\begin{equation}\label{Dodgson3}
		H_{n-m}[\mu]\left(  \frac{ \mathscr{D}\left\lbrace \begin{matrix} 0  \\  m-2 \end{matrix} \right\rbrace \cdot \mathscr{D}\left\lbrace \begin{matrix} m-1  \\  m-1 \end{matrix} \right\rbrace - \mathscr{D}\left\lbrace \begin{matrix} 0  \\  m-1 \end{matrix} \right\rbrace \cdot \mathscr{D}\left\lbrace \begin{matrix} m-1  \\  m-2 \end{matrix} \right\rbrace}{\mathscr{D}\left\lbrace \begin{matrix} 0 & m-1 \\  m-2 & m-1 \end{matrix} \right\rbrace} \right),
	\end{equation}
	with 
	\begin{equation}
		\mathscr{D} := \det 
		\begin{pmatrix}
			P_{n-m}(z_1) & \cdots & P_{n-m}(z_{m-1}) & P_{n-m}(z_{m}) \\
			\vdots & \ddots & \vdots & \vdots \\
			P_{n-1}(z_1) & \cdots & P_{n-1}(z_{m-1}) & P_{n-1}(z_{m})
		\end{pmatrix}.
	\end{equation}
	Now, due to the following Dodgson Condensation identity
	\begin{equation}
		\mathscr{D} \cdot  \mathscr{D}\left\lbrace \begin{matrix} 0 & m-1 \\  m-2 & m-1 \end{matrix} \right\rbrace =  \mathscr{D}\left\lbrace \begin{matrix} 0  \\  m-2 \end{matrix} \right\rbrace \cdot \mathscr{D}\left\lbrace \begin{matrix} m-1  \\  m-1 \end{matrix} \right\rbrace - \mathscr{D}\left\lbrace \begin{matrix} 0  \\  m-1 \end{matrix} \right\rbrace \cdot \mathscr{D}\left\lbrace \begin{matrix} m-1  \\  m-2 \end{matrix} \right\rbrace,
	\end{equation}
	we finally arrive at the desired equality $\mathscr{M}_1 = H_{n-m}[\mu] \mathscr{D}$.

\subsection{Proof of Propositions \ref{eq:semi-framed_det_CD} and \ref{eq:semi-framed_det_CDaa}}\label{Sec 5.2} 
 We prove Proposition \ref{eq:semi-framed_det_CD} in detail. The proof of Proposition \ref{eq:semi-framed_det_CDaa} is similar.

The proof of the Proposition \ref{eq:semi-framed_det_CD} relies on the LU decomposition of the Hankel matrix. Write the monic orthogonal polynomial \eqref{Hankel OPs conditions} as
	\[
	P_n(x) = \sum_{k=0}^{n} p_{n,k}x^k,
	\]
	and introduce the triangular matrix
	\begin{equation}
		\boldsymbol{L}_{n} := \begin{pmatrix}
			\frac{p_{0,0}}{\sqrt{h_0}} & 0 & \cdots & 0 \\
			\frac{p_{1,0}}{\sqrt{h_1}} & \frac{p_{1,1}}{\sqrt{h_1}} & \cdots & 0 \\
			\vdots & \vdots & \ddots & \vdots \\
			\frac{p_{n,0}}{\sqrt{h_n}} & \frac{p_{n,1}}{\sqrt{h_n}} & \cdots & \frac{p_{n,n}}{\sqrt{h_n}}
		\end{pmatrix}.
	\end{equation}
	\begin{lemma}\cite{GW}
		The LU decomposition of $\boldsymbol{H}_{n+1}[\mu]$ is given by 
		\begin{equation}\label{LDU H}
			\boldsymbol{H}_{n+1}[\mu] = \left[\boldsymbol{L}_{n}\right]^{-1} 	 \left[\boldsymbol{L}_{n}^T\right]^{-1}.
		\end{equation} 
	\end{lemma}
	\begin{proof}[Proof of Lemma]
		We have
		\begin{equation}
			\begin{split}
				\delta_{jk} & = \int_{\R} \frac{P_{j}(x)}{\sqrt{h_j}} \frac{P_{k}(x)}{\sqrt{h_k}} \dd\mu(x)  = \sum_{m=0}^{j} \sum_{\ell=0}^{k} \frac{p_{j,m}p_{k,\ell}}{\sqrt{h_jh_k}} \int_{\R} x^{\ell+m}\dd\mu(x)  \\ &
				= \sum_{m=0}^{j} \sum_{\ell=0}^{k} \frac{p_{j,m}p_{k,\ell}}{\sqrt{h_jh_k}} \mu_{\ell+m} = \sum_{m=0}^{j} \sum_{\ell=0}^{k} 	\left( \boldsymbol{L}_{n} \right)_{j,m} \left( \boldsymbol{L}_{n} \right)_{k,\ell} \left( \boldsymbol{H}_{n+1}[\mu] \right)_{\ell,m}  \\ & = \sum_{m=0}^{j} \sum_{\ell=0}^{k}  \left( \boldsymbol{L}_{n} \right)_{j,k} \left( \boldsymbol{H}_{n+1}[\mu] \right)_{j,m}  \left( \boldsymbol{L}_{n}^T \right)_{m,k} =  \left( \boldsymbol{L}_{n} \boldsymbol{H}_{n+1}[\mu] \boldsymbol{L}_{n}^T \right)_{j,k},
			\end{split}
		\end{equation} 
		which is equivalent to \eqref{LDU H}.
	\end{proof}

		Let
		\[
		{\bf Z}_n(x_1, \dotsm, x_m) = \begin{pmatrix} 1 & 1 & \dots & 1 \\ x_1 & x_2 & \dots & x_m \\ \vdots & \vdots & \ddots & \vdots \\ x_1^n& x_2^n & \cdots &  x_m^n \end{pmatrix}, \qquad {\bf P}_n(x_1, x_2, \dots, x_m) =  \begin{pmatrix} \frac{P_0(x_1)}{\sqrt{h_0}} & \frac{P_0(x_2)}{\sqrt{h_0}} & \cdots & \frac{P_0(x_m)}{\sqrt{h_0}}\\ \frac{P_1(x_1)}{\sqrt{h_1}} & \frac{P_1(x_2)}{\sqrt{h_1}} & \cdots & \frac{P_1(x_m)}{\sqrt{h_1}}\\ \vdots & \vdots & \ddots & \vdots \\ \frac{P_n(x_1)}{\sqrt{h_n}} &  \frac{P_n(x_2)}{\sqrt{h_n}} &\cdots &  \frac{P_n(x_m)}{\sqrt{h_n}}  \end{pmatrix},
		\]
		and
\[
{\bf F}_n(\mu; x_1, \dots, x_m; y_1, \dots, y_m; A) = \begin{pmatrix} \mu_0 & \mu_1 & \cdots & \mu_{n-m-1} & 1 & \cdots & 1 \\
	\mu_1 & \mu_2 & \cdots & \mu_{n-m} & x_1 & \cdots & x_m \\
	\mu_2 & \mu_3 & \cdots & \mu_{n-m+1} & x_1^2 & \cdots & x_m^2 \\
	\vdots & \vdots & \ddots & \vdots & \vdots & \ddots & \vdots\\
	\mu_{n-m-1} & \mu_{n-m} & \cdots & \mu_{2(n-m-1)} & x_1^{n-m-1} & \cdots & x_m^{n-m-1} \\
	1 & y_1 & \cdots & y_1^{n-m-1} & a_{11} & \dots & a_{1m} \\
	\vdots & \vdots & \ddots & \vdots & \vdots & \ddots & \vdots\\
	1 & y_m & \cdots & y_m^{n-m-1} & a_{m1} & \dots & a_{mm}
\end{pmatrix}.
\]

		We have
		\begin{equation}\label{312}
			\begin{split}
				&\begin{pmatrix}
					\boldsymbol{L}_{n-m-1} & \boldsymbol{0}_{(n-m)\times m} \\
					\boldsymbol{0}^T_{(n-m)\times m } & {\bf I}_m
				\end{pmatrix}
				{\bf F}_n(\mu; x_1, \dots, x_m; y_1, \dots, y_m; A) 
				\begin{pmatrix}
					\boldsymbol{L}_{n-m-1}^T & \boldsymbol{0}_{(n-m)\times m} \\
					\boldsymbol{0}^T_{(n-m)\times m} & {\bf I}_m
				\end{pmatrix} \\
				& = \begin{pmatrix}
					\boldsymbol{L}_{n-m-1} & \boldsymbol{0}_{(n-m)\times m} \\
					\boldsymbol{0}^T_{(n-m)\times m } & {\bf I}_m
				\end{pmatrix} \begin{pmatrix}
					\boldsymbol{H}_{n-m}[\mu] & \boldsymbol{Z}_{n-m-1}(x_1,\dots,x_m) \\
					\boldsymbol{Z}^T_{n-m-1}(y_1,\dots,y_m) & A
				\end{pmatrix} \begin{pmatrix}
					\boldsymbol{L}_{n-m-1}^T & \boldsymbol{0}_{(n-m)\times m} \\
					\boldsymbol{0}^T_{(n-m)\times m} & {\bf I}_m
				\end{pmatrix} \\ & = \begin{pmatrix}
					\boldsymbol{L}_{n-m-1} \boldsymbol{H}_{n-m}[\mu]  \boldsymbol{L}^T_{n-m-1} & \boldsymbol{L}_{n-m-1} \boldsymbol{Z}_{n-m-1}(x) \\
					\boldsymbol{Z}^T_{n-m-1}(y) \boldsymbol{L}^T_{n-m-1} & A
				\end{pmatrix} = \begin{pmatrix}
					\boldsymbol{I}_{n-m}	 & \boldsymbol{P}_{n-m-1}(x) \\
					\boldsymbol{P}^T_{n-m-1}(y) & A
				\end{pmatrix}.
			\end{split}
		\end{equation}

		Taking the determinant of both sides of \eqref{312} yields \eqref{eq:thick_semifremed_CD}, using the fact that
		\begin{equation}
			\det  {\bf L}_{n-m-1} = \frac{1}{\sqrt{H_{n-m}[\mu]}}.
		\end{equation}

\section{Proofs of Propositions \ref{prop:NIBE_bordered_framed} and \ref{prop:NIBB_bordered_framed}}\label{proofs of props}

\subsection{Proof of Proposition \ref{prop:NIBE_bordered_framed}}

By the reflection principle, the transition probability density for a standard Brownian motion to go from $x$ to $y$ during the time interval $[0,t]$ with an absorbing wall at zero is
\begin{equation}\label{bm1}
p_{\abs}(t; y|x)=\frac{1}{\sqrt{2\pi t}} \left[e^{-\frac{1}{2t}(x-y)^2}-e^{-\frac{1}{2t}(x+y)^2}\right].
\end{equation}
We will also consider a Brownian motion with an absorbing wall at zero and one at height $M>0$.  By repeatedly applying the reflection principle, we find that the transition probability density for a Brownian motion to go from $x$ to $y$ during the time interval $[0,t]$ with absorbing walls at zero and $M$ is
\begin{equation}\label{bm3}
p_{\abs, \abs}^M(t; y|x)=\frac{1}{\sqrt{2\pi t}} \sum_{h=-\infty}^\infty  \left[e^{-\frac{1}{2t}(x-y+2Mh)^2}-e^{-\frac{1}{2t}(x+y+2Mh)^2}\right].
\end{equation}
We can rewrite the two-wall formula using the Poisson resummation formula.  This gives
\begin{equation}\label{bm5a}
\begin{aligned}
p_{\abs,\abs}^M(t; y|x)&=\frac{1}{2M} \sum_{h=-\infty}^\infty e^{-\frac{t\pi^2 h^2}{2M^2}} e^{\frac{\ii h\pi x}{M}}\left( e^{-\frac{\ii h\pi y}{M}}- e^{\frac{\ii h\pi y}{M}}\right) \\
&=\frac{2}{M} \sum_{h=1}^\infty e^{-\frac{t\pi^2 h^2}{2M^2}} \sin\left(\frac{h\pi x}{M}\right) \sin\left(\frac{h\pi y}{M}\right).
\end{aligned}
\end{equation}

Let
\begin{equation}\label{bm6}
{\vec x} = (x_1, \dots x_n) , \quad {\vec y} = (y_1, \dots y_n) , \quad 0<x_1< \cdots < x_n, \quad 0<y_1< \cdots < y_n.
\end{equation}
Then by the Karlin--McGregor formula \cite{Karlin-McGregor59} , the transition probability for $n$ non-intersecting BMs to go from ${\vec x}$ to ${\vec y}$ over the period of time from $0$ to $1$ with an absorbing wall at $0$ is
\begin{equation}\label{bm7}
q_{\abs}( {\vec y} | {\vec x}) = \det \bigg[p_{\abs}(1; y_k|x_j)\bigg]_{j,k=1}^n,
\end{equation}
and the analogous transition probability with absorbing walls at 0 and $M$ is
\begin{equation}\label{bm8}
q_{\abs, \abs}^M( {\vec y} | {\vec x}) = \det \bigg[p_{\abs,\abs}^M(1; y_k|x_j)\bigg]_{j,k=1}^n.
\end{equation}
Then the cumulative distribution function for $\mathcal{M}_n$ is
\begin{equation}\label{bm9}
\mathbb{P}_m^{\vec\al, \vec\be}(\mathcal{M}_n< M ) =\lim_{\substack{y_1, y_2, \dots y_{n-m}, x_{1}, x_{2}, \dots x_{n-m} \to 0 \\ x_n \to \al_1, \dots, x_{n-m+1} \to \al_m \\ y_n \to \be_1, \dots, y_{n-m+1} \to \be_m}} \frac{q_{{\abs}, {\abs}}^M( {\vec y} | {\vec x}) }{q_{\abs}( {\vec y} | {\vec x})},
\end{equation}
where the limits $y_1, y_2, \dots y_{n-m}, x_{1}, x_{2}, \dots x_{n-m} \to 0$ are taken sequentially to respect the order $x_1<x_2<\cdots <x_{n-m}$ and $y_1<y_2<\cdots <y_{n-m}$.

We have 
 \begin{equation}\label{sw7}
\begin{aligned}
q_{\abs}( {\vec y} | {\vec x})&=\left(\frac{1}{2\pi}\right)^{n/2} \det \bigg[ e^{-\frac{1}{2}(y_k-x_j)^2}-e^{-\frac{1}{2}(y_k+x_j)^2} \bigg]_{j,k=1}^n \\
&=\left(\frac{2}{\pi}\right)^{n/2} \prod_{j=1}^n e^{-\frac{1}{2}(x_j^2+y_j^2)} \det \bigg[\sinh(x_jy_k) \bigg]_{j,k=1}^n .\\
\end{aligned}
\end{equation}
Expanding the first $n-m$ rows as a Taylor series about $x_j=0$, and using row operations we easily find that as $x_1, x_2, \dots, x_{n-m}\to 0$ (respecting the order $x_1<x_2<\cdots <x_{n-m}$),
 \begin{equation}
\begin{aligned}
q_{\abs}( {\vec y} | {\vec x})&=\left(\frac{2}{\pi}\right)^{n/2}\prod_{j=1}^nx_jy_j  \prod_{j=1}^n e^{-\frac{1}{2}y_j^2} \prod_{j=n-m+1}^n e^{-\frac{1}{2}x_j^2} \prod_{k=1}^{n-m} \frac{x_k^{2(k-1)}}{(2k-1)!} \\& \quad 
\times\det\begin{pmatrix}1 & 1 & \cdots & 1 \\
y_1^2 & y_2^2 & \cdots & y_n^2 \\
\vdots & \vdots & \ddots & \vdots \\
y_1^{2(n-m-1)} & y_2^{2(n-m-1)} & \cdots & y_n^{2(n-m-1)} \\
\frac{\sinh(x_{n-m+1}y_1)}{x_{n-m+1}y_1} & \frac{\sinh(x_{n-m+1}y_2)}{x_{n-m+1}y_2} & \cdots & \frac{\sinh(x_{n-m+1}y_n)}{x_{n-m+1}y_n}\\
\vdots & \vdots & \ddots & \vdots \\
\frac{\sinh(x_ny_{1})}{x_ny_{1}} & \frac{\sinh(x_ny_{2})}{x_ny_{2}} & \cdots & \frac{\sinh(x_ny_{n})}{x_ny_{n}} 
\end{pmatrix}  \left(1+O(|{\vec {x}}_{n-m}|^2)\right) ,\\
\end{aligned}
\end{equation}
where ${\vec x}_{n-m} = (x_1, x_2, \dots, x_{n-m})$. We can now take $y_1, \dots, y_{n-m} \to 0$, and $y_n=\be_1$, $y_{n-1}=\be_2, \cdots, y_{n-m+1}=\be_m$ as well as $x_n=\al_1$, $x_{n-1}=\al_2, \cdots, x_{n-m+1}=\al_m$ to obtain
\begin{multline}\label{sw7a}
q_{\abs}( {\vec y} | {\vec x})= \left(\frac{2}{\pi}\right)^{n/2}\prod_{j=1}^{n-m}x_jy_j  \prod_{j=1}^m \al_j\be_j e^{-\frac{1}{2}(\al_j^2+\be_j^2)} \prod_{k=1}^{n-m} \frac{x_k^{2(k-1)}}{(2k-1)!} \\
\times\det\begin{pmatrix}1 & \cdots & 1 & 1 & \cdots & 1 \\
y_1^2  & \cdots & y_{n-m}^2 & \be_m^2 & \cdots & \be_1^2 \\
\vdots & \ddots & \vdots &\vdots & \ddots & \vdots  \\
y_1^{2(n-m-1)} & \cdots & y_{n-m}^{2(n-m-1)} & \be_m^{2(n-m-1)} & \cdots & \be_1^{2(n-m-1)} \\
0  & \cdots & 0 & \frac{\sinh(\al_m\be_m)}{\al_m\be_m}-\sum_{\ell=0}^{n-m-1} \frac{(\al_m\be_m)^{2\ell}}{(2\ell+1)!} & \dots & \frac{\sinh(\al_m\be_1)}{\al_m\be_1}-\sum_{\ell=0}^{n-m-1} \frac{(\al_m\be_1)^{2\ell}}{(2\ell+1)!} \\
\vdots  & \ddots &\vdots & \vdots & \ddots & \vdots \\
0  & \cdots & 0 & \frac{\sinh(\al_1\be_m)}{\al_1\be_m}-\sum_{\ell=0}^{n-m-1} \frac{(\al_1\be_m)^{2\ell}}{(2\ell+1)!}  & \dots & \frac{\sinh(\al_1\be_1)}{\al_1\be_1}-\sum_{\ell=0}^{n-m-1} \frac{(\al_1\be_1)^{2\ell}}{(2\ell+1)!} \\
\end{pmatrix} \\
\qquad \qquad \times  \left(1+O(|{\vec {x}_{n-m}}|^2)+O(|{\vec {y}_{n-m}}|^2)\right).
\end{multline}
The above determinant is block upper-triangular and is the product of the two diagonal blocks, one of which is a Vandermonde determinant:
\begin{equation}
\begin{aligned}
q_{\abs}( {\vec y} | {\vec x})&=\left(\frac{2}{\pi}\right)^{n/2}\prod_{j=1}^{n-m}x_jy_j  \prod_{j=1}^m \al_j\be_j e^{-\frac{1}{2}(\al_j^2+\be_j^2)} \prod_{k=1}^{n-m} \frac{x_k^{2(k-1)}}{(2k-1)!}\Delta({\vec y}_{n-m}^2) \\ 
&\quad \times \det\left[\frac{\sinh(\al_{m-j+1}\be_{m-k+1})}{\al_{m-j+1}\be_{m-k+1}} -\sum_{\ell=0}^{n-m-1} \frac{(\al_{m-j+1}\be_{m-k+1})^{2\ell}}{(2\ell+1)!}\right]_{j,k=1}^m\left(1+O(|{\vec x}_{n-m}|^2)+O(|{\vec y}_{n-m}|^2)\right),
\end{aligned}
\end{equation}
where 
\[
\Delta({\vec y}_{n-m}^2) = \prod_{1\le j< k \le n-m}(y_k^2-y_j^2).
\]

Now consider $q_{\abs,\abs}^M({\vec y} | {\vec x})$. It is given as
 \begin{equation}
q_{\abs,\abs}^M({\vec y} | {\vec x})=\left(\frac{2}{M}\right)^n \det \left[\sum_{h=1}^\infty e^{-\frac{\pi^2h^2}{2M^2}}\sin\left(\frac{h \pi x_j}{M}\right)\sin\left(\frac{h \pi y_k}{M}\right)\right]_{j,k=1}^n. \\
\end{equation}
Taking $x_1, x_2, \dots, x_{n-m}\to 0$, using Taylor expansion and row operations, we find
\begin{multline}
q_{\abs,\abs}^M({\vec y} | {\vec x})=(-1)^{(n-m-1)(n-m)/2}\left(\frac{2}{M}\right)^n \left(\frac{\pi}{M}\right)^{(n-m)^2} \prod_{k=1}^{n-m}\frac{x_k^{2(k-1)}}{(2k-1)!}\prod_{k=1}^n x_ky_k \\
 \times\det \begin{pmatrix}
\nu_0(y_1)  & \nu_0(y_2) & \cdots  & \nu_0(y_n) \\
\nu_1(y_1)  & \nu_1(y_2) & \cdots  & \nu_1(y_n) \\
\vdots & \vdots & \ddots & \cdots \\
\nu_{n-m-1}(y_1)  & \nu_{n-m-1}(y_2) & \cdots  & \nu_{n-m-1}(y_n) \\
s( x_{n-m+1},y_1) & s(x_{n-m+1},y_2) & \cdots & s( x_{n-m+1},y_n) \\
\vdots & \vdots & \ddots & \cdots \\
s(x_{n}, y_1) & s(x_{n},y_2) & \cdots & s( x_{n},y_n)
\end{pmatrix} 
\left(1+O(|{\vec x}_{n-m}|^2)\right),
\end{multline}
where $\nu_k(\al)$ is the $k$th moment of the measure $\nu(\al)$ defined in \eqref{def:nu_NIBE}, 
and the constants $s(\al, \be)$ are as defined in \eqref{def:s_NIBE}.
Now taking $y_1,y_2,\dots, y_{n-m}\to 0$, and $y_n=\be_1$, $y_{n-1}=\be_2, \cdots, y_{n-m+1}=\be_m$ as well as $x_n=\al_1$, $x_{n-1}=\al_2, \cdots, x_{n-m+1}=\al_m$   we find	
 \begin{multline}\label{eq:qh_abs_abs}
q_{\abs,\abs}^M({\vec y} | {\vec x})= \left(\frac{2}{M}\right)^n \left(\frac{\pi}{M}\right)^{2(n-m)^2} \prod_{k=1}^{n-m}\frac{x_k^{2(k-1)}y_k^{2(k-1)}}{((2k-1)!)^2}\prod_{k=1}^{n-m} x_ky_k\prod_{k=1}^m \al_k\be_k \\
\det \begin{pmatrix}  \mu_0 & \mu_1 & \cdots & \mu_{n-m-1} & \nu_0(\be_m) & \cdots & \nu_0(\be_1) \\
\mu_1 & \mu_2 & \cdots & \mu_{n-m} & \nu_1(\be_m) & \cdots & \nu_1(\be_1) \\
\vdots & \vdots & \ddots & \vdots & \vdots & \ddots & \cdots \\
\mu_{n-m-1} & \mu_{n-m} & \cdots & \mu_{2(n-m-1)} & \nu_{n-m-1}(\be_m) & \cdots & \nu_{n-m-1}(\be_1) \\
\nu_0(\al_m) & \nu_1(\al_m) & \cdots & \nu_{n-m-1}(\al_m) & s(\al_{m},\be_m) & \cdots & s(\al_m,\be_m) \\
\vdots & \vdots & \ddots & \vdots & \vdots & \ddots & \cdots \\
\nu_0(\al_1) & \nu_1(\al_1) & \cdots & \nu_{n-m-1}(\al_1) & s(\al_{1},\be_m) & \cdots & s(\al_{1},\be_1)
\end{pmatrix} \\
 \times\left(1+O(|{\vec x}_{n-m}|^2)+O(|{\vec y}_{n-m}|^2)\right) 
\end{multline}
where $\mu_k$ is the $k$th moment of the measure $\mu$ defined in \eqref{def:mu_NIBE}. 
 The ratio of \eqref{sw7a} and \eqref{eq:qh_abs_abs}, 
taking $y_1, \dots, y_{n-m}, x_{1}, x_{2}, \dots x_{n-m} \to 0$ then
gives \eqref{eq:framed_M_dist}. 

The formula for $P_m^{\vec 0, \vec\be}( \mathcal{M}_n < M ) $ in \eqref{eq:bordered_M_dist} then obtained from \eqref{eq:framed_M_dist} in a similar way by taking $\al_1, \dots,\al_m \to 0$. We omit the details of this limit.

\subsection{Proof of Proposition \ref{prop:NIBB_bordered_framed}}

We start with the distribution function for $\mathscr{W}_n$ when all starting and ending points are distinct.  This formula was given in \cite{Baik-Liu14}, and we present it here as a lemma.

\begin{lemma}[Equation (37) from \cite{Baik-Liu14}]
For $n$ non-intersecting Browninan bridges $X_1(t)<X_2(t)<\cdots<X_n(t)$ on the time interval $[0,1]$ with distinct starting points $X_j(0) = x_j$ with $x_1<x_2<\dots < x_n$ and distinct ending points $X_j(1) = y_j$ with $y_1 < \cdots < y_n$, the width $\mathscr{W}_n := \ds \sup_{0\le t \le 1}\left(X_n(t) - X_1(t)\right)$ is given by the formula
\begin{equation}\label{eq:Wn_BB}
\mathbb{P}(\mathscr{W}_n < M) = \frac{ \int_0^{1} \det[g(x_j-y_k, \tau)]_{j,k=1}^n\,d\tau}{\det[p(x_j-y_k)]_{j,k=1}^n},
\end{equation}
where 
\[
p(x) = \frac{1}{\sqrt{2\pi}} e^{-x^2/2}, \quad g(x,\tau) = \sum_{h\in \Z} p(x+hM) e^{ 2\pi \ii h\tau}.
\]
\end{lemma}
The expression \eqref{eq:Wn_BB} has a well defined limit as $x_j\to 0$ and $y_j\to 0$ for $j=m_1+1, \dots, n-m_2$. Consider first the denominator $\det[p(x_j-y_k)]_{j,k=1}^n$. Using the formula $p(x) = \frac{1}{2\pi} \int_\R e^{-y^2/2+\ii xy}\, \dd y$, we can use the Maclaurin expansion of $p(x_j-y_k)$ and row operations to take the limit
\begin{multline}
\lim_{x_{m_1+1},x_{m_1+2},\dots,x_{n-m_2}\to 0}\frac{\det[p(x_j-y_k)]_{j,k=1}^n}{\ds\prod_{j=1}^{n-m_1-m_2} \frac{(x_{m_1+j}y_{m_1+j})^{j-1}}{((j-1)!)^2}} = \\
 \frac{1}{{\ds\prod_{j=1}^{n-m_1-m_2} \frac{(y_j)^{j-1}}{(j-1)!}}}\det\begin{bmatrix} 
\begin{pmatrix}
p(x_j-y_k)
\end{pmatrix}_{\substack{j=1\dots,m_1 \\ k=1\dots,n}} \\
\begin{pmatrix}
\frac{1}{2\pi}\int_\R e^{-y^2/2-\ii y_k}(\ii y)^{j-m_1-1}\dd y
\end{pmatrix}_{\substack{j=m_1+1\dots,n-m_1-m_2 \\ k=1\dots,n}} \\
\begin{pmatrix}
p(x_j-y_k)
\end{pmatrix}_{\substack{j=n-m_1-m_2+1\dots,n \\ k=1\dots,n}} \\
\end{bmatrix},
\end{multline}
where the limit is taken sequentially, first $x_{m_1+1}\to 0$, then $x_{m_1+2}\to 0$, etc. Now taking the limit as $y_{m_1+1},\dots, y_{n-m_2}\to 0$ (again sequentially), using Maclaurin expansion and column operations we find
\begin{multline}\label{eq:px_limit}
\lim_{\substack{x_{m_1+1},x_{m_1+2},\dots,x_{n-m_2}\to 0 \\ y_{m_1+1},y_{m_1+2},\dots,y_{n-m_2}\to 0}}\frac{\det[p(x_j-y_k)]_{j,k=1}^n}{\ds\prod_{j=1}^{n-m_1-m_2} \frac{(x_{m_1+j}y_{m_1+j}j)^{j-1}}{((j-1)!)^2}} = \\
\tiny{\det
  \begin{bmatrix}
  \begin{pmatrix} p(x_j-y_k) \end{pmatrix}_{j,k=1}^{m_1} & \begin{pmatrix} \frac{1}{2\pi}\int_\R e^{-y^2/2+\ii x_k}(-\ii y)^{k-m_1-1}\dd y\end{pmatrix}_{\substack{j=1,\dots,m_1 \\ k=m_1+1\dots,n-m_2}} &  \begin{pmatrix} p(x_j-y_k) \end{pmatrix}_{\substack{j=1,\dots,m_1 \\ k=n-m_2+1,\dots, n}} \\
\begin{pmatrix} \frac{1}{2\pi}\int_\R e^{-y^2/2-\ii y_k}(\ii y)^{j-m_1-1}\dd y\end{pmatrix}_{\substack{j= m_1+1,\dots,n-m_2 \\ k=1,\dots, m_1}} & \begin{pmatrix}\frac{1}{2\pi}  \int_\R e^{-y^2/2}(\ii y)^{j-m_1-1}(-\ii y)^{k-m_1-1}\end{pmatrix}_{j,k=m_1+1}^{n-m} & \begin{pmatrix} \frac{1}{2\pi}\int_\R e^{-y^2/2-\ii y_k}(\ii y)^{j-m_1-1}\dd y\end{pmatrix}_{\substack{j= m_1+1,\dots,n-m_2 \\ k=n-m_2+1,\dots, n}} \\
 \begin{pmatrix} p(x_j-y_k) \end{pmatrix}_{\substack{j=n-m_2+1,\dots,n \\ k = 1,\dots, m_1}} & \begin{pmatrix} \frac{1}{2\pi}\int_\R e^{-y^2/2+\ii x_k}(-\ii y)^{k-m_1-1}\dd y\end{pmatrix}_{\substack{j=n-m_2+1,\dots,n \\ k = m_1+1,\dots, n-m_2}}&  \begin{pmatrix} p(x_j-y_k) \end{pmatrix}_{\substack{j=n-m_2+1,\dots,n \\ k = n-m_2+1,\dots, n}}
\end{bmatrix}} \\
=\tiny{\det
  \begin{bmatrix}
  \begin{pmatrix} p(x_j-y_k) \end{pmatrix}_{j,k=1}^{m_1} & \begin{pmatrix} \frac{1}{2\pi}\int_\R e^{-y^2/2+\ii x_k}y^{k-m_1-1}\dd y\end{pmatrix}_{\substack{j=1,\dots,m_1 \\ k=m_1+1\dots,n-m_2}} &  \begin{pmatrix} p(x_j-y_k) \end{pmatrix}_{\substack{j=1,\dots,m_1 \\ k=n-m+1,\dots, n}} \\
\begin{pmatrix} \frac{1}{2\pi}\int_\R e^{-y^2/2-\ii y_k}y^{j-m_1-1}\dd y\end{pmatrix}_{\substack{j= m_1+1,\dots,n-m \\ k=1,\dots, m_1}} & \begin{pmatrix}\frac{1}{2\pi}  \int_\R e^{-y^2/2} y^{j+k-2(m_1+1)} \end{pmatrix}_{j,k=m_1+1}^{n-m} & \begin{pmatrix} \frac{1}{2\pi}\int_\R e^{-y^2/2-\ii y_k}y^{j-m_1-1}\dd y\end{pmatrix}_{\substack{j= m_1+1,\dots,n-m \\ k=n-m_2+1,\dots, n}} \\
 \begin{pmatrix} p(x_j-y_k) \end{pmatrix}_{\substack{j=n-m_2+1,\dots,n \\ k = 1,\dots, m_1}} & \begin{pmatrix} \frac{1}{2\pi}\int_\R e^{-y^2/2+\ii x_k} y^{k-m_1-1}\dd y\end{pmatrix}_{\substack{j=n-m_2+1,\dots,n \\ k = m_1+1,\dots, n-m_2}}&  \begin{pmatrix} p(x_j-y_k) \end{pmatrix}_{\substack{j=n-m_2+1,\dots,n \\ k = n-m_2+1,\dots, n}} 
\end{bmatrix}}. \\
\end{multline}
The above determinant is a framed Hankel determinant. Factoring out $(2\pi)^{-n}$ and writing $x_j = \al_j$ and $y_j=\be_j$ for $j=1,\dots, m_1$, as well as $x_{n-m_2+j} = \al_{m_1+j}$ and $y_{n-m_2+j} = \be_{m_1+j}$, it becomes
\[
\left(\frac{1}{2\pi}\right)^nH_n^F[\mu;\mu_{\be_{1}}\mu_{\be_{2}},\dots, \mu_{\be_{m}}; \mu_{(-\al_{1})},\dots,\mu_{(-\al_{m})}; A]
\]
after simple row and column operations, where $\mu$ and $\mu_\be$ are defined in \eqref{def:mu_NIBB} and \eqref{def:mu_beta_NIBB},  the matrix $A$ is defined in \eqref{eq:const_matrix_NIBB}.

Now consider $\det[g(\al_j-\be_k, \theta)]_{j,k=1}^n$. It is helpful to use the Poisson resummation formula to write  $g(x,\theta)$ as 
\[
g(x,\theta) = \frac{1}{M} \sum_{h\in \Z} e^{-\frac{1}{2}(\frac{2\pi}{M} (h-\tau))^2+\frac{2\pi \ii x}{M}(h-\tau)}.
\]
Furthermore, writing $y = \frac{2\pi }{M}(h-\tau)$, the above formula becomes (recall the discrete set $ \mathcal{D}_{M,\tau}$ defined in \eqref{def:discrete_set})
\[
g(x,\theta) = \frac{1}{M} \sum_{y\in \mathcal{D}_{M,\tau}} e^{-y^2/2+\ii xy},
\]
which has a very similar form to $p(x) =  \frac{1}{2\pi} \int_\R e^{-y^2/2+\ii xy}\, \dd y$. Similar to \eqref{eq:px_limit} then, we find
\begin{multline}\label{eq:gx_limit}
\lim_{\substack{x_{m_1+1},x_{m_1+2},\dots,x_{n-m_2}\to 0 \\ y_{m_1+1},y_{m_1+2},\dots,y_{n-m_2}\to 0}}\frac{\det[p(x_j-y_k)]_{j,k=1}^n}{\ds\prod_{j=1}^{n-m_1-m_2} \frac{(x_{m_1+j}y_{m_1+j}j)^{j-1}}{((j-1)!)^2}} = \\
 = \left(\frac{1}{M}\right)^n H_n^F[\mu;\mu_{\be_{1}}^{M,\tau}\mu_{\be_{2}}^{M,\tau},\dots, \mu_{\be_{m}}^{M,\tau}; \mu_{(-\al_{1})}^{M,\tau},\dots,\mu_{(-\al_{m})}^{M,\tau}; A^{M,\tau}],
\end{multline}
after writing $x_j = \al_j$ and $y_j=\be_j$ for $j=1,\dots, m_1$, as well as $x_{n-m_2+j} = \al_{m_1+j}$ and $y_{n-m_2+j} = \be_{m_1+j}$, 
where $\mu^{M,\tau}$ and $\mu_\be^{M,\tau}$ are defined in \eqref{def:mu_NIBB} and \eqref{def:mu_beta_NIBB},  the matrix $A^{M,\tau}$ is defined in \eqref{eq:const_matrix_NIBB}.

Combining \eqref{eq:Wn_BB}, \eqref{eq:px_limit}, and \eqref{eq:gx_limit}, we obtain \eqref{eq:framed_W_dist}. The proof of \eqref{eq:bordered_W_dist} is similar and we omit the details. This completes the proof of Proposition \ref{prop:NIBB_bordered_framed}.

\section{An illustration of asymptotics: continuous time non-intersecting paths}\label{Sec-Asymptotics}
Each one of the applications discussed in sections \ref{sec:NIBE}, \ref{section Nonintersecting paths}, and \ref{sec 6-vertex} require its own careful treatment in a separate work. However, in this section we showcase the asymptotic analysis associated with one of the aforementioned applications, that is finding the transition probability of the non-intersecting paths modeled by \eqref{P_T}, with prescribed starting and ending points. To this end, for certain choices of starting and ending points, we use some ideas from \cite{BEGIL} and \cite{G23} to illustrate some calculations.

The asymptotic behavior of Toeplitz determinants with respect to smooth functions can be described by the Szeg\H{o}--Widom theorem \cite{WIDOMBlock,Sz,BS}, which is formulated as
\begin{equation}\label{Szego Theorem}
	D_{n}[\phi] \sim G[\phi]^{n} E[\phi], \qquad n\to\infty,
\end{equation}
where the terms $G[\phi]$ and $E[\phi]$ are defined by:
\begin{equation}\label{G and E}
	G[\phi] = \exp \left([\log \phi]_{0}\right) \qandq E[\phi] = \exp \left( \sum_{k \geq 1} k[\log \phi ]_{k}[\log \phi]_{-k} \right).
\end{equation}
This theorem holds true when the function $\phi$ is suitably smooth, does not vanish on the unit circle, and possesses a winding number of zero. We refer to \cite{DIK1} for a comprehensive survey of the Szegő theorem, including an intriguing account of its historical developments.

 Let us consider the Toeplitz determinant generated by the symbol \begin{equation}\label{P_T1}
	P_X(z;T)  = e^{\frac{T}{2}(z+1/z-2)},
\end{equation}
first introduced in \eqref{P_T}. This is a Szeg{\H o}-type symbol, and thus the Szeg{\H o} limit theorem provides us with the large-size asymptotics. Indeed,
\begin{equation}\label{SSLT1}
	D_n[e^{\frac{T}{2}(z+1/z-2)}] = e^{-nT}D_n[e^{\frac{T}{2}(z+1/z)}] = e^{-nT} G[\phi]^nE[\phi] \left(1+O(e^{-cn})\right), \qquad \mbox{for any} \quad c>0, 
\end{equation}
where
\begin{equation}
	G[\phi] = [\log \phi]_0 = \int_{\T} \log(\phi(\tau;t)) \frac{\dd \tau}{2 \pi \ii \tau},
\end{equation}
\begin{equation}\label{E}
	E[\phi] = \sum^{\infty}_{k=0} k [\log \phi]_k [\log \phi]_{-k},
\end{equation}
and
\begin{equation}\label{phi mother}
	\phi \equiv \phi(z;T) := e^{\frac{T}{2}(z+1/z)}.
\end{equation}

\begin{remark} \normalfont
	The error term in \eqref{SSLT1} is $O(e^{-cn})$ for \textit{any} $c>0$, because  $e^{-c}$ is the radius of the contour $\Ga_0$ shown in Figure \ref{S_contour} and due to the analytic properties of $\phi$ this radius can be made as small as desired.
\end{remark}

The \textit{Szeg{\H o} function} associated with $\phi$ (see \eqref{47} and \eqref{45} in the Appendix) can be explicitly calculated as
\begin{equation}\label{al* non-int application}
	\al(z) = \begin{cases}
		e^{\frac{Tz}{2}} & |z|<1, \\
		e^{-\frac{T}{2z}} & |z|>1.
	\end{cases}
\end{equation}
Noticing that $G[\phi] = \al(0)$, we see that 
\begin{equation}\label{G non-int application}
	G[\phi]=1   
\end{equation}
in \eqref{SSLT1}. Due to the simplicity of the symbol, it is also straightforward to arrive at
\begin{equation}\label{E non-int application}
	E[\phi] = e^{\frac{T^2}{4}},
\end{equation}
via \eqref{E}.
So rewriting \eqref{SSLT1} we get
\begin{equation}\label{SSLT2}
	D_n[e^{\frac{T}{2}(z+1/z-2)}] = e^{\frac{T^2}{4}} e^{-nT}  \left(1+O(e^{-cn})\right), \qquad \mbox{for any} \quad c>0.
\end{equation}
\subsection{Top two paths with non-consecutive ending points. Proof of Theorem \ref{thm asymp bordered}} Consider $n$ continuous non-intersecting paths with starting points $x_1,\cdots,x_n$ and ending points $y_1, \cdots, y_n$, where $x_j = j$, $j=1,\cdots,n$, $y_k=k$, for $k=1,\cdots,n-2$, and $
n-1< y_{n-1} < y_n,
$ are arbitrary. For this calculation we assume
\begin{itemize}
	\item $T$ is finite, and
	\item the gaps $y_n-n$ and $y_{n-1}-n$ are finite as well.
\end{itemize}
As described in \S \ref{Sec Discrete and continuous time random walks}, the transition probability of the non-intersecting paths modeled by \eqref{P_T}, with the above starting and ending points,  is given by the two-bordered Toeplitz determinant:
\begin{equation}\label{two bordered transition probability}
	e^{-nT}D^B_n[\phi;\boldsymbol{\psi}_2],
\end{equation}
where
\begin{equation}
	\psi_1(z;n) =  \phi(z) z^{-(y_{n-1}-n)}    \qandq \psi_2(z;n) = \phi(z) z^{-(y_{n}-n)}. 
\end{equation}
Let us denote $$D^B_n[\phi;\boldsymbol{\psi}_2] \equiv \mathscr{D},$$  
and consider the Dodgson Condensation identity
\begin{equation}\label{DODGSON3}
	\mathscr{D} \cdot \mathscr{D}\left\lbrace \begin{matrix} 0 & n-1 \\  n-2& n-1 \end{matrix} \right\rbrace = \mathscr{D}\left\lbrace \begin{matrix} 0  \\  n-2 \end{matrix} \right\rbrace \cdot \mathscr{D}\left\lbrace \begin{matrix} n-1  \\  n-1 \end{matrix} \right\rbrace - \mathscr{D}\left\lbrace \begin{matrix} 0  \\  n-1 \end{matrix} \right\rbrace \cdot \mathscr{D}\left\lbrace \begin{matrix} n-1  \\  n-2 \end{matrix} \right\rbrace,
\end{equation} 
which can be written as
\begin{equation}\label{Dodgson 2-borderedAAA}		D^B_n[\phi;\boldsymbol{\psi}_2(z;n)] = D^B_{n-1}[\phi;z^{-1}\psi_1(z;n)] \frac{D^B_{n-1}[z\phi;\psi_2(z;n)]}{D_{n-2}[z \phi]} - D^B_{n-1}[\phi;z^{-1}\psi_2(z;n)] \frac{D^B_{n-1}[z\phi;\psi_1(z;n)]}{D_{n-2}[z \phi]},
\end{equation}
or
\begin{equation}\label{Dodgson 2-borderedAA}	
	\begin{split}
		D^B_{n+2}[\phi;\boldsymbol{\psi}_2(z;n+2)] & = D^B_{n+1}[\phi;z^{-1}\psi_1(z;n+2)] \frac{D^B_{n+1}[z\phi;\psi_2(z;n+2)]}{D_{n}[z \phi]}  \\ & - D^B_{n+1}[\phi;z^{-1}\psi_2(z;n+2)] \frac{D^B_{n+1}[z\phi;\psi_1(z;n+2)]}{D_{n}[z \phi]}. 
	\end{split}
\end{equation}
As illustrated in \cite{G23}, the two bordered Toeplitz determinants \[ D^B_{n+1}[\phi;z^{-1}\psi_1] \qandq  D^B_{n+1}[\phi;z^{-1}\psi_2] \]
can be found from 

\begin{equation}
	D^B_{n+1}[\phi;z^{-\ell}\phi] = \frac{D_n[\phi]}{\ell!} \frac{\dd ^{\ell}}{\dd z^{\ell}}X_{12}(z;n) \Bigg|_{z=0},
\end{equation}
where $X_{12}(z;n)$ is the $12$-entry in the solution of the Riemann--Hilbert problem \hyperref[RH-X1]{\textbf{RH-X1}} - \hyperref[RH-X3]{\textbf{RH-X3}}.  Thus it can be shown that for a Szeg{\H o}-type symbol $\phi$ we have
\begin{equation}\label{cor 271}
	D^B_{n+1}[\phi;z^{-\ell}\phi] =  G[\phi]^{n} E[\phi] \left( \frac{\al^{(\ell)}(0)}{\ell!} + O(e^{-\mathfrak{c}n}) \right), \qasq n \to \infty,	\end{equation}
where $\al$ is given by \eqref{45}, and the constants $G[\phi]$ and $E[\phi]$ are given by \eqref{G and E} and  $\mathfrak{c}$ is some positive constant (see Corollary 2.7.1 of \cite{G23}) depending on the analytic properties of $\phi$.

In view of \eqref{al* non-int application}, \eqref{G non-int application}, \eqref{E non-int application}, and \eqref{cor 271} we can write
\begin{equation}
	\begin{split}
		D^B_{n+1}[ \phi;z^{-\ell} \phi]  =  G[\phi]^{n} E[\phi] \left( \frac{\al^{(\ell)}(0)}{\ell!} + O(e^{-\mathfrak{c}n}) \right) =   e^{\frac{T^2}{4}}  \left( \frac{1}{\ell !}\left(\frac{T}{2}\right)^{\ell} + O(e^{-\mathfrak{c}n}) \right).   		
	\end{split}
\end{equation}
Therefore
\begin{equation}\label{524}
	D^B_{n+1}[\phi;z^{-1}\psi_1(z;n+2)] = D^B_{n+1}[\phi; \phi(z) z^{-(y_{n+1}-n-1)}] = e^{\frac{T^2}{4}}   \left( \frac{\left(\frac{T}{2}\right)^{y_{n+1}-n-1}}{(y_{n+1}-n-1) !}  + O(e^{-\mathfrak{c}n}) \right),
\end{equation}
and similarly
\begin{equation}\label{525}
	D^B_{n+1}[\phi;z^{-1}\psi_2(z;n+2)] = D^B_{n+1}[\phi; \phi(z) z^{-(y_{n+2}-n-1)}] = e^{\frac{T^2}{4}}   \left( \frac{\left(\frac{T}{2}\right)^{y_{n+2}-n-1}}{(y_{n+2}-n-1) !}  + O(e^{-\mathfrak{c}n}) \right).
\end{equation}
We have just arrived at the following result (compare with Theorem \ref{thm asymp bordered}) when only the top end-point is not adjacent to the rest. 
\begin{theorem}\label{thm 7.2}

Consider $n+1$ independent continuous time random walks $X_1(t), \dots, X_{n+1}(t)$ which begin at locations $x_1<x_2<\cdots<x_{n+1}$ at time $t=0$. For a given $y_1<y_2<\dots<y_{n+1}$, let $\mathbb{P}_{n+1}(x_1,\dots, x_{n+1}; y_1,\dots,y_{n+1};T)$ be the probability that $X_1(t), \dots, X_{n+1}(t)$ satisfy $X_j(T)=y_j$ for $j=1,\dots,n+1$, and $X_1(t), \dots, X_{n+1}(t)$ do not intersect for $0\le t\le T$.  Suppose  that $x_j=j$ for $1\leq j\leq n+1$, $y_k=k$ for $1\leq k\leq n$, while   $y_{n+1}$ is arbitrary. If $T$ and $\be_1:=y_{n+1}-n$ are bounded as $n\to\infty$ , $\mathbb{P}_{n+1}(x_1,\dots, x_{n+1}; y_1,\dots,y_{n+1};T)$ satisfies as $n \to \infty$,
	\begin{equation}
\mathbb{P}_{n+1}(x_1,\dots, x_{n+1}; y_1,\dots,y_{n+1};T) = 			\frac{e^{-(n+1)T}\left(\frac{T}{2}\right)^{\be_{1}-1}e^{\frac{T^2}{4}}}{(\be_{1}-1)!} \left(1+O(e^{-cn})\right),
	\end{equation}
	where $c>0$ is any positive number.
\end{theorem}

Some of the terms on the right hand side of \eqref{Dodgson 2-borderedAA} involve the symbol $z\phi$, which is not a Szeg{\H o}-type (due to non-zero winding number) nor a non-degenerate Fisher--Hartwig symbol (see e.g. \cite{DIK}). This implies that one needs to analyze a separate Riemann--Hilbert problem. We write $Z(z;n)$ to refer to the solution of the $X$-RHP when $\phi$ is replaced by $z\phi$.  More precisely, $Z(z;n)$ satisfies

\begin{itemize}
	\item  \textbf{RH-Z1}\label{RH-Z1} \qquad $Z(\cdot;n):\C\setminus \T \to \C^{2\times2}$ is analytic,
	\item \textbf{RH-Z2}\label{RH-Z2} \qquad The limits of $Z(\ze;n)$ as $\ze$ tends to $z \in \T $ from the inside and outside of the unit circle exist,  which are denoted by $Z_{\pm}(z;n)$ respectively. Moreover, the functions $z \mapsto Z_{\pm}(z)$ are continuous on $\T$ and are related by
	
	\begin{equation}
		Z_+(z;n)=Z_-(z;n)\begin{pmatrix}
			1 & z^{-n+1}\phi(z) \\
			0 & 1
		\end{pmatrix}, \qquad  z \in \T,
	\end{equation}
	
	\item \textbf{RH-Z3}\label{RH-Z3} \qquad  As $z \to \infty$
	
	\begin{equation}
		Z(z;n)=\big( I + O(z^{-1}) \big) z^{n \sigma_3}.    
	\end{equation}
\end{itemize}
The usual Deift--Zhou nonlinear steepest descent method would fail for the Z-RHP (at the level of constructing the so-called \textit{global parametrix}, see the Appendix). However, for all $n$, it is in fact possible to find algebraic relationships between the solutions $X(z;n)$ and $Z(z;n)$. Indeed, in \cite{G23} it was shown that one has the following two relationships between these solutions.  

\begin{theorem}\label{thm Z-RHP intro}\cite{G23}
	The solution $Z(z;n)$ to the Riemann--Hilbert problem  \hyperref[RH-Z1]{\textbf{RH-Z1}} through \hyperref[RH-Z3]{\textbf{RH-Z3}} can be expressed in terms of the data extracted from the solution $X(z;n)$ of the Riemann--Hilbert problem \hyperref[RH-X1]{\textbf{RH-X1}} through \hyperref[RH-X3]{\textbf{RH-X3}} as 
	\begin{equation}\label{ZX1}
		Z(z;n) =  \left[   \begin{pmatrix}
			\di \frac{\overset{\infty}{X}_{1,12}(n)X_{21}(0;n)}{X_{11}(0;n)} & -\overset{\infty}{X}_{1,12}(n) \\[12pt]
			\di	-\frac{X_{21}(0;n)}{X_{11}(0;n)} & 1
		\end{pmatrix} z^{-1} + \begin{pmatrix}
			1 & 0 \\
			0 & 0
		\end{pmatrix} \right] X(z;n) \begin{pmatrix}
			1 & 0 \\
			0 & z
		\end{pmatrix},
	\end{equation}
	where $\overset{\infty}{X}_{1,12}(n)$ is the $12$-entry of the matrix $\overset{\infty}{X}_{1}(n)$ in the asymptotic expansion \eqref{X asymp}.
\end{theorem}

\begin{theorem}\label{thmXZ intro}\cite{G23}
	The solution $Z(z;n)$ to the Riemann--Hilbert problem  \hyperref[RH-Z1]{\textbf{RH-Z1}} through \hyperref[RH-Z3]{\textbf{RH-Z3}} can be expressed in terms of the data extracted from the solution $X(z;n)$ of the Riemann--Hilbert problem \hyperref[RH-X1]{\textbf{RH-X1}} through \hyperref[RH-X3]{\textbf{RH-X3}} as
	\begin{equation}\label{Z in terms of X1}
		Z(z;n) = \begin{pmatrix}
			z + \overset{\infty}{X}_{1,22}(n-1) - \di \frac{\overset{\infty}{X}_{2,12}(n-1)}{\overset{\infty}{X}_{1,12}(n-1)} & -\overset{\infty}{X}_{1,12}(n-1) \\ \di
			\frac{1}{\overset{\infty}{X}_{1,12}(n-1)} & 0
		\end{pmatrix} X(z;n-1),
	\end{equation} where $\overset{\infty}{X}_{1,jk}(n)$ and $\overset{\infty}{X}_{2,jk}(n)$ are the $jk$-entries of the matrices $\overset{\infty}{X}_{1}(n)$ and $\overset{\infty}{X}_{2}(n)$ in the asymptotic expansion \eqref{X asymp}.
\end{theorem}
\begin{remark} \normalfont
	The compatibility of \eqref{ZX1} and \eqref{Z in terms of X1} yields the well known three-term recurrence relations for the system of bi-orthogonal polynomials on the unit circle, see Lemma 2.16 of \cite{G23}.
\end{remark}

Now let us focus on the other two terms, i.e. $\frac{D^B_{n+1}[z\phi;\psi_2]}{D_{n}[z \phi]}$ and $\frac{D^B_{n+1}[z\phi;\psi_1]}{D_{n}[z \phi]}$, on the right hand side of \eqref{Dodgson 2-borderedAA}. From Lemma 2.10 of \cite{G23} , in particular, we have
\begin{equation}\label{526}
	\frac{D^B_{n+1}[z\phi;z^{1-\ell}\phi]}{D_n[z\phi]} 	 = \frac{1}{\ell!} \frac{\dd ^{\ell}}{\dd z^{\ell}}Z_{12}(z;n) \Bigg|_{z=0}.
\end{equation}
Using \eqref{ZX1} we have
\begin{equation}\label{Z12}
	Z_{12}(z;n) = \left(\mathscr{B}(n)+z\right) X_{12}(z;n) -\overset{\infty}{X}_{1,12}(n) X_{22}(z;n),
\end{equation}
where 
\begin{equation}\label{Bn}
	\mathscr{B}(n) :=  \frac{\overset{\infty}{X}_{1,12}(n)X_{21}(0;n)}{X_{11}(0;n)}.
\end{equation}
Using \eqref{X in terms of R exact}, as $n \to \infty$, we have
\begin{equation}\label{BBB}
	\mathscr{B}(n) = - \frac{C_n[\phi]}{C_{n-1}[\phi]} \times \left(1 + O(\rho^{-2n})\right),
\end{equation} 	where
\begin{equation}\label{Cnphi}
	C_n[\phi] := \frac{1}{2 \pi \ii} \int_{\Ga_0} \tau^n \phi^{-1}(\tau) \al^2(\tau) \dd \tau \equiv -R_{1,12}(0;n+1) = O(\rho^{-n}),
\end{equation}
$\al$ is given by \eqref{45}, $\rho^{-1}$ is the radius of the circle $\Ga_0$ shown in Figure \ref{S_contour}, and $R_{1,12}$ is given by \eqref{R1}. Let us compute the asymptotics of $\mathcal{B}(n)$ via \eqref{BBB} and \eqref{Cnphi}. For $\phi$ given by \eqref{P_T1}, and using \eqref{al* non-int application} we have
\begin{equation}
	\begin{split}
		C_n[\phi] & = \frac{1}{2 \pi \ii} \int_{\Ga_0} \tau^n e^{-\frac{T}{2}(\tau+1/\tau)} e^{T \tau} \dd \tau =  \frac{1}{2 \pi \ii} \int_{\Ga_0} \tau^n e^{\frac{T}{2}(\tau-1/\tau)} \dd \tau \\
		& = \frac{1}{2 \pi \ii} \sum_{k=0}^{\infty} \left(\frac{T}{2}\right)^k \frac{1}{k!} \int_{\Ga_0} \tau^n (\tau-\tau^{-1})^{k} \dd \tau = \frac{1}{2 \pi \ii} \sum_{k=0}^{\infty} \left(\frac{T}{2}\right)^k \frac{1}{k!} \int_{\Ga_0} \tau^n \sum^{k}_{\ell=0} \binom{k}{\ell} \tau^{k-\ell}(-\tau^{-1})^{\ell}  \dd \tau \\ & =  \sum_{k=0}^{\infty} \left(\frac{T}{2}\right)^k \frac{1}{k!} \sum^{k}_{\ell=0} (-1)^{\ell} \binom{k}{\ell} \int_{\Ga_0}   \tau^{n+k-2\ell+1}  \frac{\dd \tau }{2\pi \ii \tau}.
	\end{split}
\end{equation}
Only the pairs $(k,\ell)$ for which
\begin{equation}\label{equality}
	n+k-2\ell+1=0,
\end{equation}
contribute. Notice that if $k \leq n$, no such pair exists as the above equality implies that $\ell>k$ . Moreover, if $k>n$ and $k=n+2j$ for $j \in \N$, the above equality does not yield an integer value for $\ell$, so such values of $k$ also do not contribute. Hence
\begin{equation} \begin{split}
		C_n[\phi] & =  \left(\frac{T}{2}\right)^{n+1} \frac{1}{(n+1)!} \sum^{n+1}_{\ell=0} (-1)^{\ell} \binom{n+1}{\ell} \int_{\Ga_0}   \tau^{2n-2\ell+2}  \frac{\dd \tau }{2\pi \ii \tau} \\ & +  \sum_{k=n+3}^{\infty} \left(\frac{T}{2}\right)^k \frac{1}{k!} \sum^{k}_{\ell=0} (-1)^{\ell} \binom{k}{\ell} \int_{\Ga_0}   \tau^{n+k-2\ell+1}  \frac{\dd \tau }{2\pi \ii \tau}. 
	\end{split}
\end{equation}
Let us focus on the first term on the right hand side of the above equality. Obviously, only $\ell=n+1$ contributes and we thus have 
\begin{equation}\label{Cn phi 1}
	\begin{split}
		C_n[\phi] & =  \left(\frac{T}{2}\right)^{n+1} \frac{(-1)^{n+1} }{(n+1)!}  +  \sum_{k=n+3}^{\infty} \left(\frac{T}{2}\right)^k \frac{1}{k!} \sum^{k}_{\ell=0} (-1)^{\ell} \binom{k}{\ell} \int_{\Ga_0}   \tau^{n+k-2\ell+1}  \frac{\dd \tau }{2\pi \ii \tau} \\ &
		=  \left(\frac{T}{2}\right)^{n+1} \frac{(-1)^{n+1} }{(n+1)!} \left(1 + O\left(\frac{1}{n}\right)\right),
	\end{split}
\end{equation}
and therefore
\begin{equation}\label{Cn phi asym}
	C_n[\phi] = -\frac{T}{2\sqrt{2\pi}}  \frac{\left(-\frac{Te}{2n}\right)^n}{ (n+2)^{3/2}}\left(1 + O\left(\frac{1}{n}\right)\right),
\end{equation}
where we have used the Stirling's formula. Therefore from \eqref{Cn phi 1}
\begin{equation}
	\mathscr{B}(n) = - \frac{C_n[\phi]}{C_{n-1}[\phi]} \times \left(1 + O(\rho^{-2n})\right)  =  \frac{T}{2(n+1)} \left(1 + O\left(\frac{1}{n}\right)\right).
\end{equation}
Now, let us focus on $\overset{\infty}{X}_{1,12}(n)$ in \eqref{Z12}. From \hyperref[RH-X3]{\textbf{RH-X3}} we have
\begin{equation}
	\overset{\infty}{X}_{1}(n) = \lim_{z \to \infty} z\left( X(z;n)z^{-n\sigma_3} - I \right).
\end{equation}
From this,  and recalling \eqref{X in terms of R exact} for $z \in \Om_{\infty}$, \eqref{R asymp} and the fact that $\al(z) \to 1$ as $z \to \infty$, we find
\begin{equation}\label{X1,12 infty}
	\begin{split}
		\overset{\infty}{X}_{1,12}(n) & = \lim_{z \to \infty} z R_{1,12}(z) + O(\rho^{-3n})  =  \frac{1}{2 \pi \ii} \int_{\Ga_0} \tau^n \phi^{-1}(\tau) \al^2(\tau) \dd \tau \times \left(1 + O(\rho^{-2n})\right) \\ & = C_n[\phi] \times \left(1 + O(\rho^{-2n})\right) = O \left( \frac{\left(-\frac{Te}{2n}\right)^n}{ (n+2)^{3/2}} \right),
	\end{split}
\end{equation}
by \eqref{Cn phi asym}, where leading up to the last asymptotic equality we also have used \eqref{R_k's are small}, \eqref{R1} and the fact that $R_{2\ell}(z;n)$ is diagonal and $R_{2\ell+1}(z;n)$ is off-diagonal, $\ell \in \N \cup \{0\}$.  Now, in view of \eqref{Z12} let us focus on $X_{12}(z;n)$ and $X_{22}(z;n)$ in a neighborhood of $z=0$,  using \eqref{X in terms of R exact}, as $n \to \infty$. We have
\begin{align}
	X_{12}(z;n) & =  \al(z) \left( 1 + O \left( \frac{\rho^{-2n}}{1+|z|} \right) \right), \label{X12 Om0} \\
	X_{22}(z;n) & =  R_{1,21}(z;n) \al(z) \left( 1 + O \left( \frac{\rho^{-2n}}{1+|z|} \right) \right) = O \left( \frac{\rho^{-n}}{1+|z|} \right) , \label{X22 Om0}
\end{align}
where we have used \eqref{R_k's are small}, $\al$ and $R_1$ are respectively given by \eqref{45} and \eqref{R1}, and $\rho^{-1}$ is the radius of the circle $\Ga_0$ shown in Figure \ref{S_contour}. Combining these we conclude that 
\begin{equation}\label{Z12 a}
	Z_{12}(z;n) = \left(\mathscr{B}(n)+z\right)  \al(z) \left( 1 + O \left( \frac{\rho^{-2n}}{1+|z|} \right) \right) + O \left( \frac{\left(-\frac{Te}{2\rho n}\right)^n}{ (n+2)^{3/2} (1+|z|)} \right),
\end{equation}
where, in $\Om_0$ and as $n \to \infty$, the first term is $O(1)$ and the second term decays faster than $O(e^{-cn})$, for any $c>0$. Therefore, the asymptotically relevant term for computing the RHS of \eqref{526} is the first term on the RHS of \eqref{Z12 a}. Indeed we find
\begin{equation}
	\frac{\dd ^{\ell}}{\dd z^{\ell}}Z_{12}(z;n) \Bigg|_{z=0} = \mathscr{B}(n) \al^{(\ell)}(0) + \ell \al^{(\ell-1)}(0) + O(\rho^{-2n}),
\end{equation} 
therefore from \eqref{526} and \eqref{al* non-int application} we have
\begin{equation}
	\begin{split}
		\frac{D^B_{n+1}[z\phi;z^{1-\ell}\phi]}{D_n[z\phi]} & = \frac{1}{\ell !}\mathscr{B}(n) \al^{(\ell)}(0) + \frac{1}{(\ell-1)!}\al^{(\ell-1)}(0) + O(\rho^{-2n}) \\ & = \frac{1}{(\ell-1)!}\left(\frac{T}{2}\right)^{\ell-1} + \frac{1}{(\ell)!}\left(\frac{T}{2}\right)^{\ell+1}  \frac{1}{(n+1)} \left(1 + O\left(\frac{1}{n}\right)\right).
	\end{split}
\end{equation}
Therefore 
\begin{align}
	\frac{D^B_{n+1}[z\phi;\psi_1(z;n+2)]}{D_{n}[z \phi]}  & = \frac{1}{(y_{n+1}-n-2)!}\left(\frac{T}{2}\right)^{y_{n+1}-n-2} + \frac{1}{(y_{n+1}-n-1)!}\left(\frac{T}{2}\right)^{y_{n+1}-n}  \frac{1}{(n+1)} \left(1 + O\left(\frac{1}{n}\right)\right), \\ 		\frac{D^B_{n+1}[z\phi;\psi_2(z;n+2)]}{D_{n}[z \phi]}  & = \frac{1}{(y_{n+2}-n-2)!}\left(\frac{T}{2}\right)^{y_{n+2}-n-2} + \frac{1}{(y_{n+2}-n-1)!}\left(\frac{T}{2}\right)^{y_{n+2}-n}  \frac{1}{(n+1)} \left(1 + O\left(\frac{1}{n}\right)\right).
\end{align}
Plugging these, together with \eqref{524} and \eqref{525}, into \eqref{Dodgson 2-borderedAA} and subsequently in \eqref{two bordered transition probability} yields the result of Theorem \ref{thm asymp bordered}.

\subsection{Top path with non-consecutive starting and ending points.  Proof of Theorem \ref{thm asymp framed}}

Consider $n$ continuous non-intersecting paths with starting points $x_1,\cdots,x_n$ and ending points $y_1, \cdots, y_n$, where $x_j = j$, $j=1,\cdots,n-1$, $y_k=k$, for $k=1,\cdots,n-1$, while both $x_{n}$ and $y_{n}$ are arbitrary. In this section, like the previous one, we assume
\begin{itemize}
	\item $T$ is finite, and
	\item the gaps $y_{n}-n$ and $x_{n}-n$ are finite as well.
\end{itemize}
From \eqref{NIRM_phi_continuous} we see that the transition probability of the non-intersecting paths modeled by \eqref{P_T}, with the above starting and ending points, is given by the semi-framed Toeplitz determinant
	\begin{equation}\label{half-framed 3 intro}
	\mathscr{L}_n[\phi;\psi,\eta;a] := \det \begin{pmatrix}
		\phi_0& \phi_{-1} & \cdots & \phi_{-n+2} & \psi_{n-2}  \\
		\phi_{1}& \phi_0 & \cdots  & \phi_{-n+3} & \psi_{n-3} \\
		\vdots & \vdots & \ddots & \vdots & \vdots \\
		\phi_{n-2} & \phi_{n-3} &   \cdots  & \phi_{0} & \psi_{0} \\
		\eta_{0} & \eta_{1} & \cdots  & \eta_{n-2} & a
	\end{pmatrix}, 
\end{equation}
where $\phi$ is given by \eqref{phi mother},
\begin{equation}\label{NIRW_psi_eta-1}
	\psi(\ze;n) = \phi(\ze)\ze^{-(y_n - n+1)}, \quad \eta(\ze;n) =  \phi(\ze)\ze^{x_n - 1},
\end{equation}
and $a = \phi_{y_n - x_n}$. Above we have used the fact that the symbol $\phi$ is \textit{even}, that is $\phi(z^{-1})=\phi(z)$, which implies that $\phi_{j} = \phi_{-j}$.  Let us now recall the following results from Theorem 1.9 of \cite{G23} and its corollary  (see also Corollary \ref{eq:semi-framed_det_CDbb} of the current work).

\begin{theorem}\label{semis in terms of RepKer intro}\cite{G23}
	The semi-framed Toeplitz determinant $\mathscr{L}_n[\phi;\psi,\eta;a]$ can be represented in terms of the reproducing kernel of the system of bi-orthogonal polynomials on the unit circle associated with $\phi$ given by \eqref{BOPUC Rep Ker} and \eqref{bi-orthogonality intro} as
	\begin{align}
\frac{\mathscr{L}_{n+2}\left[\phi;  \psi, \eta ;a\right]}{D_{n+1}[\phi]} & = a - \int_{\T} \left[ \int_{\T} \mathscr{K}_n(z_1,z^{-1}_2)  \eta(z_2) \frac{\dd z_2}{2 \pi \ii z_2}   \right] z^{-n}_1 \psi(z_1) \frac{\dd z_1}{2 \pi \ii z_1},
	\end{align}
	where $D_{n}[\phi]$ is given by \eqref{ToeplitzDet}. Using the Christoffel-Darboux identity we have
	\begin{align}
			\frac{\mathscr{L}_{n+2}\left[\phi;\psi, \eta ;a\right]}{D_{n+1}[\phi]}  & = a  -  \int_{\T} \int_{\T} \frac{ \Tilde{\eta}(z_2) \tilde{\psi}(z_1)}{z_1-z_2} \det \begin{pmatrix}
		X_{11}(z_2;n+1) & X_{21}(z_2;n+2) \\
		X_{11}(z_1;n+1) & X_{21}(z_1;n+2)
	\end{pmatrix}  \frac{\dd z_2}{2 \pi \ii z_2} \frac{\dd z_1}{2 \pi \ii z_1}, 				\label{L and X-RHP}
	\end{align}
	where $D_{n}[\phi]$ is given by \eqref{ToeplitzDet}, and $X_{11}$ and $X_{21}$ are respectively the $11$ and $21$ entries of the solution to  \hyperref[RH-X1]{\textbf{RH-X1}}  to \hyperref[RH-X3]{\textbf{RH-X3}}.	
\end{theorem}

We need to analyze the semi-framed Toeplitz determinant $$e^{-(n+1)T}\mathscr{L}_{n+1}[\phi;\psi(z;n+1),\eta(z;n+1);\phi_{y_{n+1}-x_{n+1}}],$$
where $\psi$ and $\eta$ are given by \eqref{NIRW_psi_eta-1} and $\phi$ is given by \eqref{phi mother}. Theorem \ref{semis in terms of RepKer intro} gives a representation of semi-framed Toeplitz determinants in terms of the reproducing kernel of the biorthogonal polynomials on the unit circle. Indeed, 
\begin{equation}\label{550}
	\begin{split}
&	\mathscr{L}_{n+1}[\phi;\psi(z;n+1),\eta(z;n+1);\phi_{y_{n+1}-x_{n+1}}] 	\\ & =  D_n[\phi] \left\{\phi_{y_{n+1}-x_{n+1}} -  \int_{\T} \int_{\T} \frac{ \Tilde{\eta}(z_2;n+1) \tilde{\psi}(z_1;n+1)}{z_1-z_2} \det \begin{pmatrix}
	X_{11}(z_2;n) & X_{21}(z_2;n+1) \\
	X_{11}(z_1;n) & X_{21}(z_1;n+1)
\end{pmatrix}  \frac{\dd z_2}{2 \pi \ii z_2} \frac{\dd z_1}{2 \pi \ii z_1} \right\}, 
	\end{split}
\end{equation}
by \eqref{L and X-RHP}.  Recall that the integral representation for the modified Bessel function of integer order $j$ is

\begin{equation}
	I_j(T) = \frac{1}{\pi} \int_{0}^{\pi} e^{T \cos \theta} \cos(j \theta) \dd \theta = \int_{\T} z^{-j} e^{\frac{T}{2}(z+z^{-1})} \frac{\dd z}{2 \pi \ii z},
\end{equation} which is precisely the $j$-th Fourier coefficient of $\phi$ given by \eqref{phi mother},  so we have
\begin{equation}\label{FC Modified Bessel}
	\phi_{y_{n+1}-x_{n+1}} = I_{y_{n+1}-x_{n+1}}(T).
\end{equation}

Now we focus on computing the double integral in \eqref{550}.	From \eqref{X in terms of R exact} in $\Om_1$ (see Figure \ref{S_contour}) we have
\begin{align}
	X_{11}(z;n) & = z^n \al(z) \phi^{-1}(z) \left( 1+ O(\rho^{-n})\right), \qquad n \to \infty, \\
	X_{21}(z;n) & = - \al^{-1}(z)  \left( 1+ O(\rho^{-n})\right), \qquad n \to \infty.
\end{align}
Therefore, as $n \to \infty$, \eqref{550} can be written as
\begin{equation}
	\begin{split}
	& \frac{\mathscr{L}_{n+1}[\phi;\psi(z;n+1),\eta(z;n+1);\phi_{y_{n+1}-x_{n+1}}]}{D_n[\phi]}-  \phi_{y_{n+1}-x_{n+1}}  \\ & =  \int_{\T} \int_{\T} \frac{ \Tilde{\eta}(z_2;n+1) \tilde{\psi}(z_1;n+1)}{z_1-z_2} \left( - z^{n}_2 \al_+(z_2) \phi^{-1}(z_2) \al^{-1}_+(z_1)   \right)\frac{\dd z_2}{2 \pi \ii z_2} \frac{\dd z_1}{2 \pi \ii z_1} \\ & + \int_{\T} \int_{\T} \frac{ \Tilde{\eta}(z_2;n+1) \tilde{\psi}(z_1;n+1)}{z_1-z_2} \left(  z^{n}_1 \al_+(z_1) \phi^{-1}(z_1) \al^{-1}_+(z_2)  \right)\frac{\dd z_2}{2 \pi \ii z_2} \frac{\dd z_1}{2 \pi \ii z_1} + O(\rho^{-n}).
	\end{split}
\end{equation} Plugging in \eqref{NIRW_psi_eta-1} and using the scalar jump condition \eqref{47} we can rewrite this after straightforward simplifications as
\begin{equation}
	\begin{split}
& \frac{\mathscr{L}_{n+1}[\phi;\psi(z;n+1),\eta(z;n+1);\phi_{y_{n+1}-x_{n+1}}]}{D_n[\phi]}-  \phi_{y_{n+1}-x_{n+1}}  \\ &  =  \int_{\T} \int_{\T} \frac{1}{z_1-z_2}   \al_+(z_1)  \al^{-1}_-(z_2) z^{1-x_{n+1}}_2 z^{y_{n+1}}_1  \frac{\dd z_1}{2 \pi \ii z_1} \frac{\dd z_2}{2 \pi \ii z_2} \\ & - \int_{\T} \int_{\T} \frac{1}{z_1-z_2} \left( \frac{z_2}{z_1} \right)^n  \al_+(z_2)   \al^{-1}_-(z_1) z^{1-x_{n+1}}_2 z^{y_{n+1}}_1  \frac{\dd z_1}{2 \pi \ii z_1} \frac{\dd z_2}{2 \pi \ii z_2 } + O(\rho^{-n}).
	\end{split}
\end{equation}

For a fixed $\frak{r}>1$, define $	\T^{\frak{r}}_{\mp} := \left\{ z \ : \ |z|= \frak{r}^{\pm1} \right\}$. We now deform the contours of integration to rewrite the previous equation as

\begin{equation}\label{3.61 a}
	\begin{split}
& \frac{\mathscr{L}_{n+1}[\phi;\psi(z;n+1),\eta(z;n+1);\phi_{y_{n+1}-x_{n+1}}]}{D_n[\phi]}-  \phi_{y_{n+1}-x_{n+1}}  \\ &  =  \int_{\T^{\frak{r}}_-} \int_{\T^{\frak{r}}_+} \frac{1}{z_1-z_2}   \al(z_1)  \al^{-1}(z_2) z^{1-x_{n+1}}_2 z^{y_{n+1}}_1  \frac{\dd z_1}{2 \pi \ii z_1} \frac{\dd z_2}{2 \pi \ii z_2} \\ & - \int_{\T^{\frak{r}}_+} \int_{\T^{\frak{r}}_-} \frac{1}{z_1-z_2} \left( \frac{z_2}{z_1} \right)^n  \al(z_2)   \al^{-1}(z_1) z^{1-x_{n+1}}_2 z^{y_{n+1}}_1  \frac{\dd z_1}{2 \pi \ii z_1} \frac{\dd z_2}{2 \pi \ii z_2 } + O(\rho^{-n}) \\ &  = - \int_{\T^{\frak{r}}_-} \int_{\T^{\frak{r}}_+}    \al(z_1)  \al^{-1}(z_2) z^{-x_{n+1}}_2 z^{y_{n+1}}_1 \left( \sum_{k=0}^{\infty} \frac{z^k_1}{z^k_2} \right)  \frac{\dd z_1}{2 \pi \ii z_1} \frac{\dd z_2}{2 \pi \ii z_2} \\ & - \int_{\T^{\frak{r}}_+} \int_{\T^{\frak{r}}_-}  \left( \frac{z_2}{z_1} \right)^n  \al(z_2)   \al^{-1}(z_1) z^{1-x_{n+1}}_2 z^{y_{n+1}-1}_1 \left( \sum_{k=0}^{\infty} \frac{z^k_2}{z^k_1} \right)  \frac{\dd z_1}{2 \pi \ii z_1} \frac{\dd z_2}{2 \pi \ii z_2 } + O(\rho^{-n}) \\ = &    - \sum_{k=0}^{\infty} \left[  \int_{\T^{\frak{r}}_{+}}  \al(z) z^{n+1+k-x_{n+1}} \frac{\dd z}{2 \pi \ii z} \right] \left[  \int_{\T^{\frak{r}}_{-}}  \al^{-1}(z) z^{y_{n+1}-n-k-1} \frac{\dd z}{2 \pi \ii z} \right] \\ & - \sum_{k=0}^{\infty} \left[  \int_{\T^{\frak{r}}_{+}}  \al(z) z^{y_{n+1}+k} \frac{\dd z}{2 \pi \ii z} \right] \left[  \int_{\T^{\frak{r}}_{-}} \al^{-1}(z) z^{-x_{n+1}-k} \frac{\dd z}{2 \pi \ii z} \right] + O(\rho^{-n}).
	\end{split}
\end{equation}
Using \eqref{al* non-int application} we obtain
\begin{equation}\label{556}
	\begin{split}
		& \frac{\mathscr{L}_{n+1}[\phi;\psi(z;n+1),\eta(z;n+1);\phi_{y_{n+1}-x_{n+1}}]}{D_n[\phi]}-  \phi_{y_{n+1}-x_{n+1}}  \\ = &    - \sum_{k=0}^{\infty} \left[  \int_{\T^{\frak{r}}_{+}}  e^{\frac{Tz}{2}} z^{n+1+k-x_{n+1}} \frac{\dd z}{2 \pi \ii z} \right] \left[  \int_{\T^{\frak{r}}_{-}}  e^{\frac{T}{2z}} z^{y_{n+1}-n-k-1} \frac{\dd z}{2 \pi \ii z} \right] \\ & - \sum_{k=0}^{\infty} \left[  \int_{\T^{\frak{r}}_{+}}  e^{\frac{Tz}{2}} z^{y_{n+1}+k} \frac{\dd z}{2 \pi \ii z} \right] \left[  \int_{\T^{\frak{r}}_{-}} e^{\frac{T}{2z}} z^{-x_{n+1}-k} \frac{\dd z}{2 \pi \ii z} \right] + O(\rho^{-n}).
	\end{split}
\end{equation}
By the Cauchy Theorem we obtain
\begin{equation}\label{556}
	\begin{split}
		& \frac{\mathscr{L}_{n+1}[\phi;\psi(z;n+1),\eta(z;n+1);\phi_{y_{n+1}-x_{n+1}}]}{D_n[\phi]}-  \phi_{y_{n+1}-x_{n+1}} \\ &  =    - \sum_{k=0}^{\infty} \left[  \int_{\T^{\frak{r}}_{+}}  e^{\frac{Tz}{2}} z^{n+1+k-x_{n+1}} \frac{\dd z}{2 \pi \ii z} \right] \left[  \int_{\T^{\frak{r}}_{-}}  e^{\frac{T}{2z}} z^{y_{n+1}-n-k-1} \frac{\dd z}{2 \pi \ii z} \right]  + O(\rho^{-n}).
	\end{split}
\end{equation}
Residue calculation combined with \eqref{FC Modified Bessel} and the Szeg{\H o} Theorem for $\phi$ (recall \eqref{G non-int application} and \eqref{E non-int application})  yields \eqref{top-path-different-asymp}.

\appendix
\section{Solution of the Riemann--Hilbert problem for BOPUC with Szeg{\H o}-type symbols}\label{Appendices}

Let us recall the Riemann--Hilbert problem for BOPUC is due to J.Baik, P.Deift and K.Johansson \cite{BDJ}.

\begin{itemize}
	\item  \textbf{RH-X1} \qquad $X(\cdot;n):\C\setminus \T \to \C^{2\times2}$ is analytic,
	\item \textbf{RH-X2} \qquad  The limits of $X(\ze;n)$ as $\ze$ tends to $z \in \T $ from the inside and outside of the unit circle exist, and are denoted $X_{\pm}(z;n)$ respectively and are related by
	
	\begin{equation}
		X_+(z;n)=X_-(z;n)\begin{pmatrix}
			1 & z^{-n}\phi(z) \\
			0 & 1
		\end{pmatrix}, \qquad  z \in \T,
	\end{equation}
	
	\item \textbf{RH-X3} \qquad  As $z \to \infty$
	
	\begin{equation}
		X(z;n)=\big( I + O(z^{-1}) \big) z^{n \sigma_3}.
	\end{equation}
\end{itemize}
For the purposes of section \ref{Sec-Asymptotics} in this Appendix we show the standard steepest descent analysis to asymptotically solve this problem for a Szeg{\H o}-type symbol which was first shown in \cite{D}.  We first normalize the behavior at $\infty$ by defining \begin{equation}\label{40}
	T(z;n) := \begin{cases}
		X(z;n)z^{-n\sigma_3}, & |z|>1, \\
		X(z;n), & |z|<1.
	\end{cases}
\end{equation}
The function $T$ defined above satisfies the following RH problem

\begin{itemize}
	\item  \textbf{RH-T1} \qquad $T(\cdot;n) :\C\setminus \T \to \C^{2\times2}$ is analytic,
	\item  \textbf{RH-T2} \qquad    $T_{+}(z;n)=T_{-}(z;n)\begin{pmatrix}
		z^n & \phi(z) \\
		0 & z^{-n}
	\end{pmatrix}, \qquad  z \in \T$,
	\item  \textbf{RH-T3} \qquad   $T(z;n)=I+O(1/z), \qquad z \to \infty.$
\end{itemize}
So $T$ has a highly-oscillatory jump matrix as $n \to \infty$. The next transformation yields a Riemann Hilbert problem, normalized at infinity, having an exponentially decaying jump matrix on the \textit{lenses}. Note that we have the following factorization of the jump matrix of the $T$-RHP : \begin{equation}\label{Appendix Pure Toeplitz G_T factorization}
	\begin{pmatrix}
		z^n & \phi(z) \\
		0 & z^{-n}
	\end{pmatrix} = \begin{pmatrix}
		1 & 0 \\
		z^{-n}\phi(z)^{-1} & 1
	\end{pmatrix}\begin{pmatrix}
		0 & \phi(z) \\
		-\phi(z)^{-1} & 0
	\end{pmatrix}\begin{pmatrix}
		1 & 0 \\
		z^{n}\phi(z)^{-1} & 1
	\end{pmatrix}  \equiv J_{0}(z;n)J^{(\infty)}(z)J_{1}(z;n).
\end{equation}
Now, we define the following function : \begin{equation}\label{42}
	S(z;n):=\begin{cases}
		T(z;n)J^{-1}_{1}(z;n), & z \in \Om_1, \\
		T(z;n)J_{0}(z;n), & z \in \Om_2, \\
		T(z;n), & z\in \Om_0\cup \Om_{\infty}.
	\end{cases}
\end{equation}
Also introduce the following function on $\Ga_S := \Ga_0 \cup \Ga_1 \cup \T$:  \begin{equation}\label{43}
	J_{ S}(z;n)=\begin{cases}
		J_{1}(z;n), & z \in \Ga_0, \\
		J^{(\infty)}(z), & z \in \T, \\
		J_{0}(z;n), & z \in \Ga_1. \\
	\end{cases}
\end{equation} 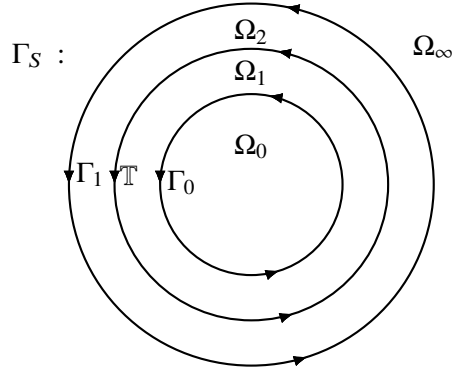
\begin{figure}
	\centering
	
	\begin{tikzpicture}[scale=0.4]

		\draw[ ->-=0.22,->-=0.5,->-=0.8,thick] (-3,0) circle (4.5cm);
		
		\draw[ ->-=0.22,->-=0.5,->-=0.8,thick] (-3,0) circle (3cm);
		
		\draw[ ->-=0.22,->-=0.5,->-=0.8,thick] (-3,0) circle (6cm);
		
		
		
		\node at (-3,4.45) [below] {$\Om_1$};
		
		\node at (-3,5.8) [below] {$\Om_2$};
		
		\node at (-8.3,-0.4) [above] {$\Ga_1$};
		
		
		
		\node at (-6.4,0.3) [left] {$\T$};
		
		\node at (-5.3,0.8) [below] {$\Ga_0$};

		\node at (-3,2) [below] {$\Om_0$};
		
		
		
		\node at (3,5.3) [below] {$\Om_{\infty}$};
		
		\node at (-10,5) [below] {$\Ga_S \ :$};
	\end{tikzpicture}

	\caption{Opening of lenses: the jump contour for the $S$-RHP.}
	\label{S_contour}
\end{figure}
Therefore we have the following Riemann--Hilbert problem for $S(z;n)$

\begin{itemize}
	\item \textbf{RH-S1} \qquad  $ S(\cdot;n):\C\setminus \Ga_S \to \C^{2\times2}$ is analytic.
	\item \textbf{RH-S2} \qquad  $S_{+}(z;n)=S_{-}(z;n) J_{S}(z;n), \qquad  z \in \Ga_S.$
	\item \textbf{RH-S3} \qquad  $ S(z;n)=I+O(1/z), \qquad \text{as} \ z \to \infty$.
\end{itemize}
Note that the matrices $J_0(z;n)$ and $J_1(z;n)$ tend to the identity matrix uniformly on their respective contours, exponentially fast as $n \to \infty$.	We are looking for a piecewise analytic function $P^{(\infty)}(z): \C \setminus \T : \to \C^{2\times2}$ such that \begin{itemize}
	\item \textbf{RH-Global1} \qquad    $P^{(\infty)}$ is holomorphic in $\C \setminus \T$.
	
	\item \textbf{RH-Global2} \qquad for $z\in \T$ we have \begin{equation}\label{44}
		P_+^{(\infty)}(z)=P_-^{(\infty)}(z)  \begin{pmatrix}
			0 & \phi(z) \\
			- \phi^{-1}(z) & 0
		\end{pmatrix}.
	\end{equation}
	
	\item \textbf{RH-Global3} \qquad   $P^{(\infty)}(z)=I+O(1/z),  \qquad  \text{as} \ z \to \infty$.
	
\end{itemize}
We can find a piecewise analytic function $\al$ which solves the following scalar multiplicative Riemann--Hilbert problem \begin{equation}\label{47}
	\al_+(z)=\al_-(z)\phi(z) \qquad  z \in \T.
\end{equation}
By Plemelj--Sokhotski formula we have \begin{equation}\label{45}
	\al(z)=\exp \left[ \frac{1}{2 \pi \ii } \int_{\T} \frac{\ln(\phi(\tau))}{\tau-z}d\tau \right],
\end{equation}
Now, using (\ref{47}) we have the following factorization \begin{equation}\label{46}
	\begin{pmatrix}
		0 & \phi(z) \\
		- \phi^{-1}(z) & 0
	\end{pmatrix}= \begin{pmatrix}
		\al_-^{-1}(z) & 0 \\
		0 & \al_-(z)
	\end{pmatrix} \begin{pmatrix}
		0 & 1 \\
		- 1 & 0
	\end{pmatrix} \begin{pmatrix}
		\al^{-1}_+(z) & 0 \\
		0 & \al_+(z)
	\end{pmatrix}.
\end{equation}
So, the function \begin{equation}\label{48}
	P^{(\infty)}(z) :=\begin{cases}
		\begin{pmatrix}
			0 & \al(z) \\
			-\al^{-1}(z) & 0
		\end{pmatrix},  & |z|<1, \\
		\begin{pmatrix}
			\al(z) & 0 \\
			0 & \al^{-1}(z)
		\end{pmatrix}, & |z|>1,
	\end{cases}
\end{equation}
satisfies (\ref{44}). Also, by the properties of the Cauchy integral, $P^{(\infty)}(z)$ is holomorphic in $\C \setminus \T$. Moreover, $\al(z)=1+O(z^{-1})$, as $z\to \infty$ and hence \begin{equation}\label{49}
	P^{(\infty)}(z)=I+O(1/z), \qquad  z \to \infty.
\end{equation}
Therefore $P^{(\infty)}$ given by (\ref{48}) is the unique solution of the global parametrix Riemann--Hilbert problem. Let us now consider the ratio

\begin{equation}
	R(z;n):= S(z;n) \left[ P^{(\infty)}(z) \right]^{-1}.
\end{equation}
We have the following Riemann--Hilbert problem for $R(z;n)$

\begin{itemize}
	\item \textbf{RH-R1} \qquad $ R$ is holomorphic in $\C\setminus (\Ga_0 \cup \Ga_1)$.
	\item \textbf{RH-R2} \qquad    $R_{+}(z;n)=R_{-}(z;n) J_{R}(z;n), \qquad  z \in \Ga_0 \cup \Ga_1 =: \Sigma_R$,
	\item \textbf{RH-R3} \qquad $ R(z;n)=I+O(1/z) \qquad \text{as} \ z \to \infty$.
\end{itemize}
This Riemann Hilbert problem is solvable for large $n$ (see e.g. \cite{DKMVZ})  and $R(z;n)$ can be written as
\begin{equation}\label{Appendix pure toep R series}
	R(z;n) = I + R_1(z;n) + R_2(z;n) + R_3(z;n) + \cdots, \ \ \ \  \ \ n \geq n_0
\end{equation}
where $R_k$ can be found recursively. Indeed 
\begin{equation}\label{Recursive R_k Def}
	R_k(z;n) = \frac{1}{2\pi \ii}\int_{\Sigma_R} \frac{\left[ R_{k-1}(\mu;n)\right]_-  \left( J_R(\mu;n)-I\right) }{\mu-z}\dd \mu,  \qquad  z \in \C \setminus \Sigma_R, \qquad k\geq1.
\end{equation}	
It is easy to check that $R_{2\ell}(z;n)$ is diagonal and $R_{2\ell+1}(z;n)$ is off-diagonal; $\ell \in \N \cup \{0\}$, and that 
\begin{equation}\label{R_k's are small}
	R_{k,ij}(z;n) = O \left( \frac{\rho^{-kn}}{1+|z|} \right), \qquad n \to \infty, \qquad k \geq 1,
\end{equation}
uniformly in $z\in \C \setminus \Sigma_R$, where $\rho$ (resp. $\rho^{-1}$) is the radius of $\Ga_1$(resp. $\Ga_0$). Let us compute $R_1(z;n)$; we have
\begin{equation}\label{JR-I}
	J_R(z)-I = \begin{cases}
		P^{(\infty)}(z) \begin{pmatrix}
			0 & 0 \\
			z^n \phi^{-1}(z) & 0 \end{pmatrix} \left[ P^{(\infty)}(z) \right]^{-1}, & z \in \Ga_0, \\
		P^{(\infty)}(z) \begin{pmatrix}
			0 & 0 \\
			z^{-n} \phi^{-1}(z) & 0 \end{pmatrix} \left[ P^{(\infty)}(z) \right]^{-1}, & z \in \Ga_1,
	\end{cases} =  \begin{cases}
		\begin{pmatrix}
			0 & -z^{n} \phi^{-1}(z)\al^{2}(z) \\
			0 & 0 \end{pmatrix}, & z \in \Ga_0, \\
		\begin{pmatrix}
			0 & 0 \\
			z^{-n} \phi^{-1}(z)\al^{-2}(z) & 0 \end{pmatrix},  & z \in \Ga_1.
	\end{cases}
\end{equation}
Therefore

\begin{equation}\label{R1}
	R_1(z;n)=  \begin{pmatrix}
		0 & -\di \frac{1}{2\pi \ii }\int_{\Ga_0} \frac{\tau^{n} \phi^{-1}(\tau)\al^{2}(\tau)}{\tau-z}d\tau \\ \di \frac{1}{2\pi \ii}\int_{\Ga_1} \frac{\tau^{-n} \phi^{-1}(\tau)\al^{-2}(\tau)}{\tau-z}d\tau
		& 0 \end{pmatrix}.
\end{equation}
If we trace back the Riemann--Hilbert problems $R \mapsto S \mapsto T \mapsto X $ we will obtain

\begin{equation}\label{X in terms of R exact}
	X(z;n) = R(z;n)\begin{cases} \begin{pmatrix} \al(z) & 0 \\ 0 & \al^{-1}(z) \end{pmatrix}z^{n\sigma_3}, & z \in \Om_{\infty}, \\
		\begin{pmatrix} \al(z) & 0 \\ -z^{-n}\al^{-1}(z) \phi^{-1}(z) & \al^{-1}(z) \end{pmatrix}z^{n\sigma_3}, & z \in \Om_{2}, \\
		\begin{pmatrix} z^{n}\al(z) \phi^{-1}(z) & \al(z) \\ -\al^{-1}(z) & 0 \end{pmatrix}, & z \in \Om_{1}, \\  \begin{pmatrix} 0 & \al(z) \\ -\al^{-1}(z) & 0 \end{pmatrix}, & z \in \Om_{0}, \\  
	\end{cases}
\end{equation}
where for $z \in \C \setminus \Sigma_{R}$, as $n \to\infty$, we have
\begin{equation}\label{R asymp}
	R(z;n)=\di \begin{pmatrix}
		1 + O \left(\frac{ \rho^{-2n}}{ 1+|z| }\right) & R_{1,12}(z;n)+ O \left(\frac{ \rho^{-3n}}{ 1+|z| }\right) \\
		R_{1,21}(z;n)+ O \left(\frac{ \rho^{-3n}}{ 1+|z| }\right) & 1 +O \left(\frac{ \rho^{-2n}}{ 1+|z| }\right)
	\end{pmatrix}.    
\end{equation}

	\section*{Acknowledgements}
This work was initiated during the special session on {\it Orthogonal Polynomials and their Applications} at the Joint Mathematics Meetings (JMM) in Boston in January, 2023. The authors are grateful to JMM and to the organizers of that session. We also thank Robert Buckingham for help with the simulations.

\bibliographystyle{plain} 
\bibliography{refs8} 

\end{document}